\documentclass[reqno]{amsart}

\usepackage{latexsym, amsfonts, amsthm, amsmath, amssymb, amscd, epsfig, amsrefs}
\usepackage{color}
\usepackage{mathrsfs}
\allowdisplaybreaks[1]

\usepackage{tikz}

\newtheorem{cor}{Corollary}
\newtheorem{lemma}{Lemma}[section]
\newtheorem{Add in proof}{Add in proof}[section]
\newtheorem{prop}{Proposition}[section]

\newtheorem{remark}{Remark}
\newtheorem{theorem}{Theorem}

\newcommand{\R}{\ensuremath{\mathbb{R}}}
\newcommand{\na}{\ensuremath{\D}}
\def \p{\partial}

\def\D{\nabla}
\def\MH{\mathcal{H}}

\def\<{\left\langle}
\def\>{\right\rangle}
\def\({\left(}
\def\){\right)}

\def\limsup{\operatornamewithlimits{lim\,sup}}

\let\div\relax \DeclareMathOperator{\div}{div}

\DeclareMathOperator{\dist}{dist}
\DeclareMathOperator{\tr}{tr}
\newcommand\alabel[1]{\addtocounter{equation}{1}\tag{\theequation}\label{#1}}

\usepackage{xcolor}
\definecolor{darkred}{rgb}{0.7,0.0,0.0}

\begin{document}
\title{Existence and convergence of the Beris-Edwards system with general Landau-de Gennes energy}

\numberwithin{equation}{section}
\author{Zhewen Feng, Min-Chun Hong and Yu Mei}

\address{Zhewen Feng, Department of Mathematics, The University of Queensland\\
	Brisbane, QLD 4072, Australia}
\email{z.feng@uq.edu.au}

\address{Min-Chun Hong, Department of Mathematics, The University of Queensland\\
Brisbane, QLD 4072, Australia}
\email{hong@maths.uq.edu.au}

\address{Yu Mei, School of Mathematics and Statistics,  Northwestern Polytechnical University, Xi'an Shaanxi,710129, P.R.China.}
\email{yu.mei@nwpu.edu.cn}

	\begin{abstract}
In this paper, we investigate the Beris-Edwards system  for both biaxial and uniaxial $Q$-tensors with a general Landau-de Gennes energy
density depending on four non-zero elastic constants.
We prove  existence of the strong solution of  the Beris-Edwards system  for uniaxial $Q$-tensors up to a maximal time. Furthermore, we
prove that   the strong solutions of the Beris-Edwards system for  biaxial $Q$-tensors  converge smoothly to   the solution of the
Beris-Edwards system  for uniaxial $Q$-tensors up to its maximal existence time.
	\end{abstract}

	\makeatletter
	\@namedef{subjclassname@2020}{%
	\textup{2020} Mathematics Subject Classification}
	\makeatother

\maketitle

\pagestyle{myheadings} \markright {Convergence for the Beris-Edwards system}
\section{Introduction}
The classical  Ericksen-Leslie
theory (\cite {Er}, \cite{Le}) successfully describes the dynamic
flow of uniaxial nematic liquid crystals. In \cite{BE94}, Beris-Edwards  pointed out that the  Ericksen-Leslie flow theory  has a limited domain of applications to liquid crystals. Therefore,    based on the celebrated Landau-de Gennes $Q$-tensor theory, Beris-Edwards \cite{BE94}  proposed a general hydrodynamic  theory to describe   flows of liquid crystals in modeling both uniaxial and biaxial nematic liquid crystals.

In 1971, de Gennes \cite{De}     introduced  a $Q$-tensor order parameter to establish   the  Landau-de Gennes theory, which    has been one of the successful continuum theories in modeling both uniaxial and biaxial nematic liquid crystals (cf. \cite{DP}, \cite{Ba}). Mathematically, the Landau-de Gennes theory is described by
the space of   symmetric and traceless $3\times3$ matrices
\begin{equation*}
	S_0:=\left\{Q\in \mathbb M^{3\times 3}:\quad Q^T=Q, \, \mbox{tr }Q  =0\right\},
\end{equation*}
where  $\mathbb M^{3\times 3}$ denotes the space of $3\times 3$ matrices.

Let $U$ be a domain in $\R^3$.  For a tensor $Q\in W^{1,2}(U; S_0)$,
the original Landau-de Gennes energy is defined by
\begin{equation}\label{LG energy}
	E_{LG}(Q; U):=\int_{U}f_{LG}\,dx=\int_{U}(  \tilde f_E +  \tilde f_B)\,dx.
\end{equation}
Here $\tilde  f_E$ is the elastic energy density  with elastic constants $L_1,...,L_4$ of the form
\begin{equation}\label{LG}
	\tilde f_E(Q,\D Q):=\frac{ L_1}{2}|\D Q|^2+\frac{ L_2}{2} \frac{\p Q_{ij}}{\p x_j} \frac{\p Q_{ik}}{\p x_k} +\frac{ L_3}{2}\frac{\p
		Q_{ik}}{\p x_j}\frac{\p Q_{ij}}{\p x_k}+\frac{ L_4}{2}Q_{lk}\frac{\p Q_{ij}}{\p x_l}\frac{\p Q_{ij}}{\p x_k}
\end{equation}
in which and in the sequel, we adopt the Einstein summation convention for repeated indices and $\tilde  f_B(Q)$ is  a bulk  energy density defined by
\begin{equation}\label{BE}
	\tilde f_B(Q):=-\frac{ a}{2}\tr( Q^2)-\frac{ b}{3}\tr( Q^3)+\frac{ c}{4}\left[\tr( Q^2)\right]^2
\end{equation}
with three positive material constants $a,b,c$.

In \cite{De}, de Gennes discovered the first  two terms of the elastic energy density in (\ref{LG}) with $L_3=L_4=0$.
Later,  combining the work of  Schiele-Trimper  \cite {ST} with the effect of Berreman-Meiboom \cite{BM84},  Dickmann  \cite  {Di}  completed  the  full density \eqref{LG} with two additional terms (cf. \cite{MGKB}, \cite{Ba}). However, for the case of $L_4\neq 0$,  Ball-Majumdar  \cite {BM} found an  example that  the  Landau-de Gennes energy density  \eqref{LG} does not  satisfy the coercivity condition. In fact, Golovaty et al. \cite{GNS20} said that ``From the standpoint of energy minimization, unfortunately, such a version of Landau-de Gennes becomes problematic, since the inclusion of the cubic term leads to an energy which is unbounded from below''. Therefore, there is a problem between mathematical and physical theory on nematic liquid crystals in the case of $L_4\neq 0$. In their book \cite{DP}, de Gennes and Prost said that {\it``the bending constant is much larger than others''}; i.e. $k_3>\max\{{k_1, k_2\}}$ at different temperatures. For example, for p-azoxyanisole (PAA) at $134^\circ C$, $k_1=4.05$, $k_2=2.1$, $k_3=5.77$, $k_4=3.08,$ where the unit is $10^{-12}\,N$. By the physical experiments on liquid crystals, the elastic constant $ L_4=\frac{1}{2s_+^3}(k_3-k_1)$  is not  zero in general.

To solve the above coercivity  problem on the Landau-de Gennes energy density,  Feng and Hong \cite {FH}  proposed a new Landau-de Gennes energy density, which  keeps physical quantities of  the original Landau-de Gennes density. More precisely, it was observed in \cite {FH} that for uniaxial tensors,
the original
third order  term  on  $L_4$ in  (\ref {LG}), proposed by
Schiele and Trimper  \cite{ST}*{p.~268} in physics, is a linear combination of a fourth order term and a second order term; i.e. for $Q\in S_*$, 	we have
\begin{align*}\alabel{third}
	Q_{lk}\frac{\p Q_{ij}}{\p x_l}\frac{\p Q_{ij}}{\p x_k}=
	\begin{cases}
		\frac{3}{s_+}Q_{ln}Q_{kn}\frac{\p Q_{ij}}{\p x_l} \frac{\p Q_{ij}}{\p x_k}-\frac {2s_+}3|\na Q|^2 &\text{ for } L_4\geq0,
		\\
		-\frac{3}{s_+} (|Q|^2|\D Q|^2+  Q_{ln}Q_{kn}\frac{\p Q_{ij}}{\p x_l} \frac{\p Q_{ij}}{\p x_k})+ \frac{4s_+}{3}|\D Q|^2&\text{ for } L_4<0,
	\end{cases}
\end{align*}
where $S_*$ is the space  of all uniaxial $Q$-tensors defined by
\[S_*:=\left\{Q  \in S_0:\quad Q =s_+ (u\otimes u-\frac 13 I),\quad  u\in S^2, \quad s_+:=\frac{b+\sqrt{b^2+24ac}}{4c} \right\}.\]
It was observed in \cite{FH}   that
 \[Q_{ln}Q_{kn}\frac{\p Q_{ij}}{\p x_l}\frac{\p Q_{ij}}{\p x_k}=\frac{8}{5} L^{(4)}_5- \frac{2}{5}L^{(4)}_6 +\frac{2}{5}L_7^{(4)},\]
where  $L^{(4)}_5, L^{(4)}_6, L^{(4)}_7$ are three fourth order terms defined in \cite{LMT}    satisfying the same physical invariance as the original cubic term $Q_{lk}\frac{\p Q_{ij}}{\p x_l}\frac{\p Q_{ij}}{\p x_k}$.

Using \eqref{third}, Feng and Hong
\cite {FH}  introduced
a new Landau-de Gennes  energy    given by
\begin{equation}\label{LDG}
	E_{LG}(Q; U):=\int_{U} f(Q,\D Q)\,dx =\int_{U} \(f_E(Q,\D Q)+\frac 1L  f_B(Q) \)\,dx,
\end{equation}
where
\[  f_B(Q):=  \tilde f_B(Q)-\min_{Q\in S_0}  \tilde f_B(Q)\geq 0\]
and
\begin{align*}
	f_E(Q,\D Q)=&\frac{\tilde L_1}{2}|\D Q|^2
	+\frac{L_2}{2} \frac{\p Q_{ij}}{\p x_j}\frac{\p Q_{ik}}{\p x_k}
	+\frac{L_3}{2} \frac{\p Q_{ik}}{\p x_j}\frac{\p Q_{ij}}{\p x_k}+\frac{3L_4}{2s_+} L^{(4)}(Q,\D Q).
	\alabel{f_E}
\end{align*}
Here $\tilde L_1=L_1-\frac{2}{3s_+}  L_4$ for $L_4\geq0$, $\tilde L_1=L_1+\frac{4}{3s_+} L_4$ for $ L_4<0$ and
\begin{align*}
	  L^{(4)}(Q,\D Q):=\begin{cases}  Q_{ln}Q_{kn}\frac{\p Q_{ij}}{\p x_l}\frac{\p Q_{ij}}{\p x_k}&\mbox{    for } L_4\geq0,
		\\
	  Q_{ln}Q_{kn}\frac{\p Q_{ij}}{\p x_l}\frac{\p Q_{ij}}{\p x_k}-|Q|^2|\D Q|^2&\mbox{ for } L_4<0.
	\end{cases}
\end{align*}
In  \eqref{LDG}, the constant $L$ is a rescaled dimensionless parameter, which drives four elastic constants
$\tilde L_1,\cdots,L_4$ to zero simultaneously as $L$ tends to zero. This corresponds to the large body limit which is of great importance in physics (cf. \cite{Ba}, \cite{Ga}).
We always assume that the constants $\tilde L_1,L_2$ and $L_3$ satisfy
\begin{align}\label{L cond}
	&  \tilde L_1+L_3>0,\quad  2\tilde L_1-L_3>0,
	\quad \tilde L_1+\frac53L_2+\frac16L_3>0.
\end{align}
Under the condition \eqref{L cond},  the new Landau-de Gennes elastic energy density in \eqref{LDG} satisfies the coercivity condition; i.e.  $f(Q,\D Q)$ for any $Q\in S_0$ is bounded from below by $\frac \alpha 2|\D Q|^2$ with some $\alpha >0$ (cf. \cite{KRSZ}, \cite{FH}).

In this paper, we investigate the  Beris-Edwards system for the  Landau-de Gennes energy \eqref{LDG} with $L_4\neq 0$.   The  Beris-Edwards system with $L_2=L_3=L_4=0$ has been extensively studied by many authors (see \cite{PZ11}, \cite{PZ12}, \cite{WWZ}). The  Beris-Edwards system  introduced in  \cite{BE94} is a system of coupling Navier-Stokes equations with the gradient flow for the Landau-de Gennes energy. More precisely, let $v:\R^3\to \R^3$ be the velocity of the fluid and let $Q:\R^3 \to S_0$ be  a $Q$-tensor order parameter, which depends on the director of the molecular field. The symmetric and skew-symmetric parts of the tensor $\D v$ are
\[ D =\frac 12(\D v +(\D v)^T ),\quad\Omega =\frac 12(\D v -(\D v)^T).\]
Define $ [Q,\Omega]  :=Q \Omega -\Omega Q $ to be  the Lie bracket product and set
\[S(Q,v)=\xi\Big(D(Q+\frac13 I)+(Q+\frac13 I)D-2(Q+\frac13 I)(Q\cdot D)\Big)- [Q,\Omega].\]
Then the    Beris-Edwards system (cf. \cite{BE94}, \cite{PZ12}) is given by
\begin{align*}
	\p_t v+v\cdot\D v- \nu \Delta v  +\D P=&\D\cdot\Big( \tau (Q,\D Q)+ \sigma (Q,\D Q)\Big),\alabel{OBE1}
	\\
	\D\cdot v=&0,\alabel{OBE2}
	\\
	\p_t Q +v\cdot \D  Q -S(Q,v)=& \Gamma H(Q,\D Q),
	\alabel{OBE3}
\end{align*}
where $H(Q,\D Q)$ is the   molecular field, $P$ is the pressure, the antisymmetric part of the distortion stress $\tau (Q,\D Q)=[Q, H]$ and $\sigma (Q,\D Q)$  is	  the  distortion stress (cf. \cite{QS})  given by
\begin{align*}
\sigma_{ij}(Q, \D Q)=-\xi (QH+HQ+\frac 23 H)_{ij}+2\xi (Q\cdot H)(Q+\frac I 3)_{ij}-\p_{p_{kl}^j} f(Q,\D Q)\D_iQ_{kl}. \alabel{Ixi}
\end{align*}
Here and in the sequel,  we denote $ \p_{p_{kl}^j} f(Q,\D Q):=\frac{\p f(Q,\nabla Q)}{\p (\D_j Q_{kl})}$ with $p=\nabla Q$.

  Paicu and Zarnescu \cite {PZ12} first used the Landau-de Gennes energy   with a parameter $L>0$ to formulate the  Beris-Edwards system    in the case of  $\xi =0$. Later, they \cite{PZ11} also extended their result to the case of $\xi\neq 0$.
  For simplicity,  we only consider the case of  $\xi =0$ and $\Gamma=\nu =1$.
   For the static case, it is well-known  (see, e.g., \cite{MZ}, \cite{NZ})  that  the Landau-de Gennes energy   \eqref{LDG}  is a standard   biaxial approximation of  uniaxial Q-tensors as $L\to 0$. For the reason,  we  investigate \eqref{OBE1}-\eqref{OBE3} for uniaxial Q-tensors through the biaxial approximation  of the Landau-de Gennes energy   \eqref{LDG}. More precisely, as in \cite{WZZ}, we study
  a new rescaled Beris-Edwards system  with a parameter $L>0$, which is different from one  in \cite{PZ12},  defined by:
\begin{align*}
	\p_t v_L+v_L\cdot\D v_L- \Delta  v_L +\D P_L=&\D\cdot\Big([Q_L,\MH(Q_L,\D Q_L)]+\sigma (Q_L,\D Q_L)\Big),\alabel{RBE1}
	\\
	\D\cdot v_L=&0,\alabel{RBE2}
	\\
	\p_t Q_L+v_L\cdot \D  Q_L +[Q_L, \Omega_L]=& \MH(Q_L,\D Q_L)+\frac 1Lg_B(Q_L),\alabel{RBE3}
\end{align*}
where
\begin{align*}
	\MH(Q_L,\D Q_L)_{ij}=&\frac12\(\D_k[\p_{p_{ij}^k} f_E(Q_L,\D Q_L)]+\D_k[\p_{p_{ji}^k} f_E(Q_L,\D Q_L)]\)
	\\
	&- \frac12\(\p_{Q_{ij}} f_E(Q_L,\D Q_L)+\p_{Q_{ji}} f_E(Q_L,\D Q_L)\)
	\\
	&-\frac{\delta_{ij}}3\sum_{l=1}^3\(\D_k[\p_{p_{ll}^k} f_E(Q_L,\D Q_L)]-\p_{Q_{ll}} f_E(Q_L,\D Q_L)\),
	\alabel{Mol}
\end{align*}
the term $g_B(Q_L)$ is defined by
\begin{align*}\alabel{gB}
	g_B(Q_L) :=&aQ_L +b \big( Q_L  Q_L -\frac {1}3\tr(Q^2)I \big)-cQ_L\tr(Q_L ^2)
\end{align*}
and $\sigma (Q_L,\D Q_L)$  is the distortion stress tensor with
\begin{equation*}\label{Mo0}
	\D_j\sigma_{ij}(Q_L,\D Q_L)=-\D_j\(\D_i(Q_L)_{kl}\p_{p_{kl}^j} f_E(Q_L,\D Q_L)\).
\end{equation*}
Set
\[H^2_{Q_e} (\R^3; S_0)=\{Q\in S_0: Q-Q_e\in H^2(\R^3)\},\]
where $Q_e=s_+ ( e \otimes  e -\frac13 I)\in S_*$ and $e \in S^2$ is a constant vector. The assumption on the constant matrix $Q_e$  is needed for studying  uniaxial Q-tensors in below Theorems 2-3.

We call $(Q_L,v_L)$  a strong solution to the system \eqref{RBE1}-\eqref{RBE3} in $\R^3\times (0,T)$ for some $T>0$ if it satisfies the system a.e. in $(x,t)\in\R^3 \times (0,T)$ and
\begin{align*}
	&Q_L\in L^2(0,T;H^3_{Q_e}(\R^3))\cap L^\infty(0,T;H^2_{Q_e}(\R^3)),
	\quad \p_tQ_L\in L^2(0,T;H^1(\R^3)),
	\\
	&v_L\in L^2(0,T;H^2(\R^3))\cap L^\infty(0,T;H^1(\R^3)).
\end{align*}
Then we have
\begin{theorem}[Local Existence]\label{thm loc} For each $L>0$,
	let $(Q_{L,0},v_{L,0})\in H^{2}_{Q_e}(\R^3; S_0)\times H^{1}(\R^3;\R^3)$ be an initial value  satisfying $\div  v_{L,0}=0$
	and $\|Q_{L,0}\|_{L^\infty(\R^3)}\leq K$ for a constant $K>0$.
	Then there is a unique strong
	solution $(Q_L, v_L)$ to the system \eqref{RBE1}-\eqref{RBE3}  in $\R^3\times[0, T_L)$ with initial data $(Q_{L,0},v_{L,0})$ for some
	$T_L>0$.
\end{theorem}
\noindent
 Theorem \ref{thm loc} may be known for some experts.  For example, in the case  of  $L_2=L_3=L_4=0$, Theorem \ref{thm loc}  was proved in
 \cite{PZ12}.
However, since there exists some new difficulties on $f_E$ with $L_4\neq 0$, we provide a detailed proof of Theorem \ref{thm loc} in Section 5.

\medskip

Next, we formulate the   Beris-Edwards system for uniaxial Q-tensors.
In  their book \cite{BE94},   Beris and Edwards  suggested a  hydrodynamic theory to describe   flows of liquid
crystals  for uniaxial Q-tensors $Q\in S_*$, but they could not write an explicit  form of molecular field $H(Q,\D Q)$ for $Q\in S_*$
with non-zero elastic constants $L_2$, $L_3$, $L_4$. Recently, the  explicit  form of the molecular field $H(Q,\D Q)$ for $Q\in S_*$  with
general elastic constants  was given in  \cite{FH}, so we can apply the  form to formulate the Beris-Edwards system for uniaxial Q-tensors.
Then,  for a uniaxial tensor $Q\in S_* $,  the molecular field    is given by
\begin{align*}
&H(Q, \D Q)
:=\frac{1}2[\nabla_k \p_{p^k} f_E-\nabla_k \p_{p^k} f_E +(\nabla_k \p_{p^k} f_E-\nabla_k \p_{p^k} f_E) ^T](s_+^{-1}Q+\frac13I)
\\
&+\frac{1}2(s_+^{-1}Q+\frac13I)[\nabla_k \p_{p^k} f_E-\nabla_k \p_{p^k} f_E +(\nabla_k \p_{p^k} f_E-\nabla_k \p_{p^k} f_E) ^T]
\\
&-(s_+^{-1}Q+\frac13I)[\nabla_k \p_{p^k} f_E-\nabla_k \p_{p^k} f_E +(\nabla_k \p_{p^k} f_E-\nabla_k \p_{p^k} f_E) ^T](s_+^{-1}Q+\frac13I)
	.\alabel{EL}
\end{align*}
Using the new  molecular field \eqref{EL}, the uniaxial Beris-Edwards system with non-zero elastic constants
$L_1,\cdots, L_4$ is:
\begin{align}
	(\p_t+v\cdot\D- \Delta )v +\D P=&\D\cdot\Big([Q,H]+\sigma(Q,\D Q)\Big),\label{BE1}\\
	\D\cdot v=&0,\label{BE2}\\
	(\p_t +v\cdot \D )Q+[Q, \Omega]=& H(Q,\D Q) \label{BE3}.
\end{align}

Secondly,  we prove   existence of the strong solution to \eqref{BE1}-\eqref{BE3}  in the following:
\begin{theorem} \label{thm 2}
	Assume that $(Q_0,v_0)\in H^2_{Q_e}(\R^3;S_*)\times H^{1}(\R^3;\R^3)$ and $\div v_0=0$. Then there is a unique strong
	solution $(Q, v)$ to the  uniaxial Beris-Edwards system \eqref{BE1}-\eqref{BE3}  in $\R^3\times[0, T^*)$ with initial data $(Q_0,v_0)$.
	Moreover, there are two uniform positive constants $\varepsilon_0$ and $R_0$, independent of the solution $(Q, v)$, such that at a singular point $x_i$, the maximal existence time $T^*$
satisfies
	\begin{equation*}
		\limsup_{t\rightarrow T^*}\int_{B_R(x_i)}|\D Q(\cdot,t)|^3+|v(\cdot,t)|^3\,dx\geq \varepsilon_0^3
	\end{equation*}
	for any $R>0$ with $R\leq R_0$.
\end{theorem}

For the proof of Theorem \ref{thm 2}, one of the key steps is to establish    Proposition \ref{prop Extension} and obtain that for a short time $T_1>0$,  the strong  solution to the system \eqref{RBE1}-\eqref{RBE3}  with  initial data $(Q_0, v_0)$ satisfies the uniform estimate:
\begin{align*}
	&\sup_{0\leq s\leq T_1}\left(\|\D Q_L(s)\|_{H^1(\R^3)}^2+\|v_L(s)\|_{H^1(\R^3)}^2+\frac 1 L\| Q_L(s)-\pi(Q_L(s))\|_{H^1(\R^3)}^2\right)
	\\
	&+\|\D^2 Q_L\|_{L^2(0,T_1;H^1(\R^3))}^2+\|\p_t Q_L\|_{L^2(0,T_1;H^1(\R^3))}^2
	\\
	&+\|\D v_L\|_{L^2(0,T_1;H^1(\R^3))}^2 +\frac 1 L\|\D( Q_L-\pi(Q_L))\|_{H^1(\R^3)}^2\leq C.
\end{align*}
Here  $\pi(Q_L)$ is the projection of $Q_L$ defined below  in the proof of Theorem \ref{thm 3}.
The proof of Proposition \ref{prop Extension} is  sophisticated and it will also play a crucial role  in the proof of Theorem \ref{thm 3} below. We will outline more details about it later.

\begin{remark}
	It was pointed out in  \cite{BE94}  that \eqref{BE1}-\eqref{BE3} can be reduced to the hydrodynamic flow of the Oseen-Franks energy, known as the Ericksen-Leslie system. In fact,    multiplying $u_j$ to \eqref{BE3} and  employing $|u|^2=1$, one can   check that \[\sigma(Q,\D Q)=-\na u^T\frac{\p W(u,\D u)}{\p (\na u)}, \quad \p_{p_{kl}^j} f_E(Q_L,\D Q_L)=s_+^{-1}u_l\frac{\p W(u,\na u)}{\p(\D_j u_k)}.\]
\end{remark}

As mentioned in  Remark 1,  the  uniaxial Beris-Edwards system \eqref{BE1}-\eqref{BE3}   is a generalization of the Ericksen-Leslie system.
Ericksen \cite {Er} and  Leslie \cite{Le} in the 1960s proposed the celebrated hydrodynamic theory to describe the behavior of liquid crystal
flows based on the Oseen-Franks theory.
The question of
global existence  of   weak solutions to the Ericksen-Leslie system is very challenging.
Similarly to the idea of Chen-Struwe \cite{CS} on    harmonic maps, Lin and Liu \cite {LL2} introduced the
Ginzburg-Landau approximation for the Ericksen-Leslie system to   solve the existence problem.
In $\R^2$,  Hong \cite{Ho} and Hong-Xin \cite {HX} proved that the solutions  of the Ginzburg-Landau approximate systems approach  the solution of the Ericksen-Leslie system in a short time by using the idea of Struwe \cite{St} on the harmonic map flow.
In $\R^3$, Hong, Li and  Xin \cite{HLX}  showed the strong  convergence  of the Ginzburg-Landau approximate system with unequal Frank
constants before the blow-up time of the Ericksen-Leslie system. In \cite{FHM},    we   proved the
smooth convergence  of the Ginzburg-Landau approximate systems for a general Ericksen-Leslie system with Leslie tensors before the blow-up
time. Furthermore, Kortum \cite{Kor} established the global existence of weak
solutions to the two-dimensional Ericksen–Leslie system by
using a concentration–cancellation method to handle the nonlinear
stress terms.

By comparing with Ginzburg-Landau models for superconductivity theory,  Gartland   \cite {Ga} emphasized  importance  of  the convergence on  Landau-de Gennes
solutions. In physics, both  the  Ericksen-Leslie
theory   and the  Beris-Edwards  theory  should   unify    in modeling uniaxial    state of nematic liquid crystals (cf. \cite[Chapter 11]{BE94}), so it  is  very
interesting to  give a rigorous mathematical proof to verify that  as $L\to 0$, the solutions of the rescaled Beris-Edwards system \eqref{RBE1}-\eqref{RBE3}
can approach a  solution of the uniaxial Beris-Edwards system  \eqref{BE1}-\eqref{BE3}.

Thirdly,   we  solve the above convergence  problem for the Beris-Edwards system in the following:

\begin{theorem}\label{thm 3} Assume that $(Q_0, v_0)\in  H_{Q_e}^{2}(\R^3;S_*)\times H^{1}(\R^3;\R^3)$ with $\div  v_{0}=0$.
	For each $L>0$, let $(Q_L, v_L)$ be the unique strong  solution to the rescaled Beris-Edwards  system \eqref{RBE1}-\eqref{RBE3} in
	$\R^3\times[0,T_L)$ with  initial data $(Q_0, v_0)$ for the maximal existence time    $T_L$.
	Let  $(Q,v)$  be the strong solution to  the uniaxial Beris-Edwards system  \eqref{BE1}-\eqref{BE3}  in  $\R^3\times[0,T^*)$   with the same initial data $(Q_0,
v_0)$ and  the maximal existence time    $T^*$  in Theorem  \ref{thm 2}. Then, for any $T\in(0, T^*)$,   there exists a sufficiently small
$L_T>0$ such that $T\leq T_L$ for any $L\leq  L_T$.
	Moreover,    as $L\to 0$, we have
	\begin{equation}\label{con1}
		(\D Q_L,v_L)\rightarrow(\D Q,v) \qquad\text{in }~~L^\infty(0,T;L^2(\R^3))\cap L^2(0,T;H^1(\R^3))
	\end{equation}
	and
	\begin{equation}\label{con2}
		(\D Q_L,v_L)\rightarrow(\D Q,v) \qquad\text{in }~~C^\infty(\tau,T;C_{loc}^\infty(\R^3)) \quad\text{for  any }\tau>0.
	\end{equation}
\end{theorem}

For the proof of Theorem \ref{thm 3}, the main ideas are to establish   uniform estimates   on  higher order derivatives of $(Q_L,v_L)$ in $L$. Using similar methods in \cite{HX},\cite{FHM},  we can  handle all terms involving $f_E(Q_L)$, but the main difficulty is to obtain the uniform estimate of the terms involving   $\frac 1Lg_B(Q_L)$ when $L\to 0$. To handle those difficult terms,  we use a concept of a projection  near $S_*$, which was first introduced on Riemannian manifolds by Schoen  and Uhlenbeck \cite{ScU}. Denote
\begin{align}\alabel{S delta}
	S_\delta:=\left \{Q\in S_0:\quad \dist(Q;S_*)\leq \delta\right\}.
\end{align}
Let $\pi: S_{\delta}\to S_*$ be the smooth projection map for a small $\delta>0$  so that $\pi (Q)$ is the nearest point; i.e. $|Q-\pi (Q)|=\dist(Q; S_*)$ for
$Q\in S_{\delta}$.
For each smooth $Q_L(x)\in S_{\delta}$,  there is a rotation $R(Q_L(x))\in SO(3)$ such that $R^T(Q_L(x))Q_L(x)R(Q_L(x))$ is diagonal. However, it was pointed out \cite {FH} that  $R(Q_L(x))$ is  smooth  in $U$ except for a singular set $\Sigma_L$.    In this paper, we find  a new approach to avoid the difficulty of the singular set $\Sigma_L$  arising from \cite {FH}.  We   outline main steps of the new approach  as follows:

The first key step is  to establish new estimates in Lemma \ref{lem f_B 1st} to overcome the difficulty arising from the term  $\frac 1Lg_B(Q_L)$. More precisely,
for each smooth  $Q\in S_\delta$ with a small $\delta$,   there exists a smooth rotation  $R (Q)\in SO(3)$  such that
\begin{align}\label{Q1}
	\tilde Q :=R^T (Q) Q    R (Q)   =\begin{pmatrix}
		\tilde Q_{11} & \tilde Q_{12}&0
		\\ \tilde Q_{21}& \tilde Q_{22}&0
		\\0&0&\tilde Q_{33}
	\end{pmatrix}
\end{align} and  in particular, for $Q=\pi(Q)\in S_{*}$, we have
\begin{align*}
 	\left . R (Q)^T  Q R (Q)\right |_{Q=\pi (Q)}   =\begin{pmatrix}
 		\frac{-s_+}{3}&0&0\\0&\frac{-s_+}{3}&0\\0&0&\frac{2s_+}{3}
 	\end{pmatrix}.
 \end{align*}
	 For any $Q\in S_\delta$ with a small $\delta$,  there is a uniform constant $\lambda >0$ such that
\begin{align*}
	&\frac{\lambda}{2} |\xi|^2\leq\p^2_{\tilde Q_{ij}\tilde Q_{kl}}  f_B(\tilde Q)\xi_{ij}\xi_{kl}
\end{align*}
for all $\xi\in S_0$ having the form of block-diagonal  matrices given by
\begin{align*}
	\xi =\begin{pmatrix}
		\xi_{11} & \xi_{12}&0
		\\ \xi_{21}& \xi_{22}&0
		\\0&0&\xi_{33}
	\end{pmatrix},\end{align*}
 which improves  a result of diagonal matrices in \cite {FH}.
Then for any $Q_L\in S_\delta$, we   derive an estimate
\begin{align*}
	&\frac{\lambda}{4} |\D (Q_L-\pi(Q_L))|^2
	\\
	\leq& \p^2_{\tilde Q_{ij}\tilde Q_{kl}}  f_B(\tilde Q_L)\D_{x_\beta}(\tilde Q_L)_{ij}\D_{x_\beta} (\tilde
Q_L)_{kl}
	+C|\D Q_L|^2|Q_L-\pi(Q_L)|^2\alabel{2f1}.
\end{align*}
As an application of \eqref{2f1},   we can handle the   term on
$\frac{1}{L} g_B(Q_L)$. More precisely,
we rotate the equation \eqref{RBE3} by $R(Q_L)\in SO(3)$ such that $g_B(\tilde Q_L)$ has  the same matrix form of $\tilde Q_L$ in \eqref{Q1}.
For any $Q_L\in S_\delta$,  we find
 in Lemma \ref{lem gb}
 \begin{align*}
		&\<\D g_B(\tilde Q_L) , R^T(Q_L) \D Q_L R(Q_L)\>
		\\
		\leq& -\frac{\lambda}{8}|\D (Q_L-\pi(Q_L))|^2+C|Q_L-\pi(Q_L)|^2|\D Q_L|^2
		.
	\end{align*}
 The second key step is  to apply the Gagliardo-Nirenberg interpolation to obtain a local $L^3$-type of estimate
\begin{align}
	\sup_{T_0\leq t\leq T_L,x_0\in\R^3}\int_{B_{R}(x_0)}|\D Q_L|^3+|v_L|^3+\frac{|Q_L-\pi(Q_L)|^3}{L^{\frac32}}\,dx\leq\varepsilon_0^3.
\end{align}
Then we establish a key Proposition \ref{prop Extension}, which  implies
   that there exists a subsequence $(Q_{L},v_{L})$ such that as $L\to 0$,
	\begin{equation*}
		(\D Q_{L},v_{L})\rightarrow(\D Q,v),\qquad\text{in }~~L^\infty(0,T_M;L_{loc}^2(\R^3))\cap L^2(0,T_M;H_{loc}^1(\R^3))
	\end{equation*}
for some $T_M>0$. Then we   prove   the  local convergence.

Another of the key steps  to prove Theorem \ref{thm 3} is to  establish the sophisticated uniform  estimate  of $(\nabla^{k+1} Q_L, \nabla^{k} v_L)$ in $L$ for any integer $k\geq 2$. We would point out that our proof on high order uniform estimates is new and  different from one  in \cite{FHM}.
The last step is to extend $T_M$
to any $T<T^*$.   For this,    we choose
$M= 2\sup_{0\leq t\leq T}\|(\D Q, v)\|_{H^1(\R^3)}^2$ in Proposition \ref{prop Extension}. Then we
combine the energy identity (see below \eqref{energy Q})  with   higher order estimates  to verify  that    $( Q_L, v_L)$ satisfies the assumption (see below \eqref{prop eq}) at $t=T_M$.  Utilizing Proposition \ref{prop Extension} with a new initial data $(Q_L,v_L)$ at $t=T_M$, we  extend the strong solution $(Q_L,v_L)$ to the
time  $T_1=:\min\{T,2T_M\}$, $T<T^*$. Then the solutions $( Q_L, v_L)$ to the rescaled Beris-Edwards  system \eqref{RBE1}-\eqref{RBE3} converge smoothly to the solution $(Q,v)$  in $\R^3\times (0, 2T_M]$ for sufficiently small $\varepsilon$. Repeating  above steps, we  establish    \eqref{con2} for any $T<T^*$ in Theorem \ref{thm 3}.

\

Finally, we would like to make two remarks.

\begin{remark}
	When $L_4=0$,  Wang, Zhang and Zhang \cite{WZZ} proved some related convergence of \eqref{RBE1}-\eqref{RBE3}  with smooth initial values
	to the Ericksen-Leslie system in $\R^3$, but   not   to
	the uniaxial Q-tensor  Beris-Edwards system   \eqref{BE1}-\eqref{BE3}.  It seems that their method only works for smooth initial values.
	Recently,  Xin and Zhang \cite{XZ} proved that  the weak convergence also holds in $\R^2$  for \eqref{RBE1}-\eqref{RBE3} with $L_2=L_3=L_4=0$.
\end{remark}

\begin{remark}  For the case $\xi \neq 0$ in \eqref{Ixi}, additional terms appear in the stress tensor
	$\sigma_{ij}(Q,\nabla Q)$ in (1.8) and in $S(Q,v)$ in (1.10).
	As observed in \cite[ p.~2013]{PZ11}, these terms cancel in the
	energy estimates and thus do not cause    trouble at the level of strong
	solutions, so   our Theorem  \ref{thm 2} can be generalized to the case of $\xi\neq 0$, but it needs  some details. However, it is unclear   whether Theorem \ref{thm 3}  holds or not for $\xi\neq 0$ since it involves a large amount of calculations due to extra terms with $\xi\neq 0$.  Therefore,  we will study the case of $\xi\neq 0$  in  future works.
\end{remark}

The  paper is organized as follows.  In Section 2, we derive  some a-priori estimates on the strong solution $(Q_L,v_L)$ of the system
\eqref{RBE1}-\eqref{RBE3} in
$\R^3\times[0,T_L]$.  In Section 3, we prove Theorem \ref{thm 2}. In Section 4, we prove   Theorem  \ref{thm 3}.  In Section 5, we prove
Theorem
\ref{thm loc}.

\section{A-priori estimates}
In this section, we derive some a-priori estimates on the strong solution $(Q_L,v_L)$ of the system \eqref{RBE1}-\eqref{RBE3} in
$\R^3\times[0,T_L]$.
\subsection{Properties of the density}
In order to obtain a-priori energy estimates, we need to establish some key properties of the energy density.
Under the condition \eqref{L cond}, one can verify using a result in  \cite{KRSZ}  that there are two uniform constants $\alpha>0$ and $\Lambda
>0$ such that for any  $Q\in \mathbb{M}^{3\times 3}$ and   $p \in\mathbb{M}^{3\times 3}\times \R^3$, $f_E(Q,p)$ also satisfies
	\begin{align*}
		\frac \alpha 2 |p|^2\leq  f_E(Q,p)\leq&  \Lambda(1+|Q|^2)|p|^2, \quad|\p_Qf_E(Q,p)|\leq  \Lambda (1+|Q|)|p|^2,
		\\
		|\p^2_{Qp}f_E(Q,p)|\leq&  \Lambda (1+|Q|)|p|,\qquad |\p^2_{pp}f_E(Q,p)|\leq  \Lambda(1+|Q|^2).\alabel{sec2 f_1}
	\end{align*}
Noting that $f_E(Q,p)$ is
quadratic in $p$ and satisfies \eqref{sec2 f_1}, one   has  (cf. \cite{HM})
	\begin{align}\label{sec2 f_E}
	 \frac \alpha 2|\xi|^2\leq	\p^2_{p^i_{kl}p^j_{mn}} f_{E}(Q,p)\xi^i_{kl}\xi^j_{mn}\leq  \Lambda (1+|Q|^2)|\xi|^2,  \quad \forall
\xi\in\mathbb{M}^{3\times 3}\times \R^3.
	\end{align}
Recall that
\begin{align}
	S_\delta=\left \{Q\in S_0:\quad \dist(Q;S_*)\leq \delta\right\}.
\end{align}
We assume that $\delta>0$ is sufficiently small throughout this paper.  Let $\pi(Q )$ be a smooth projection from $S_{\delta}$ to $S_*$.
Then $ f_B(Q )$ satisfies  (cf. \cite{NZ}, \cite{FH})
\begin{align}
	\frac{\lambda}2|Q -\pi(Q )|^2\leq  f_B&(Q)\leq  C |Q -\pi(Q )|^2
	\label{f_B and dist}
\end{align} for some $C>0$.

Note that the principal eigenvalue and eigenvector of  $Q\in S_\delta$ are smoothly close to those of its projection $\pi(Q)\in S_*$. Then for each $ Q\in S_\delta$,  there exists a smooth rotation $R(Q)\in SO(3)$ such that  $R^T(Q)QR(Q)$ is block-diagonal; i.e.
  for any $Q\in S_\delta$,  we have
\begin{align*}
	\tilde Q =R^T(Q)   Q   R(Q)   =\begin{pmatrix}
		\tilde Q_{11} & \tilde Q_{12}&0
		\\ \tilde Q_{21}& \tilde Q_{22}&0
		\\0&0&\tilde Q_{33}\alabel{Q13}
	\end{pmatrix}.
\end{align*}Since
every $\pi(Q)\in S_*$ has a constant number of distinct eigenvalues,  there is another smooth rotation $R(Q)\in SO(3)$   such that   $R^T(\pi (Q)
)\pi(Q)R(\pi(Q))$ is diagonal  (cf. \cite{No})  and   $R^T(Q)QR(Q)$  for any $Q\in S_\delta$ is still block-diagonal of the form \eqref{Q13}.
Since  $S_*$ only have three diagonal  matrices, we can assume
without loss of generality that
 \begin{align}
 	R^T(\pi(Q)) \pi (Q  )R( \pi(Q))   =\begin{pmatrix}
 		\frac{-s_+}{3}&0&0\\0&\frac{-s_+}{3}&0\\0&0&\frac{2s_+}{3}
 	\end{pmatrix}=:Q^+
 	\label{Q+}
 \end{align}
It can be checked that $Q^+\tilde Q=\tilde Q Q^+$.

\begin{lemma}
\label{lem f_B 1st}
 For any $Q\in S_\delta$ , let $\tilde Q$ be defined in \eqref{Q13}. Then for a  sufficiently small $\delta >0$, the Hessian of the bulk
density $f_B(Q)$ in $S_\delta$ with satisfies
\begin{align*}
	&\frac{\lambda}{2} |\xi|^2\leq\p^2_{\tilde Q_{ij}\tilde Q_{kl}}  f_B(\tilde Q)\xi_{ij}\xi_{kl}\alabel{eq pos}
\end{align*}
for all $\xi\in S_0$ of the form
\begin{align*}
	\xi =\begin{pmatrix}
		\xi_{11} & \xi_{12}&0
		\\ \xi_{21}& \xi_{22}&0
		\\0&0&\xi_{33}
	\end{pmatrix},\alabel{xi}\end{align*}
where $\lambda=\min\{{s_+b},{3a}\}>0$.

\end{lemma}
\begin{proof}
Calculating   second partial derivatives of $f_B(\tilde Q)$ with respect to $\tilde Q$, we have
\begin{align*}
	\p_{\tilde Q_{\tilde i\tilde j}}\p_{\tilde Q_{ij}} f_B(\tilde Q)=& -a\delta_{i\tilde i}\delta_{j\tilde j}
	-b(\delta_{\tilde ij}\tilde Q_{\tilde j i}+\delta_{\tilde ji}\tilde Q_{j\tilde i})+c(\delta_{i\tilde i}\delta_{j\tilde j}|\tilde
	Q|^2+2\tilde Q_{ij}\tilde Q_{\tilde i\tilde j}).
\end{align*}
Note  from \cite{MN} that
\begin{align*}
	|Q^+|^2=\frac23s_+^2,\qquad 2cs_+^2=3a+bs_+.\alabel{abc}
\end{align*}
Then we have
\begin{align*}
		\p_{\tilde Q_{\tilde i\tilde j}}\p_{\tilde Q_{ij}} f_B(\tilde Q)\Big|_{\tilde Q =Q^+}=&\(b\Big(\frac13s_+\delta_{i\tilde
i}\delta_{j\tilde j}-(\delta_{\tilde ij}\tilde Q_{\tilde j i}+\delta_{\tilde ji}\tilde Q_{j\tilde i})\Big)+2c\tilde Q_{ij}\tilde Q_{\tilde
i\tilde j}\)\Big|_{\tilde Q =Q^+}
	.\alabel{fB 2}
\end{align*}
In the case of  $i=j=\tilde i=\tilde j$ in \eqref{fB 2}, we apply the relation \eqref{abc} to obtain
\begin{align}\label{f_B p1.0}
	\p_{\tilde Q_{11}}\p_{\tilde Q_{11}} f_B(Q^+)  =&\(s_+b+\frac{2s_+^2}{9}c\)=\frac 13a+\frac{10s_+}{9}b=\p_{\tilde
		Q_{22}}\p_{\tilde Q_{22}} f_B(Q^+),
	\\
	\p_{\tilde Q_{33}}\p_{\tilde Q_{33}} f_B(Q^+)=&-s_+ b+\frac{8s_+^2}{9}c=\frac43a-\frac{5s_+}{9}b.
\end{align}
For the case of $i=j\neq\tilde i=\tilde j$ in \eqref{fB 2}, we compute
\begin{align}
	\p_{\tilde Q_{11}}\p_{\tilde Q_{22}} f_B(Q^+) =&2cQ^+_{11}Q^+_{22}=\frac{2s_+^2}{9}c=\frac13a+\frac{s_+}{9}b,
	\\
	\p_{\tilde Q_{11}}\p_{\tilde Q_{33}} f_B(Q^+) =&2cQ^+_{11}Q^+_{33}=-\(\frac23a+\frac{2s_+}{9}b\)=\p_{\tilde Q_{22}}\p_{\tilde
		Q_{33}} f_B(Q^+).\label{f_B p1.1}
\end{align}
For the remaining case  of either $i\neq j$ or $ \tilde i\neq \tilde j$ in \eqref{fB 2}, we find
\begin{align*}
	&\left.\(\sum_{i\neq j}\sum_{\tilde i,\tilde j}+\sum_ {\tilde i\neq \tilde j}\sum_{i,j}\)\p_{\tilde Q_{\tilde i\tilde j}}\p_{\tilde
Q_{ij}}f_B(\tilde Q)\xi_{\tilde i\tilde j}\xi_{ij}\right|_{\tilde Q=Q^+}
	\\=& \(\sum_{i\neq j}\sum_{\tilde i,\tilde j}+\sum_ {\tilde i\neq \tilde j}\sum_{i,j}\)\(2cQ^+_{ij}Q^+_{\tilde i\tilde j}+b(\frac
13s_+\delta_{i\tilde i}\delta_{j\tilde j}-\delta_{\tilde ij}Q^+_{\tilde j i}-\delta_{\tilde ji}Q^+_{j\tilde i})\)\xi_{\tilde i\tilde
j}\xi_{ij}
	\\
	=&\sum_{i\neq j}b\(\frac 13s_+-Q^+_{ii}-  Q^+_{jj}\)\xi_{ij}^2=s_+b(\xi_{12}^2+\xi_{21}^2),\alabel{Second derivative case 3}
\end{align*}
where we employed \eqref{abc} in the last step.

Using   \eqref{f_B p1.0}-\eqref{Second derivative case 3} with the fact that $\tr(\xi)=0$ for $\xi\in S_0$ defined in  \eqref{xi}, we have
\begin{align*}
	&\p^2_{\tilde Q_{ij}\tilde Q_{kl}}  f_B(Q^+)\xi_{ij}\xi_{kl}
	\\=& \(\frac 13a+\frac{10s_+}{9}b\)(  \xi_{11}^2 +  \xi_{22}^2)
	+\(\frac23a+\frac{2s_+}{9}b\)(
	\xi_{11}    \xi_{22})+2s_+b\xi_{12}^2
	\\
	&+\(\frac43a-\frac{5s_+}{9}b \) \xi_{33}^2 -\(\frac43a+\frac{4s_+}{9}b\)   \xi_{33}(
	\xi_{11}+  \xi_{22})
	\\
	=& s_+b(  \xi_{11}^2 +  \xi_{22}^2)+({\frac 83}a-\frac{s_+}{9}b)  \xi_{33}^2+2s_+b\xi_{12}^2
+\(\frac13a+\frac{s_+}{9}b\)(  \xi_{11}+  \xi_{22})^2
	\\
	=& s_+b(  \xi_{11}^2 +  \xi_{22}^2+\xi_{12}^2+\xi_{21}^2)+3a  \xi_{33}^2 \geq \lambda |\xi|^2,
	\alabel{f_B p1.2}
\end{align*}
where $\lambda=\min\{{s_+b},{3a}\}>0$.

Due to the continuity of  second derivatives of $f_B(\tilde Q)$ and the fact that $|\tilde Q-Q^+|\leq |Q-\pi(Q)|+2|\pi(Q)||R(Q)-R(\pi(Q))|\leq
C\delta,$ the claim \eqref{eq pos} follows from using \eqref{f_B p1.2} with sufficiently small $\delta$.
\end{proof}

\begin{cor}\label{lem fB}
	For any $Q \in S_\delta$, let $\tilde Q$ be defined in \eqref{Q13}. We have
 \begin{align*}
 		\frac{\lambda}{4}|\D (Q-\pi(Q))|^2\leq &\p^2_{\tilde Q_{ij}  \tilde Q_{kl}}  f_B(\tilde Q)\D  \tilde Q _{ij}\D\tilde Q_{kl}
 		 +C|Q-\pi(Q)|^2|\D Q|^2 .
 		\alabel{D2 tilde Q}
 	\end{align*}
 Moreover, for any  $k\geq 2$,  we have
 	\begin{align*}
 		\frac\lambda 4 |\D^k (Q -\pi(Q ))|^2
 		&\leq  \p^2_{ \tilde Q_{ij}  \tilde Q_{kl}}  f_B(\tilde Q )\D^k \tilde Q _{ij}\D^k\tilde Q_{kl}
 		\\
 		&+C \sum_{\substack{\mu_1\leq k-1\\\mu_1+\cdots+\mu_{k+1}=k}}|\D^{\mu_1}(Q-\pi(Q))|^2|\D^{\mu_2} Q|^2\cdots |\D^{\mu_{k+1}} Q|^2
 		.
 		\alabel{f_B kth rule}
 	\end{align*}
\end{cor}
\begin{proof}
Recall $\tilde Q, Q^+ $ from \eqref{Q13}-\eqref{Q+}. Then  we have
\begin{align*}
&\D  \tilde Q =\D[R^T(Q) QR(Q)-Q^+]
\\
=&\D(R^T(Q) [Q-\pi(Q)]R(Q))+\D[R^T(Q)\pi(Q)R(Q)-R^T(\pi(Q))\pi(Q)R(\pi(Q))]
\\
=&R^T(Q) \D[Q-\pi(Q)]R(Q)+ \D R^T(Q) [Q-\pi(Q)]R(Q)+R^T(Q)[Q-\pi(Q)]\D R(Q)
\\
&+R^T(Q)\D\pi(Q) [R(Q)-R(\pi(Q))]+[R^T(Q)-R^T(\pi(Q))]\D\pi(Q) R(\pi(Q))
\\
&+\D R^T(Q)\pi(Q)[R(Q)-R(\pi(Q))]+[R^T(Q)-R^T(\pi(Q))]\pi(Q)\D R(\pi(Q))
\\
&+\D [R^T(Q)-R^T(\pi(Q))]\pi(Q) R(\pi(Q))+R^T(Q)\pi(Q)\D[R(Q)-R(\pi(Q))].\alabel{A2}
\end{align*}
Note that
\begin{align*}
|\D[R(Q)-R(\pi(Q))]|=|\p_{Q_{ij}}[R(Q)-R(\pi(Q))]\D Q_{ij}|\leq C|Q-\pi(Q)||\D Q|
.\alabel{eq: R4}
\end{align*}
Using Young's inequality in \eqref{A2} with \eqref{eq: R4} and Lemma \ref{lem f_B 1st}, we get \eqref{D2 tilde Q}.

To establish \eqref{f_B kth rule} with $k\geq 2$, we note that
\begin{align*}
	&|\D^k R|+|\D^k \pi(Q)|
	\leq C\(\sum_{\mu_1+\cdots+\mu_{k}=k}|\D^{\mu_1} Q|\cdots |\D^{\mu_{k}} Q| \)
	,\alabel{kth R}
	\end{align*}
and
\begin{align*}
&|\D^k[R(Q)-R(\pi(Q))]|
\\
\leq&C|\p_{Q_{ij}}[R(Q)-R(\pi(Q))]\D^k Q_{ij}|+C\sum_{l=1}^{k-1}|\D^{k-l}(\p_{Q_{ij}}[R(Q)-R(\pi(Q))])||\D^lQ_{ij}|
\\
\leq&C\sum_{\mu_1+\cdots+\mu_k=k}|Q-\pi(Q)||\D^{\mu_1} Q|\cdots |\D^{\mu_{k}} Q|
.\alabel{eq: R6}
\end{align*}
It follows from Young's inequality for $|\D^k \tilde Q|$ and \eqref{A2} that
	\begin{align*}
&\frac{\lambda}{4}|\D^k (Q-\pi(Q))|^2
\\
\leq& \frac{\lambda}{2}|\D^k \tilde Q|^2+C(\lambda)\sum_{\mu_1+\cdots+\mu_3=k-1}|\D^{\mu_3}\D R(Q)|^2 |\D^{\mu_2}(Q-\pi(Q))|^2|\D^{\mu_1} R|^2
\\
&+C(\lambda)\sum_{\mu_1+\cdots+\mu_3=k}|\D^{\mu_1}[R(Q)-R(\pi(Q))]|^2|\D^{\mu_2}\pi(Q)|^2|\D^{\mu_3} R|^2
		.\alabel{kth fQQ}
	\end{align*}
Applying Lemma \ref{lem f_B 1st}  to \eqref{kth fQQ} with $\xi = \D^k \tilde Q$ for $k\geq 2$, using \eqref{kth R}-\eqref{eq: R6}, we
prove \eqref{f_B kth rule}.
\end{proof}

For two  matrices $A$ and $B$ in $S_0$, we denote
\[\<A,B\>:=\sum_{i,j}A_{ij}B_{ij}.\]

\begin{lemma}\label{lem gb} For any $Q\in S_\delta$, let $\tilde Q$ be defined in \eqref{Q13}. Recall that
\begin{align*}
	g_B(\tilde Q)=&a\tilde Q +b \big( \tilde Q  \tilde Q -\frac {1}3\tr(\tilde Q^2)I \big)-c\tilde Q\tr(\tilde Q ^2).
\end{align*}Then we have
 \begin{align*}
		\<\D g_B(\tilde Q) , (R^T(Q) \D Q R(Q))\>
		\leq -\frac{\lambda}{8}|\D (Q-\pi(Q))|^2+C|Q-\pi(Q)|^2|\D Q|^2
		.\alabel{gB pi}
	\end{align*}
Moreover, for any  $k\geq 2$,  we have
	\begin{align*}
		&\int_{\R^3}\<\D^k g_B(\tilde Q) ,\D^{k-1} (R^T(Q) \D Q R(Q))\>\,dx
		\\
\leq&-\frac{\lambda}{8}\int_{\R^3}|\D^k (Q-\pi(Q))|^2\,dx
\\
&+C\int_{\R^3}\sum_{\substack{\mu_1+\cdots+\mu_{k+1}=k\\\mu_1< k}}|\D^{\mu_1}(Q-\pi(Q))|^2|\D^{\mu_2} Q|^2\cdots |\D^{\mu_{k+1}} Q|^2\,dx
.\alabel{gB pi k}
	\end{align*}
\end{lemma}
\begin{proof}
It follows from the definition of $\tilde Q$ in \eqref{Q13} that
\begin{align*}
&R^T(Q) \D Q R(Q)=\D \tilde Q -\D R^T(Q)  Q R(Q)-R^T(Q)  Q \D R(Q)
\\
=&\D \tilde Q - A -\D R^T(\pi(Q))\pi(Q) R(\pi(Q))-R^T(\pi(Q))\pi(Q) \D R(\pi(Q)),\alabel{eq: R0}
\end{align*}
where $A$ is given by
	\begin{align*}
		A:=&\D  R^T(Q) [Q-\pi(Q)]R(Q)+R^T(Q)[Q-\pi(Q)]\D  R(Q)
		\\
		&+\D  [R^T(Q)-R^T(\pi(Q))]\pi(Q) R(Q)+R^T(Q)\pi(Q)\D [R(Q)-R(\pi(Q))]
		\\
		&+\D  R^T(\pi(Q))\pi(Q)[R(Q)-R(\pi(Q))]+[R^T(Q)-R^T(\pi(Q))]\pi(Q)\D  R(\pi(Q)).
	\end{align*}
Note that each term of $A$ contains the factor $(Q-\pi(Q))$. Let $\xi$  be a block-diagonal matrix  defined in \eqref{xi}. Then we have
	\begin{align*}
	&\<\xi,\D[R^T(\pi(Q))] \pi(Q) R(\pi(Q))+ R^T(\pi(Q)) \pi(Q) \D[R(\pi(Q))]\>
	\\
	=&\<\xi,\D[R^T(\pi(Q))] R(\pi(Q)) Q^++Q^+R^T(\pi(Q))\D[R(\pi(Q))]\>
	\\
	=&\sum_{i,j,k =1}^3\xi_{ij}\(\D[R^T(\pi(Q))] R(\pi(Q))\)_{ik} Q^+_{kj}+\xi_{ij}Q^+_{ik}\(R^T(\pi(Q))\D[R(\pi(Q))]\)_{kj}
	\\
	=&\sum_{i,j =1}^2\xi_{ij}Q^+_{ii}\Big(\D[R^T(\pi(Q))] R(\pi(Q))+R^T(\pi(Q))\D[R(\pi(Q))]\Big)_{ij}
	\\
	&+\xi_{33}Q^+_{33}\Big(\D[R^T(\pi(Q))] R(\pi(Q))+R^T(\pi(Q))\D[R(\pi(Q))]\Big)_{33}=0,
	\alabel{eq: R}
	\end{align*}
where we used that $Q^+_{11} = Q^+_{22}$ and $\big(\D R^TR+R^T\D R\big)_{ij}=0$ for each $i,j=1,2,3$. Note that $(R^T(Q) \D Q R(Q))$ is trace
free. It follows from \eqref{eq: R4}, \eqref{eq: R0} and \eqref{eq: R} with $\xi=\D(\p_{\tilde Q }f_B(\tilde Q))$ that
\begin{align*}
&\<\D g_B(\tilde Q) , (R^T(Q) \D Q R(Q))\>
=-\<\D(\p_{\tilde Q }f_B(\tilde Q)+\frac{b}{3}\tr(\tilde Q^2)I),(R^T(Q) \D Q R(Q))\>
\\
\leq&-\p^2_{\tilde Q_{ij}  \tilde Q_{kl}}  f_B(\tilde Q)\D  \tilde Q _{ij}\D\tilde Q_{kl}
+C|\D(\p_{\tilde Q }f_B(\tilde Q))||Q-\pi||\D Q|.\alabel{gB 0}
\end{align*}
It follows from \eqref{A2} and \eqref{eq: R4} that
\begin{align*}
|\D(\p_{\tilde Q }f_B(\tilde Q))|\leq C|\D \tilde Q|\leq&C|\D (Q-\pi(Q))|+C|\D Q||Q-\pi(Q)|
.\alabel{eq: R5}
\end{align*}
Using \eqref{gB 0}-\eqref{eq: R5}, Young's inequality and Corollary \ref{lem fB}, we obtain \eqref{gB pi}.

For the case of $k\geq 2$ in \eqref{gB pi k},  we obtain from \eqref{eq: R0} and integration by parts that
\begin{align*}
		&\int_{\R^3}\<\D^k g_B(\tilde Q) ,\D^{k-1} (R^T(Q) \D Q R(Q))\>\,dx
		\\
		=&(-1)^{k-1}\int_{\R^3}\<\D^{k-1}\D^k g_B(\tilde Q) ,(R^T(Q) \D Q R(Q))\>\,dx
		\\
		=&-\int_{\R^3}\<\D^k \p_{\tilde Q}f_B(\tilde Q) ,\D^{k-1}(\D \tilde Q-A)\>\,dx
		.\alabel{eq gb k}
		\end{align*}
		Here we used \eqref{eq: R} with $\xi = \D^{k-1}\D^k \p_{\tilde Q}f_B(\tilde Q)$.
In view of \eqref{kth R}-\eqref{eq: R6}, we also have that
\begin{align*}
|\D^k \tilde Q|^2\leq& C\sum_{\mu_1+\cdots+\mu_{k+1}=k}|\D^{\mu_1}(Q-\pi(Q))|^2|\D^{\mu_2} Q|^2\cdots |\D^{\mu_{k+1}} Q|^2\alabel{k Q}
\\
|\D^{k-1}A|^2\leq& C\sum_{\mu_1+\cdots+\mu_{k+1}=k-1}|\D^{\mu_1}(Q-\pi(Q))|^2|\D^{\mu_2}(\D Q)|^2\cdots |\D^{\mu_{k+1}} Q|^2\alabel{k A}.
\end{align*}
Using \eqref{eq gb k}-\eqref{k A}, Young's inequality and Corollary \ref{lem fB},  we obtain \eqref{gB pi k}.
\end{proof}

\subsection{Some a-priori estimates}For simplicity, we denote $f_E(Q_L,\D Q_L)$ by $f_E$ and omit the subscript $L$ in  all proofs in this section.
\begin{lemma}\label{Lie}
	Let $F$ be a $3\times 3$ matrix. For any symmetric $A,B$ matrices, we have
	\begin{align*}
		\<[A, F],B\> =\<F, [A,B]\>=-\<F^T, [A,B]\>
		.\alabel{energy p2.1}
	\end{align*}
\end{lemma}
\begin{proof}
	Note the following identity
	\begin{align*}
		\<[A,F],B \> =& \<(AF-FA),B\>=\tr\((AF)^TB-(FA)^TB\)
		\\
		=&\tr\(F^TA^TB-A^T(F^TB)\)=\tr\(F^TA^TB-(F^TB)A^T\)
		\\
		=&\<F,[A^T,B]\>=\<F,[A,B]\>.
	\end{align*}
	For the second identity in \eqref{energy p2.1}, we observe that
	\[\<F,[A,B]\>=\<F^T,[A,B]^T\>=-\<F^T,(A^TB^T-B^TA^T)\>=-\<F^T,[A,B]\>.\]
	Here we used the fact that  $A,B$ are symmetric in the last step.
\end{proof}

Now, we show  the following energy identity:
\begin{lemma}\label{lem energy}Let $(Q_L, v_L)$ be a strong solution to the system \eqref{RBE1}-\eqref{RBE3}  in $\R^ 3 \times  (0, T_L
	)$ with the initial condition $(Q_{L,0},v_{L,0})\in H^2_{Q_e}(\R^3;S_*)\times H^1(\R^3;\R^3)$ and $\div v_{L,0}=0$. Then, for any $s\in
(0, T_L)$, we have
	\begin{align*}
		&\int_{\R^3}\(f_E(Q_L,\D Q_L)+\frac1L f_B(Q_L)+\frac{|v_L|^2}{2}\)(\cdot,s)\,dx\\
		&+\int_0^s\int_{\R^3}\left|\MH(Q_L,\D Q_L)+\frac 1Lg_B(Q_L)\right|^2\,dxdt
		+ \int_0^s\int_{\R^3}|\D v_L|^2\,dxdt
		\\
		&=\int_{\R^3}\Big(f_E(Q_{L,0},\D Q_{L,0})+\frac1L f_B(Q_{L,0})+\frac{|v_{L,0}|^2}{2}\Big)\,dx.
		\alabel{energy eq}
	\end{align*}
\end{lemma}

\begin{proof}
	Taking $L^2$ inner product of \eqref{RBE1} with $v$ and using integration by part yield
	\begin{align*}
		&\frac{1}{2}\frac{d}{\,dt}\int_{\R^3}|v|^2\,dx+\int_{\R^3}|\D v|^2\,dx
		\\
		=&\int_{\R^3} \p_{p_{kl}^j}
		f_E\D_iQ_{kl} \D_j v_i\,dx-\int_{\R^3}[Q, \MH(Q ,\D Q )]_{ij}\D_j v_i\,dx.\alabel{kinetic energy}
	\end{align*}
	Next, multiplying \eqref{RBE3} with $(\MH(Q ,\D Q ) +\frac 1Lg_B(Q ))$ gives
	\begin{align*}
		&-\int_{\R^3} \<\p_tQ ,\MH(Q ,\D Q )+\frac 1Lg_B(Q )\>\,dx+\int_{\R^3}|\MH(Q ,\D Q )+\frac 1Lg_B(Q )|^2\,dx\\
		=&\int_{\R^3}\<(v\cdot \D) Q+[Q , \Omega ],\MH(Q ,\D Q )+\frac 1Lg_B(Q )\>\,dx.
		\alabel{energy p2}
	\end{align*}
	In view of \eqref{Mol} and the relation that
	\[\<\p_tQ,g_B(Q)\>=\<\p_tQ,-\p_{Q }f_B(Q)\>,\]we have
	\begin{align}\label{energy p2.00}
		&-\int_{\R^3} \<\p_tQ,\MH(Q,\D Q)+\frac 1Lg_B(Q)\>\,dx=\frac{d}{\,dt}\int_{\R^3} (f_E(Q,\D Q)+\frac1L f_B(Q))\,dx.
	\end{align}
	Utilizing  \eqref{Mol}-\eqref{gB}  and integrating  by parts, we   have
	\begin{align*}
		&\int_{\R^3} \<(v\cdot\D) Q,\MH(Q,\D Q)+\frac 1Lg_B(Q)\>\,dx
		\\
		=&\int_{\R^3}   \<(v\cdot\D)Q,  \D_j\big(\p_{p^j}
		f_E\big)- \p _{Q}f_E \> \,dx -\int_{\R^3} \<(v\cdot\D)Q,\frac1L\p_{Q} f_B(Q) \>\,dx
		\\
		=&-\int_{\R^3}\D_jv_i\D_iQ _{kl}\p_{p_{kl}^j}
		f_E + v_i\( \D^2_{ij}Q _{kl}\p_{p_{kl}^j}
		f_E-\D_iQ_{kl} \p _{Q_{kl}}f_E -\frac{1}{L}\D_i f_B(Q)\)\,dx
		\\
		=&-\int_{\R^3} \p_{p_{kl}^j}
		f_E\D_iQ_{kl} \D_j v_i\,dx-\int_{\R^3} v_i\D_i f\,dx=-\int_{\R^3}\p_{p^j_{kl}} f_E\D_iQ _{kl}\D_jv_i \,dx
		,\alabel{energy p2.0}
	\end{align*}
	where we have used $\tr (Q)=0$ in the second equality.
	
	Choosing $A=Q,B= \MH(Q,\D Q)+\frac 1Lg_B(Q), F=\D v$ in  Lemma \ref{Lie} and using the fact that $[Q,g_B]=0$, we have
	\begin{align*}
		& \<[Q , \Omega ],\MH(Q,\D Q)+\frac 1Lg_B(Q)\> = \D_jv_i[Q,\MH(Q,\D Q)]_{ij} .\alabel{energy p2.1 +}
	\end{align*}
	Integrating \eqref{energy p2.1 +} in $x$ and
	substituting \eqref{energy p2.00}-\eqref{energy p2.0} into \eqref{energy p2} give
	\begin{align*}
		&\frac{d}{\,dt}\int_{\R^3} f(Q,\D Q)\,dx+\int_{\R^3}|\MH(Q,\D Q)+\frac 1Lg_B(Q)|^2\,dx
		\\
		=&\int_{\R^3}\D_lv_k[Q,\MH(Q,\D Q)]_{kl}\,dx- \int_{\R^3} \p_{p_{kl}^j}
		f_E\D_iQ_{kl} \D_j v_i\,dx
		.\alabel{free energy}
	\end{align*}
	Therefore, the energy identity \eqref{energy eq} follows from taking the sum of  \eqref{kinetic energy} and \eqref{free energy} and
integrating over the
	time interval $[0,s]$.
\end{proof}

We rotate the equation \eqref{RBE3} by $R(Q_L)$; i.e.
\begin{align*}\alabel{ROT RBE}
	&R^T(Q_L)\big(\p_t Q_L+(v_L\cdot \D Q_L) +[Q_L, \Omega_L]\big)R(Q_L)
	\\
	=& R^T(Q_L)\MH(Q_L,\D Q_L) R(Q_L)+\frac 1Lg_B(\tilde Q_L),
\end{align*}
where we used the fact that $R(Q_L)^Tg_B(Q_L)R(Q_L)=g_B(\tilde Q_L)$.

The strong solutions also admit the following local energy inequality:
\begin{lemma}\label{lem 1ord}
Let $(Q_L, v_L)$ be a strong solution to the system \eqref{RBE1}-\eqref{RBE3}  in $\R^ 3 \times  (0, T_L
)$. Assume that $Q\in S_\delta$ for   sufficiently small $\delta$  on  $\R^3\times(0, T_L)$. Then, for any $\phi  \in  C^\infty_0  (\R^3 ) $
and $s \in  (0, T_L  )$, we have
	\begin{align*}
	&\int_{\R^3}\(|\D Q_L|^2+ |v_L |^2+\frac{|Q_L -\pi(Q_L )|^2}{L}\)(\cdot,s)\phi^2\,dx
	\\
	&+\int_0^{s }\int_{\R^3}\(|\D^2 Q_L|^2+|\D v_L|^2+|\p_t Q_L|^2+\frac{|\D(Q_L -\pi(Q_L ))|^2}{L}\)\phi^2\,dxdt
	\\
	\leq& C\int_{\R^3}\(|\D Q_{L,0}|^2+ |v_{L,0}|^2+\frac{|Q_{L,0}-\pi(Q_{L,0})|^2}{L}\)\phi^2\,dx
	\\
	&+C\int_0^{s }\int_{\R^3}|\D Q_L|^2\(|\D Q_L|^2+|v_L|^2+\frac{|Q_L-\pi(Q_L)|^2}{L}\)\phi^2\,dxdt
	\\
	&+C\int_0^{s }\int_{\R^3}|P_L-c_L^*(t)||v_L||\D \phi||\phi|+(|\D Q_L|^2+|v_L|^2)|\D\phi|^2\,dxdt
	,\alabel{1ord eq}
	\end{align*}
where $C$ is a constant independent of $L$ and  $c_L^*(t)$ is a   function in $t$ to be chosen later.
\end{lemma}
\begin{proof}
It follows from using \eqref{ROT RBE} that
	\begin{align*}
		&\int_{\R^3}\<\D\(R^T(Q)(\p_t Q+v\cdot \D  Q +[Q, \Omega])R(Q)\),R^T(Q)\D_\beta QR(Q)\>\phi^2\,dx
		\\
		=&\int_{\R^3}\<\D(R^T(Q) \MH(Q,\D Q)R(Q))+\frac 1L \D g_B(\tilde Q),R^T(Q)\D_\beta QR(Q)\>\phi^2\,dx
		.\alabel{Delta Q eq}
	\end{align*}
To estimate the first term on the right-hand side of \eqref{Delta Q eq}, we observe
\begin{align*}
	&\int_{\R^3}\<\D(R^T(Q) \MH(Q,\D Q)R(Q)),R^T(Q)\D_\beta QR(Q)\>\phi^2\,dx
	\\
	\leq &\int_{\R^3}\<\D \MH(Q,\D Q) , \D_\beta Q \>\phi^2\,dx+C\int_{\R^3}|\D R(Q)||\D Q| |\MH(Q,\D Q) |\phi^2\,dx
	\\
	\leq&\int_{\R^3}\D_\beta\D_k (\p_{p_{ij}^k} f_E)\D_\beta Q_{ij}\phi^2\,dx+C\int_{\R^3}|\p_Qf_E||\D Q||\D \phi|\phi\,dx
	\\
	&+C\int_{\R^3}( |\MH(Q,\D Q) ||\D Q|^2+| \p_{Q} f_E(Q,\D Q)||\D^2 Q|)\phi^2\,dx.\alabel{fE 1st}
\end{align*}
Using  the   condition  \eqref{sec2 f_E} on $f_E$   and integrating by parts, we have
\begin{align*}
	&\int_{\R^3}\D_\beta\D_k (\p_{p_{ij}^k} f_E)\D_\beta Q_{ij}\phi^2\,dx
	\\
	=&\int_{\R^3}\D_k\( \p^2_{p_{ij}^kp_{mn}^l} f_E \D^2_{\beta l}Q_{mn}+ \p^2_{p_{ij}^kQ_{mn}} f_E \D_\beta Q_{mn}\)\D_\beta
	Q_{ij}\phi^2\,dx
	\\
	\leq&-\int_{\R^3}\p^2_{p_{ij}^kp_{mn}^l} f_E \D^2_{\beta l}Q_{mn}\D^2_{k\beta}Q_{ij}\phi^2\,dx +C\int_{\R^3}|\p^2_{p p } f_E| |\D^2Q||\D
Q||\D\phi||\phi|\,dx
	\\
	&+C\int_{\R^3}\Big(|\D(\p^2_{p Q} f_E)| |\D Q|+|\p^2_{p Q} f_E| |\D^2 Q|\Big)|\D Q|\phi^2\,dx
	\\
	\leq& -\int_{\R^3}\frac{3\alpha}{8}|\D^2 Q|^2\phi^2+C\(|\D  Q|^4\phi^2+|\D Q|^2|\D \phi|^2\)\,dx
	,\alabel{fE second}
\end{align*}
where we used that $|\D (\p^2_{pQ} f_E)|\leq C(|\D^2 Q|+|\D Q|^2)$.

Recall from \eqref{gB pi} in Lemma \ref{lem gb} that
\begin{align*}
& \int_{\R^3}\<\frac 1L \D_\beta  g_B(\tilde Q) ,R^T(Q)\D_\beta QR(Q)\>\phi^2\,dx
\\
\leq&-\frac\lambda8\int_{\R^3} \frac{|\D(Q-\pi(Q))|^2}{L}\phi^2\,dx+C\int_{\R^3}|\D Q|^2\frac{|Q-\pi(Q)|^2}{L}\phi^2\,dx
.\alabel{1st gb}
\end{align*}
It follows from \eqref{fE 1st}-\eqref{1st gb} that
	\begin{align*}
	&\int_{\R^3}\<\D(R^T(Q)(\MH(Q,\D Q)R(Q))+\frac 1L \D g_B(\tilde Q),R^T(Q)\D_\beta QR(Q)\>\phi^2\,dx
	\\
	\leq&-\int_{\R^3} \(\frac{\alpha}{4} |\D^2 Q|^2+ \frac{\lambda }8\frac{|\D  (Q-\pi(Q))|^2}{L}\)\phi^2\,dx
	\\
	&+C\int_{\R^3}|\D Q|^2|\D\phi|^2+|\D Q|^2\(|\D Q|^2+\frac{|Q-\pi(Q)|^2}{L} \)\phi^2\,dx
	.\alabel{1ord p4.1}
	\end{align*}

Through integration by parts, we estimate the left-hand side of \eqref{1ord p4.1} by
	\begin{align*}
		&-\int_{\R^3}\<\D_\beta\(R^T(Q)(\p_t Q+v\cdot \D  Q +[Q, \Omega])R(Q)\),R^T(Q)\D_\beta QR(Q)\>\phi^2\,dx
		\\
	\leq& -\int_{\R^3}\<\D_\beta\p_tQ+(v\cdot \D) \D_\beta Q ,\D_\beta Q \>\phi^2\,dx+ \int_{\R^3}\<[Q, \Omega],\D_\beta(\D_\beta Q\phi^2)\>\,dx
	\\
	&+C\int_{\R^3}|\D Q|^2\(| \p_tQ| +|v|| \D  Q| +|\D v| \)\phi^2\,dx
	\\
	\leq&-\frac{1}{2}\frac{d}{\,dt}\int_{\R^3}|\D Q|^2\phi^2\,dx+\int_{\R^3}\Big(\frac{\alpha}{8}|\D^2
	Q|^2+\frac{1}{4}|\p_t
	Q|^2+C|\D v|^2\Big)\phi^2\,dx
	\\
	&+C\int_{\R^3}|\D Q|^2(|v|^2+|\D Q|^2)\phi^2+|\D Q|^2|\D\phi|^2\,dx
	.\alabel{1ord p4}
	\end{align*}
Adding \eqref{1ord p4.1} to \eqref{1ord p4}, we have
\begin{align*}
	&\frac{1}{2}\frac{d}{\,dt}\int_{\R^3}|\D Q|^2\phi^2\,dx+\int_{\R^3} \(\frac{\alpha}4 |\D^2 Q|^2+ \frac{\lambda }8\frac{|\D
(Q-\pi(Q))|^2}{L}\)\phi^2\,dx
	\\
	\leq&\int_{\R^3}\Big(\frac{1}{2}|\p_tQ|^2+C|\D v|^2\Big)\phi^2\,dx+C\int_{\R^3}|\D Q|^2|\D\phi|^2\,dx
	\\
	&+\int_{\R^3}|\D Q|^2\(|\D Q|^2+|v|^2+\frac{|Q-\pi(Q)|^2}{L} \)\phi^2\,dx
	.\alabel{1ord Delta Q}
\end{align*}
Multiplying \eqref{RBE3}  by $\p_t Q\phi^2$ and using \eqref{energy p2.00} in Lemma \ref{lem energy} yield
	\begin{align*}
	&\frac{d}{\,dt}\int_{\R^3}(f_E(Q,\D Q)+\frac1L f_B(Q))\phi^2\,dx+\int_{\R^3}|\p_tQ|^2\phi^2\,dx\\
	=&-2\int_{\R^3}\p_tQ _{ij}\p_{p^k_{ij}} f_E \D_k\phi\phi\,dx-\int_{\R^3}\<(v\cdot\D)
	Q +[Q,\Omega] ,\p_tQ\>\phi^2\,dx\\
	\leq& \int_{\R^3}\( \frac14|\p_tQ|^2+C|\D
	v|^2\)\phi^2\,dx+C\int_{\R^3}|\D Q|^2|v|^2\phi^2+|\D Q|^2|\D\phi|^2\,dx
	.\alabel{1ord Q_t}
	\end{align*}
Adding \eqref{1ord Delta Q} to \eqref{1ord Q_t}, integrating in $t$ and using \eqref{f_B and dist}, we see
	\begin{align*}
	&\int_{\R^3}\(|\D Q|^2+\frac{|Q-\pi(Q)|^2}{L}\)(\cdot,s)\phi^2\,dx
	\\
	&+\int^s_0\int_{\R^3} \(  |\D^2 Q|^2+ |\p_tQ|^2+   \frac{|\D  (Q-\pi(Q))|^2}{L}\)\phi^2\,dxdt
	\\
	\leq
	&C\int_{\R^3}\(|\D Q_0|^2+\frac{|Q_0-\pi(Q_0)|^2}{L}\)\phi^2\,dx
	\\
	&+C\int^s_0\int_{\R^3}|\D v|^2\phi^2\,dxdt+C\int^s_0\int_{\R^3}|\D Q|^2|\D\phi|^2\,dxdt
	\\
	&+\int^s_0\int_{\R^3}|\D Q|^2\(|\D Q|^2+|v|^2+\frac{|Q-\pi(Q)|^2}{L} \)\phi^2\,dxdt
	.\alabel{p3}
	\end{align*}
It remains to estimate the term $\D v$ on the right-hand side of \eqref{p3}.
We multiply  \eqref{RBE1} by $v\phi^2$ and \eqref{RBE3} by $\Big(\MH(Q,\D Q)+\frac{1}{L}g_B(Q)\Big)\phi^2$. Then    it follows from using the
same argument of \eqref{kinetic energy}-\eqref{free energy}  that
\begin{align*}
	&\frac{d}{\,dt}\int_{\R^3}\left(\frac{1}{2}|v|^2+f_E(Q,\D Q)+\frac1L f_B(Q)\right)\phi^2\,dx
	\\
	&+\int_{\R^3}\(|\MH(Q,\D Q)+\frac{1}{L}g_B(Q)|^2+|\D
	v|^2\)\phi^2\,dx
	\\
	=&\int_{\R^3}\(|v|^2+2\(P-c^\ast(t)\)\)v\cdot\D \phi\phi-2\D_k v_i v_i\D_k\phi\phi \,dx
	\\
	&-2\int_{\R^3}[Q, \MH(Q,\D Q)]_{ij}v_i\D_j
	\phi\phi\,dx+2\int_{\R^3}\p_{p_{kl}^j}
	f_E\D_i Q_{kl}v_i\D_j \phi\phi\,dx
	\\
	&-2\int_{\R^3}\p_tQ _{kl}\p_{p^j_{kl}} f_E
	\D_j\phi\phi\,dx-\int_{\R^3}v\cdot\D f\phi^2\,dx-2\int_{\R^3}v_i\D_i Q_{kl}\p_{p_{kl}^j}
	f_E\D_j\phi\phi\,dx \\
	\leq& \eta\int_{\R^3}(|\p_tQ|^2+ |\D^2Q|^2)\phi^2\,dx+\frac{1}{2}\int_{\R^3}|\D v|^2\phi^2\,dx+C\int_{\R^3}|\D Q|^2(|\D Q|^2+|v|^2)\phi^2\,dx
	\\
	&+C\int_{\R^3}(|\D Q|^2+|v|^2)|\D\phi|^2\,dx+C\int_{\R^3}|P_L-c_L^*(t)||v_L||\D \phi||\phi|\,dx
	\alabel{p1}.
\end{align*}
Integrating  \eqref{p1} in $t$, employing \eqref{p3} and   choosing sufficiently small $\eta$, we   obtain
\begin{align*}
	&\int_{\R^3} |v(\cdot,s)|^2\phi^2\,dx+\int_0^s\int_{\R^3} |\D
	v|^2\phi^2\,dxdt
	\\
	\leq&C\int_{\R^3}\(|\D Q_0|^2+|v_0|^2+\frac{|Q_0-\pi(Q_0)|^2}{L}\)\phi^2\,dx+C\int_{\R^3}|\D Q|^2(|\D Q|^2+|v|^2)\phi^2\,dx
	\\
	&+C\int_0^s\int_{\R^3}(|\D Q|^2+|v|^2)|\D\phi|^2\,dx+C\int_0^s\int_{\R^3}|P_L-c_L^*(t)||v_L||\D \phi||\phi|\,dx
	\alabel{v}.
\end{align*}
Applying \eqref{v} to \eqref{p3},  we prove \eqref{1ord eq}.
\end{proof}

Through Lemma \ref{lem gb}  and the equation \eqref{ROT RBE}, we obtain second order estimates on $(\D^2 Q_L, \D v_L)$  in the following:
\begin{lemma}\label{lem 2ord}
Let $(Q_L, v_L)$ be a strong solution to the system \eqref{RBE1}-\eqref{RBE3}  in $\R^ 3 \times  (0, T_L )$. Assume that $Q\in S_\delta$ for
sufficiently small $\delta$  on  $\R^3\times(0, T_L)$. Then for any $\phi  \in  C^\infty_0  (\R^3 ) $ and $s \in  (0, T_L  )$, we have the
following local estimate
	\begin{align*}
		&\int_{\R^3}\(|\D^2 Q_L|^2+ |\D v_L |^2+\frac{|\D  (Q_{L}-\pi(Q_{L}))|^2}{L}\)(\cdot,s)\phi^2\,dx
		\\
		&+\int_0^{s}\int_{\R^3}\(|\D^3 Q_L|^2+|\D^2 v_L|^2+|\D\p_t Q_L|^2+\frac{|\D^2  (Q_{L}-\pi(Q_{L}))|^2}{L}\)\phi^2\,dxdt
		\\
		\leq& C\int_{\R^3}\(|\D^2 Q_{L,0}|^2+ |\D v_{L,0}|^2+\frac{|\D  (Q_{L,0}-\pi(Q_{L,0}))|^2}{L}\)\phi^2\,dx
		\\
		& +C\int_{\R^3} \frac{|Q_{L,0}-\pi(Q_{L,0})|^2}{L}|\D Q_{L,0}|^2\phi^2+ \(\frac{|Q-\pi(Q)|^2}{L}|\D Q|^2\)(\cdot,s)\phi^2\,dx
		\\
		&+C\int_0^{s}\int_{\R^3}e(Q_L,v_L)\(|\D^2 Q_L|^2+|\D v_L|^2+|\p_t Q_L|^2\)\phi^2\,dxdt
		\\
		&+C\int_0^{s}\int_{\R^3}e(Q_L,v_L)\(\frac{|\D (Q_L-\pi(Q_L))|^2}{L}+e^2(Q_L,v_L)\)\phi^2\,dxdt
		\\
		&+C\int_0^{s}\int_{\R^3}\(e^2(Q_L,v_L)+|P_L-c_L^*(t)|^2\)
		(|\D \phi|^2+|\D^2 \phi||\phi|)\,dxdt
		\\
		&+C\int_0^{s}\int_{\R^3}\(|\D^2 Q_L|^2+|\D v_L|^2+|\p_t Q_L|^2\)(|\D \phi|^2+|\D^2 \phi||\phi|)\,dxdt
		,\alabel{2ord eq}
	\end{align*}
where $C$ is a constant independent of $L$ and  $c_L^*(t)$ is a   function in $t$ to be chosen later.  Here we denote
\[e(Q_L,v_L):=|\D Q_L|^2+|v_L|^2+\frac{|Q_L-\pi(Q_L)|^2}{L}.\]
\end{lemma}
\begin{proof}
Differentiating \eqref{ROT RBE} with respect to $x_\beta$ and $x_\gamma$,  we multiply the resulting expression by $\D_{\beta}(R^T(Q) \D_\gamma
QR(Q))\phi^2$  to obtain
	\begin{align*}
		&\int_{\R^3}\<\D^2_{\beta_\gamma} \Big(R^T(Q)\big(\p_t Q+v\cdot \D Q +[Q, \Omega]\big)R(Q)\Big), \D_{\beta}(R^T(Q) \D_\gamma QR(Q))
\>\phi^2\,dx
		\\
		=&\int_{\R^3}\<\D^2_{\beta_\gamma}\(R^T(Q)\MH(Q,\D Q) R(Q)+\frac1Lg_B(\tilde Q)\),\D_{\beta}(R^T(Q) \D_\gamma QR(Q))\>\phi^2\,dx
		.\alabel{D Delta Q eq}
	\end{align*}
  Integrating by parts twice and using \eqref{sec2 f_1}-\eqref{sec2 f_E}, we have
	\begin{align*}
		&\int_{\R^3}\<\D^2_{\beta\gamma}\D_k\( \p^2_{p^k} f_E\), \D^2_{\beta\gamma}  Q \>\phi^2\,dx
		\\
		=&-\int_{\R^3}\D_\gamma\(\p^2_{p_{ij}^kp_{mn}^l} f_E \D^2_{\beta l}Q_{mn}+\p^2_{p_{ij}^k  Q_{mn}} f_E \D_{\beta}Q_{mn}\)
\D_k(\D^2_{\beta\gamma}  Q_{ij}\phi^2)\,dx
		\\
		\leq&-\int_{\R^3}\D_\gamma\(\p^2_{p_{ij}^kp_{mn}^l} f_E \D^2_{\beta l}Q_{mn}\)\D^3_{\beta \gamma k} Q_{ij}\phi^2\,dx
		\\
		&+C\int_{\R^3}(|\D Q|^3+ |\D Q||\D^2 Q|)|\D^3 Q|\phi^2\,dx
		\\
		&+C\int_{\R^3}(| \D^2 Q||\D^3 Q |+|\D Q||\D^2 Q|^2+|\D Q|^3|\D^2Q|)|\D\phi||\phi|\,dx
		\\
		\leq&-\int_{\R^3}\frac{3\alpha}{8}|\D^3 Q|^2\phi^2\,dx+C\int_{\R^3}|\D Q|^2(|\D^2 Q|^2+|\D Q|^4)\phi^2+|\D^2 Q|^2|\D \phi|^2  \,dx
		.\alabel{fE third}
	\end{align*}
	Using \eqref{fE third} and integration by parts, we find
	\begin{align*}
		&\int_{\R^3}\<\D^2_{\beta\gamma}\(R^T(Q)\MH(Q,\D Q) R(Q)\),\D_{\beta}(R^T(Q) \D_\gamma QR(Q))\>\phi^2\,dx
		\\
		\leq&-\int_{\R^3}\<\D^2_{\beta\gamma}\p_{p^k} f_E, \D_k \D^2_{\beta\gamma}  Q \>\phi^2\,dx+C\int_{\R^3}|\D^2 \p_{p} f_E||\D^2  Q||\D\phi||\phi|\,dx
		\\
		&+C\int_{\R^3}|\D Q||\MH(Q,\D Q)|\Big(|\D^2(R^T(Q) \D QR(Q))|\phi^2+(|\D^2 Q|^2+|\D Q|^4)|\D \phi||\phi|\Big)\,dx
				\\
				&+|\D \MH(Q,\D Q)|(|\D Q||\D^2 Q|+|\D Q|^3)\phi^2\,dx
		\\
		\leq&-\int_{\R^3}\frac{\alpha}{4}|\D^3 Q|^2\phi^2\,dx+C\int_{\R^3}(|\D^2 Q|^2+|\D Q|^4)|\D Q|^2\phi^2+|\D^2 Q|^2|\D\phi|^2\,dx
		.\alabel{2ord p4}
	\end{align*}
Here we used that \[|\D^2(R^T(Q) \D QR(Q))|+|\D \MH(Q,\D Q)|\leq C(|\D^3 Q|+|\D^2 Q||\D Q|+|\D Q|^3).\]
It follows from \eqref{gB pi k} in Lemma \ref{lem gb} with $k=2$ and Young's inequality,   we obtain
	\begin{align*}
		& \int_{\R^3}\< \D^2_{\beta\gamma}\frac{g_B(\tilde Q)}L ,\D_{\beta}(R^T(Q) \D_\gamma
QR(Q))\>\phi^2\,dx\leq-\frac{\lambda}{8}\int_{\R^3} \frac{|\D^2(Q-\pi(Q))|^2}{L}\phi^2
		\,dx
		\\
		&+C\int_{\R^3}\frac{| \D (Q-\pi(Q))|^2}{L}|\D Q|^2\phi^2+\frac{|Q-\pi(Q)|^2}{L}(|\D^2 Q|^2+|\D Q|^4)\phi^2  \,dx
		.\alabel{2ord g_B D^2 Q}
	\end{align*}
We compute the left-hand side of \eqref{D Delta Q eq} to get
	\begin{align*}
		&-\int_{\R^3}\< \D^2_{\beta\gamma}\Big(R^T(Q)\big(\p_t Q+v\cdot \D Q +[Q, \Omega]\big)R(Q)\Big),\D_{\beta} (R^T(Q) \D_\gamma
QR(Q))\>\phi^2\,dx
		\\
		\leq&-\int_{\R^3} \<\D^2_{\beta\gamma}\p_t Q,\D^2_{\beta\gamma} Q\>\phi^2\,dx
		+C\int_{\R^3} |\D R||\p_t Q||\D(R^T(Q) \D QR(Q))||\D\phi||\phi|\,dx
		\\
		&+C\int_{\R^3}|\D\p_t Q|(|\D^2 Q||\D R(Q)|+|\D Q||\D^2 R(Q)|+|\D Q||\D R(Q)|^2)\phi^2\,dx
		\\
		&+C\int_{\R^3}\Big(|\p_t Q||\D Q|+|v|(|\D^2 Q|+|\D Q|^2)\Big)|\D^2(R^T(Q) \D QR(Q))|\phi^2\,dx
		\\
		&+C\int_{\R^3}\Big(|\D v||\D Q|+|\D^2 v|\Big)|\D^2(R^T(Q) \D QR(Q))|\phi^2+|\D\p_t Q||\D Q|^2|\D \phi||\phi|\,dx
		\\
		\leq&-\frac 12\frac{d}{dt}\int_{\R^3} |\D^2 Q|^2 \phi^2\,dx+\int_{\R^3}\(\frac14|\D \p_t Q|^2+\frac{\alpha}8|\D^3 Q|^2+C|\D^2
v|^2\)\phi^2\,dx
		\\
		&+C\int_{\R^3} (|\D^2 Q|^2+|\D v|^2+|\p_t Q|^2+|\D Q|^4)(|\D Q|^2+|v|^2)\phi^2\,dx
		\\
		&+C\int_{\R^3}(|\D^2 Q|^2+|\D Q|^4)|\D \phi|^2\,dx
		.\alabel{2ord p4.1}
	\end{align*}
In view of \eqref{D Delta Q eq}-\eqref{2ord p4.1}, we have
\begin{align*}
	&\frac12\frac{d}{dt}\int_{\R^3} |\D^2 Q|^2 \phi^2\,dx+ \int_{\R^3} \(\frac{\alpha}{8}|\D^3 Q|^2+\frac{\lambda}{8}\frac{|\D^2(
Q-\pi(Q))|^2}{L}\)\phi^2
	\,dx
	\\
	\leq&  \int_{\R^3}\(\frac14|\D \p_t Q|^2+C|\D^2 v|^2\)\phi^2\,dx+C \int_{\R^3}\(e^2(Q,v)+|\D^2 Q|^2\)|\D \phi|^2\,dx
	\\
	&+C \int_{\R^3}e(Q,v)\(|\D^2 Q|^2+|\D v|^2+|\p_t Q|^2+|\D Q|^4+\frac{|\D (Q-\pi(Q))|^2}{L}\)\phi^2\,dx.
	\alabel{2nd Delta Q}
\end{align*}
We differentiate \eqref{ROT RBE} in   $x_\beta$ and multiply it by
$\D_\beta(R^T(Q)\p_t QR(Q))\phi^2$  to obtain
\begin{align*}
	&\int_0^s\int_{\R^3}\<\D_\beta\( R^T(Q)(\p_t Q +v\cdot \D  Q +[Q, \Omega])R(Q)\), \D_\beta (R^T(Q)\p_t  QR(Q))\>\phi^2\,dxdt
	\\
	=&\int_0^s\int_{\R^3}\<\D_\beta\((R^T(Q)\MH(Q,\D Q) R(Q))+\frac1Lg_B(\tilde Q)\),\D_\beta(R^T(Q)\p_t QR(Q))\>\phi^2\,dxdt
	.\alabel{2nd pt Q}
\end{align*}
Using \eqref{sec2 f_1}-\eqref{sec2 f_E}, we derive
\begin{align*}
	&\int_0^s\int_{\R^3}\<\D_\beta(R^T(Q)\MH(Q,\D Q) R(Q)),\D_\beta(R^T(Q)\p_t QR(Q))\>\phi^2\,dx dt
	\\
	\leq&\int_0^s\int_{\R^3}\<\D_\beta(\D_k\p_{p^k}f_E-\p_Qf_E),\D_\beta\p_t Q\>\phi^2\,dx dt
	\\
	&+\int_0^s\int_{\R^3}(|\D\MH(Q,\D Q)||\D R(Q)||\p_t Q|+|\D R(Q)||\MH(Q,\D Q)||\D(R^T(Q) \p_t QR(Q))|)\phi^2\,dx dt
	\\
	\leq& - \frac12\int_0^s\int_{\R^3}\p_t\Big(\p^2_{p_{ij}^kp_{mn}^l} f_E \D^2_{\beta
		l}Q_{mn}\D^2_{\beta k} Q_{ij}-\p_t\p^2_{p_{ij}^kp_{mn}^l} f_E\D^2_{\beta l}Q_{mn}\D^2_{\beta k} Q_{ij}\Big)\phi^2\,dxdt
	\\
	&+C\int_0^s\int_{\R^3}(|\D^2 Q|+|\D Q|^2)|\D^2 Q||\D
\p_t Q||\D\phi||\phi|\,dxdt
	\\
	&+C\int_0^s\int_{\R^3}|\D \MH(Q,\D Q)||\D Q||\p_tQ|\phi^2+(|\D^2 Q|+|\D Q|^2)|\D \p_t Q|\phi^2\,dxdt
\\
	&+C\int_0^s\int_{\R^3}|\MH(Q,\D Q)||\D Q|(|\D \p_t Q|+|\D Q||\p_t Q|)\phi^2\,dxdt
	\\
	\leq& \int_{\R^3}-\frac {\alpha}4|\D^2  Q(\cdot,s)|^2 \phi^2+C|\D^2  Q_{0}|^2 \phi^2\,dx+\int_0^s\int_{\R^3}(\frac{\alpha}{16}|\D^3 Q|^2+\frac
18|\D\p_t Q|^2)\phi^2\,dxdt
	\\
	&+C\int_0^s\int_{\R^3}|\D Q|^2(|\D^2 Q|^2+|\p_t Q|^2+|\D Q|^4)\phi^2+(|\D^2 Q|^2+|\D Q|^4)|\D \phi|^2\,dxdt
	.\alabel{2ord H Q_t}
\end{align*}
Here we have used that
\begin{align*}
|\D(R^T(Q)\p_t QR(Q))|\leq& C(|\D R(Q)||\p_t Q||R(Q)|+|\D\p_t Q||R(Q)|^2)
\\
\leq& C(|\D\p_t Q|+|\D Q||\p_t Q|).
\end{align*}
Using integration by parts and \eqref{eq: R5}, we have
\begin{align*}
	&\int_{\R^3}\frac1{2L}\p_t \p^2_{\tilde  Q_{ij}\tilde Q_{kl}} f_B(\tilde Q)\D  \tilde  Q_{kl}\D \tilde Q_{ij}\phi^2\,dx
	\\
	\leq&\frac C{L}\int_{\R^3}|Q-\pi(Q)|\(|\D Q||\p_t Q||\D \tilde Q|+ |\D\p_t Q||\D \tilde Q|+|\p_t Q||\D^2 \tilde Q|\)\phi^2\,dx
	\\
	&+\frac C{L}\int_{\R^3}|Q-\pi(Q)|\(|\p_t Q||\D \tilde Q|\)|\D\phi||\phi|\,dx
	\\
	\leq&\int_{\R^3}\(\frac14|\D\p_t Q|^2+\frac\lambda {16}\frac{|\D^2 (Q-\pi(Q))|^2}{L}\)\phi^2+C|\p_t Q|^2|\D Q|^2+C|\p_t Q|^2|\D\phi|^2\,dx
	\\
	&+C\int_{\R^3}\frac{|\D (Q-\pi(Q))|^2}{L}\(|\D Q|^2+\frac{|Q-\pi(Q)|^2}{L}\)\phi^2\,dx
	\\
	&+C\int_{\R^3}\frac{|Q-\pi(Q)|^2}{L}\(\frac{|Q-\pi(Q)|^2}{L}|\D Q|^2+|\D^2 Q|^2+|\D Q|^4+|\p_t Q|^2\)\phi^2\,dx
	,\alabel{2nd pt fB}
\end{align*}
where in the last inequality, we have used \eqref{k Q} with $k=2$ that
\begin{align*}
	|\D^2\tilde Q|^2 \leq & C(|\D^2(Q-\pi(Q))|^2+|\D(Q-\pi(Q))|^2|\D Q|^2
	+|Q-\pi(Q)|^2(|\D^2 Q|^2+|\D Q|^4)).
\end{align*}
Replacing $\D$ by $\p_t$ in \eqref{eq: R4}, \eqref{eq: R0}-\eqref{eq: R} and choosing $\xi=\D_\beta (\D_\beta g_B(\tilde Q)\phi^2)$, we have
\begin{align*}
	A:=&\p_t  R^T(Q) [Q-\pi(Q)]R(Q)+R^T(Q)[Q-\pi(Q)]\p_t  R(Q)
	\\
	&+\p_t  [R^T(Q)-R^T(\pi(Q))]\pi(Q) R(Q)+R^T(Q)\pi(Q)\p_t [R(Q)-R(\pi(Q))]
	\\
	&+\p_t  R^T(\pi(Q))\pi(Q)[R(Q)-R(\pi(Q))]+[R^T(Q)-R^T(\pi(Q))]\pi(Q)\p_t  R(\pi(Q)).
\end{align*}
and
\begin{align*}
	&\int_0^s\int_{\R^3}\<\frac{\D_\beta g_B(\tilde Q)}{L},\D_\beta(R^T(Q)\p_t QR(Q))\>\phi^2\,dxdt
	\\
	=&\int_0^s\int_{\R^3}\<\frac{ \D_\beta g_B(\tilde Q)}{L},\p_t \D_\beta\tilde Q\>\phi^2\,dxdt
	\\
	&+C\int_0^s\int_{\R^3}\frac{|\D^2 g_B(\tilde Q)|\phi^2+|\D g_B(\tilde Q)|\D\phi||\phi|}{L}|A|\,dxdt
	\\
	=&-\int_0^s\int_{\R^3}\frac1{2L}\p_t\(\p^2_{\tilde  Q_{ij}\tilde Q_{kl}}  f_B(\tilde Q)\D_\beta \tilde Q_{ij}\D_\beta \tilde
Q_{kl}\)\phi^2\,dxdt
	\\
	&+\int_0^s\int_{\R^3}\frac1{2L}\p_t \p^2_{\tilde  Q_{ij}\tilde Q_{kl}} f_B(\tilde Q)\D_\beta \tilde  Q_{kl}\D_\beta \tilde
Q_{ij}\phi^2\,dxdt
	\\
	&+C\int_0^s\int_{\R^3}\frac{|\D^2 g_B(\tilde Q)|\phi^2+|\D g_B(\tilde
Q)||\D\phi||\phi|}{L^\frac12}|\p_tQ|\frac{|Q-\pi(Q)|}{L^\frac12}\,dxdt
	\\
	\leq&\int_{\R^3}  -\frac \lambda8\frac{|\D (Q-\pi(Q))|^2(\cdot,s) }{L}\phi^2+C\frac{|\D ( Q_0-\pi(Q_0))|^2}{L} \phi^2\,dx
	\\
	&+C\int_{\R^3} \frac{|Q_{L,0}-\pi(Q_{L,0})|^2}{L}|\D Q_{L,0}|^2\phi^2+\frac{|Q-\pi(Q)|^2}{L}|\D Q|^2(\cdot,s)\phi^2\,dx
	\\
	&+ \int_0^s\int_{\R^3}\(\frac14|\D\p_t Q|^2+\frac\lambda {16}\frac{|\D^2(Q-\pi(Q))|^2}{L}\)\phi^2\,dxdt
	\\
	&+C\int_0^s\int_{\R^3}e(Q,v)\(|\D^2 Q|^2+|\p_t Q|^2+\frac{|\D( Q-\pi(Q))|^2}{L}+e^2(Q,v)\)\phi^2\,dxdt
	.\alabel{2ord g_B Q_t}
\end{align*}
Applying Young's inequality to the left-hand side of \eqref{2nd pt Q}, we obtain
\begin{align*}
	&-\int_{\R^3}\<\D_\beta\( R^T(Q)(\p_t Q +v\cdot \D  Q +[Q, \Omega])R(Q)\), \D_\beta (R^T(Q)\p_t  QR(Q))\>\phi^2\,dx
	\\
	\leq&- \int_{\R^3}\frac34|\D\p_t Q|^2\phi^2+C |\D^2 v|^2\phi^2-C\(|\D Q|^2(|\p_t Q|^2+|\D v|^2)+|v|^2|\D^2 Q|^2\)\phi^2  \,dx
	.\alabel{lem 2nd order p3}
\end{align*}
Substituting \eqref{2ord H Q_t}, \eqref{2ord g_B Q_t} and \eqref{lem 2nd order p3} into \eqref{2nd pt Q} yields
\begin{align*}
	&\int_{\R^3} \( \frac {\alpha}4|\D^2 Q |^2+\frac \lambda{16} \frac{|\D (Q -\pi(Q ))|^2 }{L}\)(\cdot,s) \phi^2\,dx
	+\frac 38\int^s_0\int_{\R^3} |\D\p_t Q|^2\phi^2\,dxdt
	\\
	\leq&C\int_{\R^3} \(|\D^2  Q_{0}|^2+\frac{|\D ( Q_0-\pi(Q_0))|^2}{L}+\frac{|Q_{L,0}-\pi(Q_{L,0})|^2}{L}|\D Q_{L,0}|^2\)\phi^2\,dx
	\\
	&+C\int_{\R^3} \(\frac{|Q-\pi(Q)|^2}{L}|\D Q|^2\)(\cdot,s)\phi^2\,dx+C\int_0^s\int_{\R^3} |\D^2 v|^2\phi^2\,dxdt
	\\
	&+\int_0^s\int_{\R^3}\(\frac{\alpha}{16}|\D^3 Q|^2+\frac\lambda 8\frac{|\D^2(Q-\pi(Q))|^2}{L}\)\phi^2\,dxdt
	\\
	&+C \int_0^s\int_{\R^3}\(|\D^2 Q|^2+|\D v|^2+|\p_t Q|^2\)(e(Q,v)\phi^2+|\D \phi|^2)\,dxdt
	\\
	&+C \int_0^s\int_{\R^3}\(\frac{|\D( Q-\pi(Q))|^2}{L}+e^2(Q,v)\)(e(Q,v)\phi^2+|\D \phi|^2)\,dxdt
	.\alabel{2nd Qt}
\end{align*}
Integrating \eqref{2nd Delta Q} in $t$ then adding it to  \eqref{2nd Qt}, we derive
	\begin{align*}
		&\int_{\R^3} \( |\D^2 Q|^2+\frac{|\D (Q-\pi(Q))|^2 }{L}\)(\cdot,s) \phi^2\,dx
		\\
		&+\int^s_0\int_{\R^3} \(|\D^3 Q|^2+|\D\p_t Q|^2+\frac{|\D^2 (Q-\pi(Q))|^2}{L}\)\phi^2\,dxdt
		\\
		\leq&C\int_{\R^3} \(|\D^2  Q_{0}|^2+\frac{|\D ( Q_0-\pi(Q_0))|^2}{L}+\frac{|Q_{L,0}-\pi(Q_{L,0})|^2}{L}|\D Q_{L,0}|^2\)\phi^2\,dx
		\\
		&+C\int_{\R^3} \(\frac{|Q-\pi(Q)|^2}{L}|\D Q|^2\)(\cdot,s)\phi^2\,dx+\int_0^s\int_{\R^3} C|\D^2 v|^2\phi^2\,dxdt
		\\
		&+C \int_0^s\int_{\R^3}\(|\D^2 Q|^2+|\D v|^2+|\p_t Q|^2\)(e(Q,v)\phi^2+|\D \phi|^2)\,dxdt
		\\
		&+C \int_0^s\int_{\R^3}\(\frac{|\D( Q-\pi(Q))|^2}{L}+e^2(Q,v)\)(e(Q,v)\phi^2+|\D \phi|^2)\,dxdt
		.\alabel{2ord p5}
	\end{align*}
To estimate the term $\D^2 v$ in \eqref{2ord p5}, we take $L^2$ inner product of \eqref{RBE1} with $-\Delta v\phi^2$ and calculate
\begin{align*}
	&\frac{1}{2}\frac{d}{\,dt}\int_{\R^3}|\D v|^2\phi^2\,dx+\int_{\R^3}|\D^2 v|^2\phi^2 \,dx
	\\
	=&-\int_{\R^3}2\p_t v_i \D_j v_i\D_j \phi\phi\,dx+\int_{\R^3}2(\D_j v_i\Delta v_i-\D_k v_i\D_{kj}^2v_i)\D_j \phi\phi\,dx
	\\
	&-\int_{\R^3}2(P-c^\ast)\Delta v_i\D_i\phi \phi+\(\D_j\sigma_{ij}+\D_j[Q,\MH(Q,\D
Q)]_{ij}-v\cdot \D v^i\)\Delta v_i\phi^2\,dx
	\\
	\leq& -\int_{\R^3}\D_j[Q,\MH(Q,\D Q)]_{ij}\Delta v_i \phi^2\,dx-\int_{\R^3}2\p_t v_i \D_j v_i\D_j \phi\phi\,dx
	\\
	&+\frac{1}{4}\int_{\R^3}|\D^2
v|^2\phi^2 \,dx+C\int_{\R^3}(|\D v|^2+|P-c^\ast|^2)|\D\phi|^2\,dx
	\\
	&+C\int_{\R^3}(|\D^2Q|^2+|\D Q|^4)|\D Q|^2\phi^2+|v|^2|\D v|^2\phi^2\,dx,
\end{align*}
where we have used the fact that $|\D\sigma(Q,\D Q)|\leq C(|\D^2Q|+|\D Q|^2)|\D Q|$.

By using \eqref{RBE1} and integrating by parts, we have
\begin{align*}
	&-2\int_{\R^3}\p_t v_i \D_j v_i\D_j \phi\phi\,dx
	\\
	=&2\int_{\R^3}(v_k\D_k v_i-\Delta v_i+\D_k\sigma_{ik})\D_j v_i\D_j\phi\phi\,dx+2\int_{\R^3}(P-c^\ast)\D_jv_i\D_i(\D_j\phi\phi)\,dx
	\\
	&+2\int_{\R^3}[Q,\MH(Q,\D Q)]_{ik}\Big(\D_{kj}^2v_i\D_j\phi\phi+\D_jv_i\D_k(\D_j\phi\phi)\Big)\,dx
	\\
	\leq& \frac{1}{4}\int_{\R^3}|\D^2 v|^2\phi^2
	\,dx+C\int_{\R^3}(|v|^2|\D v|^2+(|\D^2Q|^2+|\D Q|^4)|\D Q|^2)\phi^2\,dx
	\\
	&+C\int_{\R^3}(|P-c^\ast|^2+|\D v|^2+|\D^2 Q|^2+|\D Q|^4)(|\D^2\phi||\phi|+|\D\phi|^2)\,dx,
\end{align*}
which, plugging into the previous inequality, yields
\begin{align*}
	&\frac{1}{2}\frac{d}{\,dt}\int_{\R^3}|\D v|^2\,dx+\frac 12\int_{\R^3}|\D^2 v|^2\phi^2 \,dx\\
	\leq&-\int_{\R^3}\D_j[Q,\MH(Q,\D Q)]_{ij}\Delta v_i \phi^2\,dx+C\int_{\R^3}|P-c^\ast|^2(|\D^2\phi||\phi|+|\D\phi|^2)\,dx\\
	&+C\int_{\R^3}(|\D^2 Q|^2+|\D v|^2+|\D Q|^4)(|\D^2\phi||\phi|+|\D\phi|^2)\,dx\\
	&+C\int_{\R^3}(|\D Q|^2+|v|^2)(|\D^2 Q|^2+|\D v|^2 +|\D Q|^4)\phi^2\,dx
	.\alabel{2ord v}
\end{align*}
Choosing $A=Q, B=\MH(Q,\D Q)+\frac 1Lg_B(Q), F=\Delta\D v$ in  Lemma \ref{Lie}, we observe
\begin{align*}
	\<[Q , \Delta\Omega ],\MH+\frac 1Lg_B\>= \Delta\D_jv_i[Q,\MH]_{ij}
	.\alabel{Lie 1}
\end{align*}
Then, integrating by parts and using \eqref{Lie 1} and \eqref{RBE3} on the term $(\MH(Q,\D Q)+\frac 1Lg_B(Q))$, we have
\begin{align*}
	&\int_{\R^3}\<\D_\beta\(R^T[Q, \Omega]R\),\D_\beta\(R^T(Q)(\MH(Q,\D Q)+\frac 1Lg_B(Q))R(Q)\)\>\phi^2\,dx
	\\
	\leq& -\int_{\R^3}\<[\Delta Q , \Omega ]+2[\D_\beta Q , \D_\beta\Omega ]+[ Q , \Delta\Omega ],\MH(Q,\D Q) +\frac 1Lg_B(Q)\>\phi^2\,dx
	\\
	&+C\int_{\R^3}|\D[Q,\Omega]|\Big|\MH(Q,\D Q) +\frac 1Lg_B(Q)\Big|(|\D Q|\phi^2+|\D \phi||\phi|)\,dx
	\\
	&+C\int_{\R^3}|\D Q||\D v|\Big|\D\(R^T(Q)\big(\MH(Q,\D Q) +\frac 1Lg_B(Q)\big)R(Q)\)\Big| \phi^2\,dx
	\\
	\leq& \int_{\R^3}\Big(\Delta v_i\D_j[ Q ,\MH(Q,\D Q) ]_{ij}+\eta|\D\p_t Q|^2+\frac 14|\D^2 v|^2\Big) \phi^2\,dx
	\\
	&+C\int_{\R^3} (|\D^2 Q|^2+|\D v|^2+|\p_t Q|^2+|\D Q|^4+|v|^4)((|\D Q|^2+|v|^2)\phi^2+|\D \phi|^2)\,dx
	\\
	&+\frac 14\int_{\R^3}\Big|\D\(R^T(Q)\big(\MH(Q,\D Q) +\frac 1Lg_B(Q)\big)R(Q)\)\Big|^2 \phi^2\,dx
	.
	\alabel{2ord Lie}
\end{align*}
Here we used that $|[Q,\MH(Q,\D Q)]|^2|\D\phi|^2\leq C(|\D^2 Q|^2+|\D Q|^4)|\D\phi|^2$.

We differentiate \eqref{ROT RBE} with respect to $x_\beta$, multiply the resulting expression by $\D_\beta \(R^T(Q)\big(\MH(Q,\D Q) +\frac 1Lg_B(Q)\big)R(Q)\)$ and
substitute \eqref{2ord Lie} to find
\begin{align*}
	&\int_{\R^3}\left|\D \(R^T(Q)\big(\MH(Q,\D Q) +\frac 1Lg_B(Q)\big)R(Q)\)\right|^2\phi^2\,dx
	\\
	\leq&  \int_{\R^3}\<\D_\beta\( R^T(Q)\p_t QR(Q)\), \D_\beta \(R^T(Q)(\MH(Q,\D Q)+\frac 1Lg_B(Q))R(Q)\)\>\phi^2\,dx
	\\
	 &+\int_{\R^3}\Big(\Delta v_i\D_j[ Q ,\MH(Q,\D Q) ]_{ij}+\eta|\D\p_t Q|^2+\frac 14|\D^2 v|^2\Big) \phi^2\,dx
	\\
	&+C\int_{\R^3} (|\D^2 Q|^2+|\D v|^2+|\p_t Q|^2+|\D Q|^4+|v|^4)((|\D Q|^2+|v|^2)\phi^2+|\D \phi|^2)\,dx
	\\
	&+\frac14\int_{\R^3}\left|\D \(R^T(Q)\big(\MH(Q,\D Q) +\frac 1Lg_B(Q)\big)R(Q)\)\right|^2\phi^2\,dx
	.\alabel{2ord can}
\end{align*}
Combining \eqref{2ord v} with \eqref{2ord can} and integrating  in $t$, it follows from the arguments in \eqref{2ord H Q_t}-\eqref{2ord g_B Q_t} that
\begin{align*}
	&\int_{\R^3}\(\frac {\alpha}4|\D^2  Q |^2+\frac 12 |\D v |^2+\frac \lambda8 \frac{|\D ( Q-\pi(Q)|^2}{L}\) (\cdot,s)\phi^2\,dx
	+\frac14\int_0^s\int_{\R^3} |\D^2 v|^2 \phi^2 \,dxdt
	\\
	\leq& C\int_{\R^3}\(|\D^2  Q_{0}|^2+|\D v_0|^2+\frac{|\D ( Q_0-\pi(Q_0))|^2}{L}\)\phi^2\,dx
	\\
	&+C\int_{\R^3} \(\frac{|Q_{L,0}-\pi(Q_{L,0})|^2}{L}|\D Q_{L,0}|^2+\frac{|Q-\pi(Q)|^2}{L}|\D Q|^2(\cdot,s)\)\phi^2\,dx
	\\
	&+\int_{\R^3}2\eta\Big(|\D^3 Q|^2+|\D\p_t Q|^2+\frac{|\D^2 (Q-\pi(Q))|^2}{L}\Big) \phi^2\,dx
	\\
	&+C\int_{\R^3} (|\D^2 Q|^2+|\D v|^2+|\p_t Q|^2)(e(Q,v)\phi^2+|\D \phi|^2)\,dxdt
	\\
	&+C\int_0^s\int_{\R^3}\(\frac{|\D( Q-\pi(Q))|^2}{L}+e^2(Q,v)\)(e(Q,v)\phi^2+|\D \phi|^2)\,dxdt
	.\alabel{2ord p2.2}
\end{align*}
Combining  \eqref{2ord p5} with \eqref{2ord p2.2} and choosing suitable $\eta$, we prove  \eqref{2ord eq}.
\end{proof}

\section{Proof of Theorem  \ref{thm 2}}
In this section, we prove Theorem  \ref{thm 2}. At first,
we derive   a local estimate on the pressure $P_L(x,t)$.
\begin{lemma}\label{lem pressure estimate}
Let $(Q_L, v_L)$ be a strong solution to \eqref{RBE1}-\eqref{RBE3} in $\R^3\times(T_0, T_L)$. Assume that $Q\in S_\delta$ with  sufficiently
small $\delta$  on  $\R^3\times(0, T_L)$ and
\begin{equation}\label{pressure est eq}
	\sup_{T_0\leq t\leq T_L,x_0\in\R^3}\int_{B_R(x_0)}\(|\D Q_L|^3+|v_L |^3\)(\cdot,t)\,dx\leq \varepsilon_0^3.
\end{equation}
Then for any $t\in(T_0,T_L)$, there exists a constant  $c^*_L(t)\in \R$ such that the pressure $P_L$ satisfies the following estimate
\begin{align*}
	&\sup_{x_0\in\R^3}\int_{T_0}^{T_L}\int_{B_{2R}(x_0)}|P_L-c_L^\ast|^2\,dxdt
	\\
	\leq& C\sup_{y\in\R^3}\int_{T_0}^{T_L}\int_{B_{R}(y)}(|\D^2 Q_L|^2+|\D v_L|^2)+\frac{\varepsilon_0^2}{R^2}(|\D
Q_L|^2+|v_L|^2)\,dxdt.\alabel{pressure estimate.1}
\end{align*}
\end{lemma}
\begin{proof} The proof is essentially the same as the proof of Lemma 2.4 in \cite{FHM}. For completeness, we outline an approach here. Let $\phi$ be a
cut-off function satisfying
	$0\leq \phi\leq 1$, $\text{supp } \phi\subset B_{2R}(x_0)$ for some $x_0\in \R^3$ and $|\D \phi|\leq\frac{C}{R}$. Note that the pressure
$P_L$ satisfies
\begin{equation*}
	-\Delta P_L=\D_{ij}^2\([Q_L,H(Q_L,\D Q_L)]_{ij}-\sigma_{ij}(Q_L,\D Q_L) +v^i_Lv^j_L\)\quad\text{on}~~\R^3\times[T_0,T_L],
\end{equation*}
which implies $P_L=\mathcal{R}_i\mathcal{R}_j(F^{ij}),$ and \[|F^{ij}|=|[Q_L,H(Q_L,\D Q_L)]_{ij}-\sigma_{ij}(Q_L,\D Q_L) +v^i_Lv^j_L|\leq
C(|\D^2 Q_L|+|\D Q_L|^2+|v_L|^2),\] where $\mathcal{R}_i$ is the $i$-th Riesz transform on $\R^3$. Then  we have
\begin{equation}\label{pressure estimateP1.1}
	(P_L-c_L^\ast)\phi=\mathcal{R}_i\mathcal{R}_j(F^{ij}\phi)+[\phi,\mathcal{R}_i\mathcal{R}_j](F^{ij})-c_L^\ast\phi
\end{equation}
for a cut-off function $\phi$, where the commutator $[\phi,\mathcal{R}_i\mathcal{R}_j]$ is defined by
\[[\phi,\mathcal{R}_i\mathcal{R}_j](\cdot)=\phi\mathcal{R}_i\mathcal{R}_j(\cdot)-\mathcal{R}_i\mathcal{R}_j(\cdot\,\phi).\]
By using the Riesz operator maps $L^q$ into $L^q$
spaces for any $1<q<+\infty$ and the assumption \eqref{pressure est eq}, we have
\begin{align*}
	&\int_{T_0}^{T_L}\int_{\R^3}|\mathcal{R}_i\mathcal{R}_j(F^{ij}\phi)|^2\,dxdt\\
	\leq& C \int_{T_0}^{T_L}\int_{B_{2R(x_0)}}|\D^2 Q_L|^2+|\D v_L|^2dx\,dt+\frac{C}{R^2}\int_{T_0}^{T_L}\int_{B_{2R(x_0)}}|\D
	Q_L|^2+|v_L|^2\,dxdt.\alabel{pressure estimateP1.2}
\end{align*}
Since $\text{supp}\, \phi\subset B_{2R(x_0)}$, the commutator can be expressed as
\begin{align*}
	&[\phi,\mathcal{R}_i\mathcal{R}_j](F^{ij})(x,t)-c_L^\ast(t)\phi(x)\\
	=&\int_{B_{4R}(x_0)}\frac{(\phi(x)-\phi(y))(x_i-y_i)(x_j-y_j)}{|x-y|^5}F^{ij}(y,t)\,dy\\
	&\quad+\phi(x)\left[\int_{\R^3\backslash B_{4R}(x_0)}\frac{(x_i-y_i)(x_j-y_j)}{|x-y|^5}F^{ij}(y,t)\,dy-c_L(t)\right]\\
	=&:f_1(x,t)+f_2(x,t).\alabel{pressure estimateP2.1}
\end{align*}
By using the Hardy-Littlewood-Sobolev inequality  (cf. \cite{HM})
\[\left\|  \int_{\R^n}\frac{f(y)}{|x-y|^{n-\alpha}}\,dy\right\|_{L^q(\R^n)}\leq
C\|f\|_{L^r(\R^n)},\quad\frac{1}{q}=\frac{1}{r}-\frac{\alpha}{n}\] and the H\"older inequality,   a standard covering argument yields
\begin{align*}
	&\int_{T_0}^{T_L}\int_{\R^3}|f_1(\cdot,s)|^2\,dxdt\leq CR^{-2}\int_{T_0}^{T_L}\|(F^{ij})\chi_{B_{4R(x_0)}}\|_{L^\frac{6}{5}(\R^3)}^2\,dt
	\\
	\leq& \frac C{R^2}\int_{T_0}^{T_L}\|(|\D Q_L|+|v_L|)\chi_{B_{4R}(x_0)}\|_{L^3(\R^3)}^2\|(|\D
	Q_L|+|v_L|)\chi_{B_{4R}(x_0)}\|_{L^2(\R^3)}^2\,dt
	\\
	&+ \frac C{R^2}\int_{T_0}^{T_L}\|\chi_{B_{4R}(x_0)}\|^2_{L^3(\R^3)}\|(|\D^2 Q_L|)\chi_{B_{4R}(x_0)}\|_{L^2(\R^3)}^2\,dt
	\\
	\leq&\frac{C\varepsilon_0^2}{R^2}\int_{T_0}^{T_L}\int_{B_{4R}(x_0)}|\D Q_L|^2+|v_L|^2\,dxdt+C\int_{T_0}^{T_L}\int_{B_{4R}(x_0)}|\D^2
	Q_L|^2\,dxdt,\alabel{pressure estimateP2.2}
\end{align*}
where $\chi_{B_{4R}(x_0)}(x)=1$ for $x\in B_{4R}(x_0)$ and $0$  for $x\in \R^3\backslash B_{4R}(x_0)$.   Choosing
\[c_L^\ast(t)=\int_{\R^3\backslash B_{4R}(x_0)}\frac{(x_{0i}-y_i)(x_{0j}-y_j)}{|x_0-y|^5}F^{ij}(y,t)\,dy\] and using the H\"older inequality,
we estimate
\begin{align*}
	&\int_{T_0}^{T_L}\int_{\R^3}|f_2(z,s)|^2\,dzdt
	\\
	\leq& CR^{5}\int_{T_0}^{T_L}\left|\sum^\infty_{k=4}\frac{C}{(kR)^4}\int_{B_{(k+1)R(x_0)}\backslash B_{kR(x_0)}}F^{ij}(x,t)\,dx\right|^2
	\,dt
	\\
	\leq& C\sup_{y\in\R^3}\int_{T_0}^{T_L}\sum^\infty_{k=4}k^{-4}\int_{B_{R(y)}}|F^{ij}|^2\,dxdt
	\\
	\leq& C \sup_{y\in\R^3}\int_{T_0}^{T_L}\int_{B_{R}(y)}\frac{\varepsilon_0^2}{R^2}(|\D Q_L|^2+|v_L|^2)+(|\D^2 Q_L|^2+|\D v_L|^2)\,dxdt
	.\alabel{pressure estimateP2.3}
\end{align*}
Combining \eqref{pressure estimateP1.2}, \eqref{pressure estimateP2.2} with \eqref{pressure estimateP2.3}, we can apply a standard covering
argument  to complete the proof.
\end{proof}

 Using Lemma \ref{lem 1ord} and Lemma \ref{lem 2ord},  we have:
\begin{lemma}\label{L3 small}
Let $(Q_L,v_L)$ be a strong solution of $\eqref{RBE1}-\eqref{RBE3}$ in $\R^3 \times [T_0,T_L)$ with initial value $(Q_{L,T_0},v_{L,T_0})\in
H^2_{Q_e}(\R^3;S_0)\times H^1(\R^3;\R^3)$ and $\div v=0$. Assume that $Q\in S_\delta$ for   sufficiently small $\delta$  on  $\R^3\times(0,
T_L)$. There exist two  constants  $\varepsilon_0$ and $R$ that
\begin{align}\label{L3  cond}
	\sup_{T_0\leq t\leq T_L,x_0\in\R^3}\int_{B_{R}(x_0)}|\D Q_L|^3+|v_L|^3+\frac{|Q_L-\pi(Q_L)|^3}{L^{\frac32}}\,dx\leq\varepsilon_0^3.
\end{align}
Then we have
\begin{align*}
	&\sup_{T_0\leq s\leq T_L,x_0\in\R^3}\frac{1}{R}\int_{B_{R}(x_0)}\(|\D Q_L |^2+ |v_L |^2+\frac{|Q_{L}-\pi(Q_{L})|^2}{L}\)(\cdot,s)\,dx
	\\
	&+\sup_{x_0\in\R^3}\frac{1}{R}\int^{T_L}_{T_0}\int_{B_{R}(x_0)}|\D^2 Q_L|^2+|\D v_L|^2 +|\p_t Q_L|^2 +\frac{|\D(Q_L-\pi(Q_L))|^2}{L}\,dxdt
	\\
	\leq&  \frac{C}{R}\sup_{x_0\in\R^3}\int_{B_{R}(x_0)}|\D Q_{L,T_0}|^2+| v_{L,T_0}|^2+\frac{|Q_{L,T_0}-\pi(Q_{L,T_0})|^2}{L} \,dx
	+C\varepsilon_0^2\frac{(T_L-T_0)}{R^2}
	\alabel{L3 1st order}
\end{align*}
and
\begin{align*}
	&\sup_{T_0\leq s\leq T_L,x_0\in\R^3}R\int_{B_{R}(x_0)}\(|\D^2 Q_L|^2+|\D v_L|^2+\frac{|\D(Q_L-\pi(Q_L))|^2}{L} \)(\cdot,s)\,dx
	\\
	&+\sup_{x_0\in\R^3}R\int^{T_L}_{T_0}\int_{B_{R}(x_0)}|\D^3Q_L|^2+|\D^2v_L|^2+|\D\p_t Q_L|^2 +\frac{|\D^2(Q_L-\pi(Q_L))|^2}{L}\,dxdt
	\\
	\leq& CR\sup_{x_0\in\R^3}\int_{B_{R}(x_0)} |\D^2 Q_{L,T_0}|^2+ |\D v_{L,T_0}|^2+\frac{|\D  (Q_{L,T_0}-\pi(Q_{L,T_0}))|^2}{L}  \,dxdt
	\\
	&+\frac{C}{R}\sup_{x_0\in\R^3}\int_{B_{R}(x_0)}|\D Q_{L,T_0}|^2+| v_{L,T_0}|^2+\frac{|Q_{L,T_0}-\pi(Q_{L,T_0})|^2}{L} \,dx
	+C\varepsilon_0^2\frac{(T_L-T_0)}{R^2}
	.\alabel{L3 2nd ord}
\end{align*}
\end{lemma}

\begin{proof} Let $\{B_R(x_i)\}^{\infty}_{i=1}$ be a standard open cover of $\R^3$ such that at each $x\in\R^3$,  there are finite
intersections of open balls  $B_{R}(x_i)$.
	Let $\phi\in C_0^\infty(B_{2R}(x_0))$ with $\phi\equiv 1$ on $B_{R}(x_0)$, $|\D \phi|\leq \frac{C}{R}$ and $|\D^2 \phi|\leq
\frac{C}{R^2}$. Recall from Lemma \ref{lem 1ord} that
	\begin{align*}
		&\frac1R\int_{B_{R}(x_0)}\(|\D Q_L|^2+ |v_L|^2+\frac{|Q_L -\pi(Q_L)|^2}{L}\)(\cdot,s) \,dx
		\\
		&+\frac1R\int_{T_0}^{T_L}\int_{B_{R}(x_0)}|\D^2 Q_L|^2+|\p_t Q_L|^2+|\D v_L|^2+\frac{|\D (Q_L -\pi(Q_L ))|^2}{L}\,dxdt
		\\
		\leq& \frac{C}{R}\int_{B_{2R}(x_0)}|\D Q_{L,T_0}|^2+ |v_{L,T_0}|^2+\frac{|Q_{L,T_0}-\pi(Q_{L,T_0})|^2}{L}\,dx
		\\
		&+ \frac{C}{R}\int_{T_0}^{T_L}\int_{B_{2R}(x_0)}|\D Q_L|^2\(|\D Q_L|^2+|v_L|^2+\frac{|Q_L -\pi(Q_L)|^2}{L}\)\,dxdt
		\\
		&+\frac{\eta}{R}\int_{T_0}^{T_L}\int_{B_{2R}(x_0)}|P_L-c_L^*(t)|^2\,dxdt+\frac{C(\eta)}{R^3}\int_{T_0}^{T_L}\int_{B_{2R}(x_0)}|\D
Q_L|^2+ |v_L|^2 \,dxdt
		\alabel{1ord eq L3}
	\end{align*}
for some small $\eta$ to be chosen later. Using H\"older's inequality and \eqref{L3  cond}, we have
\begin{align*}
	& \frac1R\int^{T_L}_{T_0}\int_{B_{2R}(x_0)}\(|\D Q_L|^2+|v_L|^2 +\frac{|Q_L -\pi(Q_L)|^2}{L}\) \,dxdt
	\leq  C\varepsilon_0^2(T_L-T_0).\alabel{L3 L2}
\end{align*}
Then, using  the Sobolev inequality, \eqref{L3  cond} and \eqref{L3 L2}, we find
\begin{align*}
	&\sup_{ x_0\in\R^3}\frac1R\int^{T_L}_{T_0}\int_{B_{2R}(x_0)}|\D Q_L|^2|\D Q_L|^2\,dxdt
	\\
	\leq&\frac{C}{R} \sup_{ x_0\in\R^3}\int_{T_0}^{T_L}\sum_{i}\(\int_{B_R(x_i)}|\D Q_L|^3 \,dx\)^{\frac23}\(\int_{B_R(x_i)}|\D
Q_L|^6\,dx\)^{\frac13}\,dt
	\\
	\leq&  \frac{C\varepsilon_0^2}{R}\sup_{ y\in\R^3}\int^{T_L}_{T_0}\int_{B_{R}(y)}|\D^2 Q_L|^2\,dxdt+\frac{C}{R}\sup_{T_0\leq s\leq
T_L,y\in\R^3}\int^{T_L}_{T_0} \int_{B_R(y)}|\D Q_L|^2\,dxdt
	\\
	\leq&  \frac{C\varepsilon_0^2}{R}\sup_{ y\in\R^3}\int^{T_L}_{T_0}\int_{B_{R}(y)}|\D^2 Q_L|^2\,dxdt+C\varepsilon_0^2\frac{T_L-T_0}{R^2}
	.\alabel{L3}
\end{align*}
Similarly, we obtain
\begin{align*}
	&\sup_{ x_0\in\R^3}\frac1R\int^{T_L}_{T_0}\int_{B_{2R}(x_0)}|\D Q_L|^2\(|v_L|^2+\frac{|Q_L -\pi(Q_L)|^2}{L}\)\,dxdt
	\\
	\leq&  \frac{C\varepsilon_0^2}{R}\sup_{ y\in\R^3}\int^{T_L}_{T_0}\int_{B_{R}(y)}|\D
v_L|^2+\frac{|\D(Q_L-\pi(Q_L))|^2}{L}\,dxdt+C\varepsilon_0^2\frac{T_L-T_0}{R^2}
	.\alabel{L3  L4 part 1}
\end{align*}
Substituting \eqref{L3 L2}-\eqref{L3  L4 part 1} into \eqref{1ord eq L3}, using Lemma \ref{lem pressure estimate} and taking the
supremum of $x_0\in \R^3$, we prove \eqref{L3 1st order} by choosing $\eta$ sufficiently small.

To show \eqref{L3 2nd ord}, recall from Lemma \ref{lem 2ord} that
\begin{align*}
	&R\int_{B_R(x_0)}\(|\D^2 Q_L|^2+ |\D v_L |^2+\frac{|\D (Q_L -\pi(Q_L ))|^2}{L}\)(x,T_L)\,dx
	\\
	&+R\int_{T_0}^{T_L}\int_{B_R(x_0)}|\D^3 Q_L|^2+|\D\p_t Q_L|^2+|\D^2 v_L|^2+\frac{|\D^2 (Q_L -\pi(Q_L ))|^2}{L}\,dxdt
	\\
	\leq& CR\int_{B_{2R}(x_0)} |\D^2 Q_{L,T_0}|^2+ |\D v_{L,T_0}|^2+\frac{|\D  (Q_{L,T_0}-\pi(Q_{L,T_0}))|^2}{L} \,dx
	\\
	&+CR\int_{B_{2R}(x_0)} \frac{|Q_{L,0}-\pi(Q_{L,0})|^2}{L}|\D Q_{L,0}|^2+ \(\frac{|Q-\pi(Q)|^2}{L}|\D Q|^2\)(x,T_L)\,dx
	\\
	&+CR\int_{T_0}^{T_L}\int_{B_{2R}(x_0)}e(Q_L,v_L)\(|\D^2 Q_L|^2+|\D v_L|^2+|\p_tQ_L|^2\)\,dxdt
	\\
	&+CR\int_{T_0}^{T_L}\int_{B_{2R}(x_0)}e(Q_L,v_L)\(\frac{|\D(Q_L-\pi(Q_L))|^2}{L}+e^2(Q_L,v_L)\)\,dxdt
	\\
	&+\frac{C}{R}\int_{T_0}^{T_L}\int_{B_{2R}(x_0)}|\D^2 Q_L|^2+|\D v_L|^2+|\p_tQ_L|^2+\frac{| \D (Q-\pi(Q))|^2}{L}\,dxdt
	\\
	&+\frac{C}{R}\int_{T_0}^{T_L}\int_{B_{2R}(x_0)}e^2(Q_L,v_L)+|P_L-c_L^*(t)|^2\,dxdt
	\\
	\leq& CR\int_{B_{2R}(x_0)} |\D^2 Q_{L,T_0}|^2+ |\D v_{L,T_0}|^2+\frac{|\D  (Q_{L,T_0}-\pi(Q_{L,T_0}))|^2}{L} \,dx
	\\
	&+C\varepsilon_0^2R\sup_{ y\in\R^3}\int_{B_{R}(y)}|\D^2 Q_{L,T_0}|^2 +|\D^2 Q_{L}(x,T_L)|^2 \,dx
	\\
	&+\frac{C}{R}\sup_{y\in\R^3}\int_{B_{R}(y)}|\D Q_{L,T_0}|^2+| v_{L,T_0}|^2+\frac{|Q_{L,T_0}-\pi(Q_{L,T_0})|^2}{L} \,dx
	\\
	&+C\varepsilon_0^2R\sup_{ y\in\R^3}\int_{B_{R}(y)}\int_{T_0}^{T_L}|\D^3 Q_L|^2+|\D^2
v_L|^2+|\D\p_tQ_L|^2+\frac{|\D^2(Q_L-\pi(Q_L))|^2}{L}\,dxdt
	\\
	&+\varepsilon_0^2\frac{C(T_L-T_0)}{R^2}+CR\int_{T_0}^{T_L}\int_{B_{2R}(x_0)}e^3(Q_L,v_L)\,dxdt
	.\alabel{2ord eq L3}
\end{align*}
Here we used the argument in \eqref{L3}, Lemma \ref{lem pressure estimate} and substituted \eqref{L3 1st order}. Using the Sobolev inequality, we deduce the last term in \eqref{2ord eq L3} to
\begin{align*}
	&R\int_{T_0}^{T_L}\int_{B_{2R}(x_0)}e^3(Q_L,v_L)\,dxdt
	\\
	\leq&CR\int_{T_0}^{T_L}\int_{B_{2R}(x_0)}e(Q_L,v_L)\(|\D Q_L|^4+|v_L|^4+\frac{|Q_L -\pi(Q_L)|^4}{L^2}\)\,dxdt
	\\
	\
	\leq&C\varepsilon_0^2R\sup_{ y\in\R^3}\int^{T_L}_{T_0}\int_{B_{R}(y)} |\D|\D Q_L|^2|^2+|\D|v_L|^2|^2+\frac{|\D|Q_L
-\pi(Q_L)|^2|^2}{L^2}\,dxdt
	\\
	&+\frac{C\varepsilon_0^2}{R}\sup_{ y\in\R^3}\int^{T_L}_{T_0}\int_{B_{R}(y)} |\D Q_L|^4+|v_L|^4+\frac{|Q_L -\pi(Q_L)|^4}{L^2}\,dxdt
	\\
	\leq& C\varepsilon_0^2R\sup_{y\in\R^3}\int^{T_L}_{T_0}\int_{B_{R}(y)}|\D^3 Q_L|^2+|\D^2 v_L|^2+\frac{|\D^2 (Q_L -\pi(Q_L ))|^2}{L} \,dxdt
	\\
	&+\frac{C}{R}\sup_{x_0\in\R^3}\int_{B_{R}(x_0)}|\D Q_{L,T_0}|^2+|
v_{L,T_0}|^2+\frac{|Q_{L,T_0}-\pi(Q_{L,T_0})|^2}{L}\,dx+\varepsilon_0^2\frac{C(T_L-T_0)}{R^2}
	.\alabel{L3  L4 part 3}
\end{align*}
Here we also used \eqref{L3 1st order}. Substituting \eqref{L3  L4 part 3} into \eqref{2ord eq L3} and then taking the supremum of
$x_0\in \R^3$ on the resulting expression, we obtain \eqref{L3 2nd ord}.
\end{proof}

Using the Gagliardo-Nirenberg interpolation, we establish  a uniform  local existence of the strong solutions:
\begin{prop}\label{prop Extension}
Assume that $(Q_{L,T_0}, v_{L,T_0})$  satisfies
\begin{align}\label{prop eq}
	\|Q_{L,T_0}\|^2_{H^1_{Q_e}(\R^3)}+\|v_{L,T_0}\|^2_{H^1(\R^3)}+
	\frac {\|Q_{L,T_0}-\pi(Q_{L,T_0})\|^2_{H^1(\R^3)}}{L}\leq M
\end{align}
for some $M>0$. Then there are uniform constants $T_M,R_M$ and $L_M$ depending on $M$ such that the system \eqref{RBE1}-\eqref{RBE3} with initial data
$(Q_{L,T_0},v_{L,T_0})$ has
a unique strong solution $(Q_L,v_L)$ in $\R^3\times[T_0,T_M]$ satisfying
\begin{align}\alabel{prop L3 small}
	\sup_{T_0\leq t\leq T_M,x_0\in\R^3}\int_{B_{R_M}(x_0)}\(|\D
Q_L|^3+|v_L|^3+\frac{|Q_L-\pi(Q_L)|^3}{L^{\frac32}}\)(\cdot,t)\,dx\leq\frac{\varepsilon_0^3}{2}
\end{align}
and
\begin{align*}
	&\sup_{T_0\leq s\leq T_M}\left(\|\D Q_L(s)\|_{H^1(\R^3)}^2+\|v_L(s)\|_{H^1(\R^3)}^2+\frac 1 L\| Q_L(s)-\pi(Q_L(s))\|_{H^1(\R^3)}^2\right)
	\\
	&+\|\p_t Q_L\|_{L^2(T_0,T_M;H^1(\R^3))}^2+\|\D^2 Q_L\|_{L^2(T_0,T_M;H^1(\R^3))}^2
	\\
	&+\|\D v_L\|_{L^2(T_0,T_M;H^1(\R^3))}^2 +\frac 1 L\|\D( Q_L-\pi(Q_L))\|_{H^1(\R^3)}^2\leq C\(1+\frac{\varepsilon_0^2}{R_M^2}\)M\alabel{L3 global}
\end{align*}
provided $L\leq L_M$.
\end{prop}
\begin{proof}

It follows from the Sobolev embedding theorem with the constant $C_s$ that  for any $0<\varepsilon_0<1$, there exists a positive constant
$R_M:=\frac{\varepsilon_0^2}{C_s^2N^2M}$ (cf. \cite{FHM}) such that
\begin{align*}
	&\sup_{x_0\in\R^3}\int_{B_{R_M}(x_0)}|\D Q_{L,T_0}|^3+|v_{L,T_0}|^3+\frac{|Q_{L,T_0}-\pi(Q_{L,T_0})|^{3}}{L^\frac32} \,dx\leq
\frac{\varepsilon_0^3}{N^3},\alabel{prop p1}
\end{align*}
where $N>1$ is an absolute constant independent of $L$ and $M$ to be chosen later. By using the Gagliardo–Nirenberg interpolation (cf.
\cite{FHM}) at $T_0$, we have
\begin{align*}
	&\dist(Q_L(x_0, T_0);S_*)\leq \| Q_L(T_0)-\pi(Q_L(T_0))\|_{L^\infty(\R^3)}
	\\
	\leq& C\|Q_L(T_0)-\pi(Q_L(T_0))\|_{L^2(\R^3)}^{\frac14}\|\D^2(Q_L(T_0)-\pi(Q_L(T_0)))\|_{L^2(\R^3)}^{\frac34}
	\\
	\leq& C(LM)^{\frac18}\Big(\int_{\R^3}|\D^2Q_L(T_0)|^2 +|\p_Q \pi|^2|\D^2Q_L(T_0)|^2+|\p^2_{QQ} \pi|^2|\D Q_L(T_0)|^2\,dx\Big)^{\frac38}
	\\
	\leq& C_dL^{\frac18}M^{\frac14}\leq \frac{\delta}{2},\alabel{Dis}
\end{align*}
where we have used the condition \eqref{prop eq} and  chosen $L\leq L_M:=\(\frac{\delta}{2C_dM^{\frac14}}\)^8$.

Using Theorem \ref{thm loc}, there is a  unique local strong solution $(Q_L,v_L)$ such that  $(Q_L,v_L)$ is continuous in $t$, which follows
from the Sobolev inequality (cf. \cite{HLX}).  Then there is a maximal time $T_L^*\in(T_0,T_L]$ such that
\begin{align}\label{prop dist 2}
	\dist(Q_L;S_*)\leq \delta \text{  on }\R^3\times (T_0,T_L^*)
\end{align} and
\begin{align}\label{prop L3 2}
	\sup_{T_0\leq t\leq T_L^*,x_0\in\R^3}\int_{B_{R_M}(x_0)}\(|\D Q_L|^3+|v_L|^3+\frac{|Q_L-\pi(Q_L)|^3}{L^{\frac32}}\)(\cdot ,t)\,dx\leq{\varepsilon_0^3}.
\end{align}
Next, we claim that $T_L^*\geq T_0+\sigma R_M^2$ for a small constant $\sigma$ to be chosen later.

Otherwise, we assume $T_L^*\leq T_0+\sigma R_M^2$.
For \eqref{L3 global},  using a standard open cover $\{B_{R_M}(x_i)\}_{i=1}^\infty$ of $\R^3$ with finite intersections at each $x\in \R^3$,
the H\"older inequality and the Sobolev inequality, we find
\begin{align*}
	&\int_{T_0}^{T_L^*}\int_{\R^3}|\D Q_L|^4\,dxdt \leq\int_{T_0}^{T_L^*}\sum_{i=1}^{\infty}\(\int_{B_{R_M}(x_i)}|\D
Q_L|^3\,dx\)^{\frac23}\(\int_{B_{R_M}(x_i)}|\D Q_L|^6\,dx\)^{\frac13}dt
	\\
	\leq& C\varepsilon_0^2\int_{T_0}^{T_L^*}\int_{\R^3}|\D^2 Q_L|^2\,dxdt+C\varepsilon_0^2\frac{T_L^*-T_0}{R^2}\sup_{T_0\leq s\leq
T^*}\int_{\R^3}|\D Q_L(\cdot,s)|^2\,dx
	\\
	\leq& \frac12\int_{T_0}^{T_L^*}\int_{\R^3}|\D^2 Q_L|^2\,dxdt+\frac12\sup_{T_0\leq s\leq T_L^*}\int_{\R^3}|\D Q_L(\cdot,s)|^2\,dx\alabel{L3 small
L4  prop}
\end{align*}
for some small $\sigma$. Similarly, we obtain
\begin{align*}
	&\int_{T_0}^{T_L^*}\int_{\R^3}|\D Q_L|^2\(|\D Q_L|^2+|v_L|^2+\frac{|Q_L-\pi(Q_L)|^2}L\)\,dxdt
	\\
	\leq& \frac12\int_{T_0}^{T_L^*}\int_{\R^3}|\D^2 Q_L|^2+|\D v_L|^2 +\frac{|\D(Q_L-\pi(Q_L))|^2}L\,dxdt
	\\
	&+\frac12\sup_{T_0\leq s\leq  T_L^*}\int_{\R^3}\(|\D Q_L|^2+|v_L|^2+\frac{|Q_L-\pi(Q_L)|^2}L\)(\cdot,s)\,dx
	.\alabel{L3  L4  prop 2}
\end{align*}
  Choosing $\phi\equiv 1$ in Lemma \ref{lem 1ord}, using  \eqref{prop eq} and \eqref{L3  L4  prop 2}, we have
\begin{align*}
	&\sup_{T_0\leq s\leq   T_L^*}\int_{\R^3}\(|\D Q_L|^2+ |v_L|^2+\frac{|Q_L-\pi(Q_L)|^2}{L}\)(\cdot,s)\,dx
	\\
	&+\int_{T_0}^{T_L^* }\int_{\R^3}|\D^2 Q_L|^2+|\D v_L|^2+|\p_t Q_L|^2+\frac{|\D  (Q_L-\pi(Q_L))|^2}{L}\,dxdt
	\leq CM.\alabel{prop R3 1}
\end{align*}
Applying \eqref{prop R3 1} and the method in \eqref{L3  L4  prop 2}   to Lemma \ref{lem 2ord} with $\phi\equiv 1$, we find
\begin{align*}
	&\sup_{T_0\leq t\leq T_L^*,x_0\in\R^3}\int_{\R^3}\(|\D^2 Q_L|^2+ |\D v_L|^2+\frac{|\D (Q_L-\pi(Q_L))|^2}{L}\)(\cdot,s)\,dx
	\\
	&+\frac 12\int_{T_0}^{T_L^*}\int_{\R^3}|\D^3 Q_L|^2+|\D^2 v_L|^2+|\D\p_t Q_L|^2+\frac{|\D^2 (Q_L-\pi(Q_L))|^2}{L}\,dxdt
	\\
	\leq&CM+C\varepsilon_0^2\int_{\R^3} |\D^2 Q_L(x,T_L^*)|^2\,dx+\frac{C\varepsilon_0^2}{R_M^2}\int_{\R^3} |\D Q_L(x,T_M)|^2+|\D
Q_{L,0}|^2\,dx
	\\
	&+\frac{C\varepsilon_0^2}{R_M^2}\int_{T_0}^{T_L^*}\int_{\R^3}|\D^2 Q_L|^2+|\D v_L|^2+|\p_t Q_L|^2+\frac{|\D  (Q_L-\pi(Q_L))|^2}{L}\,dxdt
	\\
	&+\frac{C\varepsilon_0^2}{R_M^2}\int_{T_0}^{T_L^*}\int_{\R^3}|\D Q_L|^4+|v_L|^4+\frac{|Q_L-\pi(Q_L)|^4}{L^2}\,dxdt
	\leq C\(1+\frac{\varepsilon_0^2}{R_M^2}\)M
	.\alabel{prop R3 2}
\end{align*}
Combining \eqref{prop R3 1} with \eqref{prop R3 2}, we prove \eqref{L3 global}.

Using \eqref{prop eq}, \eqref{prop p1}, $R_M=\frac{\varepsilon_0^2}{C_s^2N^2M}$ and choosing $T_L^*\leq T_0 + \sigma R_M^2$ for some small
$\sigma$, we have
\begin{align*}
	&CR_M\sup_{x_0\in\R^3}\int_{B_{R_M}(x_0)} |\D^2 Q_{L,T_0}|^2+ |\D v_{L,T_0}|^2+\frac{|\D  (Q_{L,T_0}-\pi(Q_{L,T_0}))|^2}{L}  \,dx
	\\
	&+\frac{C}{R_M}\sup_{x_0\in\R^3}\int_{B_{R_M}(x_0)}|\D Q_{L,T_0}|^2+| v_{L,T_0}|^2+\frac{|Q_{L,T_0}-\pi(Q_{L,T_0})|^2}{L}
\,dx+\varepsilon_0^2\frac{C(T_L-T_0)}{R_M^2}
	\\
	\leq& CMR_M+\frac{C|B_{R_M}|^{\frac 13}}{R_M}\frac{\varepsilon_0^2}{N^2}+C\varepsilon_0^2\sigma\leq
\frac{C\varepsilon_0^2}{N^2}+C\varepsilon_0^2\sigma
	.\alabel{GN}
\end{align*}
By using the Gagliardo-Nirenberg interpolation and applying \eqref{GN} to \eqref{L3 1st order}-\eqref{L3 2nd ord}, we obtain for
$T_L^*\leq T_0+\sigma R_M^2$
\begin{align*}
	&\sup_{T_0\leq t\leq T^*_L,x_0\in\R^3}\int_{B_{R_M}(x_0)} |\D Q_L|^3+|v_L|^3+\frac{|Q_L-\pi(Q_L)|^3}{L^{\frac32}}\,dx
	\\
	\leq&  C\sup_{T_0\leq t\leq T^*_L,x_0\in\R^3}\left(\frac{1}{R_M}\int_{B_{R_M}(x_0)}|\D
Q_L|^2+|v_L|^2+\frac{|Q_L-\pi(Q_L)|^2}{L}\,dx\right)^{3/2}
	\\
	&+C\sup_{T_0\leq t\leq T^*_L,x_0\in\R^3}\left(R_M\int_{B_{R_M}(x_0)}|\D^2 Q_L|^2+|\D
v_L|^2+\frac{|\D(Q_L-\pi(Q_L))|^2}{L}\,dx\right)^{3/2}
	\\
	\leq& \left(\frac{C_1\varepsilon_0^2}{N^2}+C_2\sigma\varepsilon_0^2 \right)^{3/2}\leq\frac{\varepsilon_0^3}{2},
\end{align*}
where we choose $N\geq (8C_1+1)^{\frac12}$ and $\sigma \leq \min\{(8C_2)^{-1},1\}$.

Using a similar argument to the one in \eqref{Dis}, we can prove
 that $\dist(Q_L(t);S_*)\leq \delta/2$  for any $t\in (T_0,T_L^*)$ with $T_L^*\leq T_0+\sigma R_M^2$.
This proves that if $T_L^*\leq T_0+\sigma R_M^2$, then $T_L^*$ is not the maximal time satisfying \eqref{prop dist 2}-\eqref{prop L3 2}.
Therefore   $T_L^*\geq T_M =T_0+\sigma R_M^2$.
\end{proof}

\begin{proof}[\bf Proof of Theorem \ref{thm 2}]
By using Proposition \ref{prop Extension} and Lemma \ref{L3 small},  there exist two uniform positive constants $T_1$
and  $L_\ast$ such that for   any $L\leq L_\ast$, the strong solution  $(Q_L,v_L)$ to  \eqref{RBE1}-\eqref{RBE3} satisfies
\begin{align*}
	&\sup_{0\leq t\leq T_1}\left(\|\D Q_L(t)\|_{H^1(\R^3)}^2+\|v_L(t)\|_{H^1(\R^3)}^2+\frac 1L\| Q_L(t)-\pi(Q_L(t))\|_{H^1(\R^3)}^2 \right)
	\\
	&+\|\p_t Q_L\|_{L^2(0,T_1;H^1(\R^3))}^2+\|\D^2 Q_L\|_{L^2(0,T_1;H^1(\R^3))}^2+\|\D v_L\|_{L^2(0,T_1;H^1(\R^3))}^2
	\\
&+\frac1L \| Q_L-\pi(Q_L)\|_{L^2(0,T_1;H^2(\R^3))}^2  \leq C\(1+\frac{\varepsilon_0^2}{R^2_M}\)M
	.\alabel{L3 global p2}
\end{align*}
Note the pressure $P_L$ satisfies
\eqref{pressure estimateP1.1}. Then  using \eqref{L3 global p2}, we find
\begin{align*}
	\int_{0}^{T_1}\int_{\R^3}|P_L|^2\,dxdt\leq \int_{0}^{T_1}\int_{\R^3}|\D Q_L|^4+|v_L|^4+|\D^2 Q_L|^2\,dxdt\leq C
	\alabel{eq P}
\end{align*}
and
\begin{align*}
	&\int_{0}^{T_1}\int_{\R^3}|\D P_L|^2\,dxdt
	\\
	\leq& C\int_{0}^{T_1}\int_{\R^3}\left(|\D[Q_L,\MH(Q_L,\D Q_L)]|^2+|\D\sigma(Q_L,\D Q_L) |^2+|v_L|^2|\D v_L|^2\right)\,dxdt
	\\
	\leq&  C\int_{0}^{T_1}\int_{\R^3}|\D^3 Q_L|^2+|\D^2 Q_L|^2|\D Q_L|^2+|\D^2 Q_L|^2+|\D v_L|^2|v_L|^2\,dxdt\leq C
	.\alabel{eq P2}
\end{align*}
Multiplying \eqref{RBE1} with $(Q_L-Q_e)$, one can show that $(Q_L-Q_e)\in L^\infty(0,T_1; L^2(\R^3))$.
It follows from \eqref{RBE3}, \eqref{L3 global p2} and \eqref{eq P2} that
\begin{align*}
&\|\p_t v_L\|_{L^2(0,T_1;L^2(\R^3))}^2\leq C \int_0^{T_1}\int_{\R^3}|\D^3 Q_L|^2+|\D^2 v_L|^2+|v|^2|\D v_L|^2\,dxdt
\\
&+C \int_0^{T_1}\int_{\R^3}|\D P_L|^2+|\D v_L|^2+|\D Q_L|^2|\D^2 Q_L|^2+|\D^2 Q_L|^2\,dxdt\leq C.\alabel{eq v_t}
\end{align*}
Then,  letting $L\to0$ (up to a subsequence), we have $Q\in S_*$ and
\begin{align*}
	Q_L\rightharpoonup&\; Q\text{  in  } L^2(0,T_1; H^3_{Q_e}(\R^3))\cap H^1(0,T_1; H^2_{Q_e}(\R^3)),
	\\
	\p_t Q_L\rightharpoonup&\;  \p_t Q\text{  in  }L^2(0,T_1;H^1(\R^3)),
	\\
	v_L\rightharpoonup&\; v\text{  in  } L^2(0,T_1; H^2(\R^3))\cap H^1(\R^3\times(0,T_1)),
	\\
	\p_tv_L\rightharpoonup&\; \p_tv\text{  in  }L^2(0,T_1;L^2(\R^3)),
	\\
	P_L\rightharpoonup&\; P \text{  in  }L^2(0,T_1;H^1(\R^3)).
\end{align*}
Utilizing the Aubin-Lions Lemma (cf. \cite{HLX}) with \eqref{L3 global p2} and \eqref{eq v_t}, we also have
\[(\D Q_L,v_L)\rightarrow\; (Q,v)\text{  in   }L^2(0,T_1;H^1(B_R(0)))\cap C([0,T_1];L^2(B_R(0)))\]
for any $R\in(0,\infty)$.

Since $Q_L$ commutes with itself, we obtain
\begin{align*}
[g_B(Q_L),Q_L]
=&[aQ_L+b(Q_LQ_L-\frac13 \tr(Q_L^2)I)-cQ_L\tr(Q_L^2),Q_L]
\\
=&[bQ_LQ_L,Q_L]=0.
\end{align*}
Taking the Lie bracket of \eqref{RBE3} with $Q_L$ twice, we find
\begin{align*}
	[[\p_t Q_L+v_L\cdot \D  Q _L+[Q_L, \Omega_L],Q_L],Q_L]=& [[\MH(Q_L,\D Q_L),Q_L],Q_L]
	.\alabel{Lie PDE}
\end{align*}
Letting $L \to 0$ in \eqref{Lie PDE}, we have
\begin{align*}
	[[(\p_t +v\cdot \D )Q+[Q, \Omega],Q],Q]= [[\MH(Q,\D Q),Q],Q],\alabel{lie 2}
\end{align*}
where
\begin{align*}
\MH_{ij}=&\frac12\(\D_k[\p_{p_{ij}^k} f_E]+\D_k[\p_{p_{ji}^k} f_E]\)
	- \frac12\(\p_{Q_{ij}} f_E+\p_{Q_{ji}} f_E\)
	\\
	&-\frac{\delta_{ij}}3\sum_{l=1}^3\(\D_k[\p_{p_{ll}^k} f_E]-\p_{Q_{ll}} f_E\).
\end{align*}
Note that  $[I,A]=0$, $\forall A\in \mathbb M^{3\times 3}$ and $Q =s_+ (u\otimes u-\frac 13 I)$ for $|u|=1$. Then
\begin{align*}
[[\p_t Q,Q],Q]_{ij}=& s_+^2(\p_t Q_{ik}u_ku_j +u_iu_k\p_t Q_{kj}-2u_iu_k\p_t Q_{kl} u_lu_j) =s_+^2\p_t Q_{ij},
\\
[[[Q,\Omega],Q],Q]=&s_+^2[\((u\otimes u)\Omega+\Omega(u\otimes u)-2(u\otimes u)\Omega(u\otimes u)\),Q]=[Q,\Omega],
\\
[[\D\p_{p} f,Q],Q]
=&s_+^2\([\nabla_k \p_{p^k} f_E(s_+^{-1}Q+\frac13 I)+(s_+^{-1}Q+\frac13 I)\nabla_k \p_{p^k} f_E\)
\\
&-2s_+^2(s_+^{-1}Q+\frac13 I)\nabla_k \p_{p^k} f_E(s_+^{-1}Q+\frac13 I).
\end{align*}
Recall from \eqref{EL} that
\begin{align*}
	H(Q, \D Q)=&\frac{1}2[\nabla_k \p_{p^k} f_E-\p_{Q} f_E +(\nabla_k \p_{p^k} f_E-\p_{Q} f_E) ^T](s_+^{-1}Q+\frac13 I)
	\\
	&+\frac{1}2(s_+^{-1}Q+\frac13 I)[\nabla_k \p_{p^k} f_E-\p_{Q} f_E +(\nabla_k \p_{p^k} f_E-\p_{Q} f_E) ^T]
	\\
	&-(s_+^{-1}Q+\frac13 I)[\nabla_k \p_{p^k} f_E-\p_{Q} f_E +(\nabla_k \p_{p^k} f_E-\p_{Q} f_E) ^T](s_+^{-1}Q+\frac13 I).
\end{align*}
Then we deduce from the above and \eqref{lie 2} that
\[(\p_t +v\cdot \D )Q+[Q, \Omega]= H(Q,\D Q).\]
Thus,  as $L\to0$, the solution $(Q_L,v_L)$ of \eqref{RBE1}-\eqref{RBE3} converges
 to a solution $(Q, v)$ of \eqref{BE1}-\eqref{BE3}. Taking the difference between two solutions under $L^2$ estimates, it can be
 shown (cf. \cite{HLX} or \cite{FHM}) that  the strong solution  $(Q,v)$ is unique. The proof of uniqueness is similar to the claim 2 in the
 appendix, so we omit the details here.

 Next we verify the criteria of the maximal solution in Theorem \ref{thm 2}. Let $(Q,v)$ be a solution of \eqref{BE1}-\eqref{BE3} in $\R^3\times [0,T^1)$ for any $T^1< T^*$. Assume that
 \begin{align*}\alabel{L3 Q}
 \sup_{0\leq t\leq {T^1},x_0\in\R^3}\int_{B_{R}(x_0)} |\D Q|^3+|v|^3\,dx<\varepsilon_0
 \end{align*}
 for some $\varepsilon_0>0$ and $R>0$.

 Similarly to the proof of Lemma \ref{lem 2ord}, we multiply \eqref{BE1} by $\Delta v$,  and    \eqref{BE3}  by $\Delta H(Q,\D Q)$ and  $\Delta^2 Q$.  Then we apply Sobolev's inequality with \eqref{L3 Q} and use a cover argument to obtain
 \begin{align*}
 	&\int_{\R^3}\(|\D^2 Q |^2+ |\D v  |^2\)(\cdot,T^1)\,dx
 	+\int_0^{T^1}\int_{\R^3}|\D^3 Q |^2+|\D^2 v|^2+|\D\p_t Q |^2\,dxdt
 	\\
 	\leq& C\int_{\R^3}|\D^2 Q_{0}|^2+ |\D v_{0}|^2\,dx
 	\\
 	&+C\int_0^{T^1}\int_{\R^3}(|\D Q |^2+|v |^2)\(|\D^2 Q |^2+|\D Q|^4+|\p_t Q|^2+|\D v |^2\)\,dxdt
 	\\
 	\leq&C+C\varepsilon_0^2\int_0^{T^1}\sum_i\left[\int_{B_R(x_i)}|\D^2 Q |^6+(|\D Q|^2)^{6}+|\p_t Q|^6+|\D v |^6\,dx\right]^{\frac13}dt
 	\\
 	\leq&C+C\varepsilon_0^2\int_0^{T^1}\int_{\R^3}|\D^3 Q |^2+|\D \p_t Q|^2+|\D^2 v |^2\,dxdt
 	\\
 	&+\frac{C}{R^2}\int_0^{T^1}\int_{\R^3}|\D^2 Q |^2+|\D Q|^{2}+|\p_t Q|^2+|\D v |^2\,dxdt.\alabel{3.36}
 \end{align*}
 Using \eqref{3.36} for a sufficiently small $\varepsilon_0$, we know
$(Q(T^1),v(T^1)) \in H^2_{Q_e}(\mathbb{R}^3)\times H^1(\mathbb{R}^3)$. Letting $(Q(T^1),v(T^1))$   be a new initial value at $T^1$,  the local existence  guarantees that
 the solution can be extended passing $T^1$. Therefore, we can extend the solution up to $T^*$.
\end{proof}

\section{Smooth convergence}
In this section, we prove Theorem  \ref{thm 3}. At first,
we obtain the following higher order estimate:
\begin{lemma}\label{kth} 	Let $(Q_L,v_L)$ be a strong solution of $\eqref{RBE1}-\eqref{RBE3}$ in $\R^3 \times [T_0,T_M)$ with initial value
$(Q_{L,T_0},v_{L,T_0})\in H^2_{Q_e} (\R^3)\times H^1(\R^3)$ and $\div v=0$. For any $\tau>T_0$,  $s\in(\tau,T_M]$ and any integer $m\geq 0$,
there exists  a positive constant $C_m$  independently of $Q_L$ and $L$  (but depending on $m$) such  that
	\begin{align*}
		&\sup_{\tau\leq s\leq T_M}\int_{\R^3}\left(|\D^{m+1}Q_L|^2+|\D^m v_L|^2 +\frac{1}{L}|\D^m(Q_L-\pi(Q_L))|^2\right)(\cdot,t)\,dx
		\\
		&+\int_{\tau}^{T_M}\int_{\R^3} |\D^{m+2}Q_L|^2+|\D^{m+1}v_L|^2+|\D^m\p_t Q_L|^2\,dxdt
		\\
		&+\int_{\tau}^{T_M}\int_{\R^3} \frac{1}{L}|\D^{m+1}(Q_L-\pi(Q_L))|^2\,dxdt \leq C_m.
		\alabel{kth estimates}
	\end{align*}
\end{lemma}
\begin{proof}
	We prove this lemma by induction. In view of \eqref{prop R3 1} and \eqref{prop R3 2}, one has shown  \eqref{kth estimates} holds for
$m=0,1$. Assume that \eqref{kth estimates} holds for  $m=1,\cdots, k$ with $k\geq 1$. Then we have
	\begin{align*}
		&\sup_{\frac{\tau}2\leq s\leq T_M}\sum_{i=0}^{k}\int_{\R^3}\left(|\D^{i+1}Q_L|^2+|\D^i v_L|^2
+\frac{1}{L}|\D^i(Q_L-\pi(Q_L))|^2\right)(\cdot,s)\,dx
		\\
		&+\sum_{i=0}^{k}\int_{\frac{\tau}2}^{T_M}\int_{\R^3}|\D^{i+2}Q_L|^2+|\D^{i+1}v_L|^2+|\D^i\p_t Q_L|^2\,dxdt
		\\
		&+\sum_{i=0}^{k}\int_{\frac{\tau}2}^{T_M}\int_{\R^3} \frac{1}{L}|\D^{i+1}(Q_L-\pi(Q_L))|^2  \,dxdt \leq C_k(\tau).
		\alabel{kth aspt}
	\end{align*}
	For $m=k$, it follows from using \eqref{kth aspt}   and the mean value theorem that there exists a   $\tau_{L}\in(\tau/2, \tau)$ such that
	\begin{align*}
		&\int_{\R^3}\left(|\D^{k+2}Q_L|^2+|\D^{k+1} v_L|^2 +\frac{1}{L}|\D^{k+1}(Q_L-\pi(Q_L))|^2  \right)(\cdot,\tau _L)\,dx\leq C_k(\tau).
		\alabel{kth tau}
	\end{align*}
Applying the Sobolev inequality to \eqref{kth aspt}, we obtain
\begin{align*}
	&\sup_{\tau_L\leq s\leq T_M}\sum_{i=0}^{k-1}\|\D^{i}(Q_L-Q_e)(s)\|_{L^\infty(\R^3)}\leq C_k(\tau).
	\alabel{kth Sob1}
\end{align*}
For functions $f_1,f_2\in H^1(\R^3)$, it follows from  H\"older's inequality and   Sobolev's inequality  that
 \begin{align*}
 	&\int_{\R^3}|f_1|^2|f_2|^2\,dx\leq \|f_1\|_{L^3(\R^3)}^{2}\|f_2\|_{L^6(\R^3)}^2
 	\\
 	\leq&C\(\|f_1\|_{L^6(\R^3)}^{\frac12}\|f_1\|_{L^2(\R^3)}^{\frac12}\)^2\|\D f_2\|^2_{L^2(\R^3)}\leq C\|f_1\|^2_{H^1(\R^3)}\|\D f_2\|^2_{L^2(\R^3)},\alabel{kth f1f2}
 	\\
 	\mbox{and}
 	\\
 	&\int_{\R^3}|f_1|^2|f_2|^4\,dx\leq\|f_1\|_{L^6(\R^3)}^2\|f_2\|_{L^6(\R^3)}^4\leq C\|\na f_1\|^2_{L^2(\R^3)}\| \na f_2\|^4_{L^2(\R^3)}.\alabel{kth f1f2^2}
 \end{align*}
	Next, we show that \eqref{kth estimates} holds for  $m=k+1$.
	
	In order to derive the $L^2$-norm of $\D^{k+3} Q_L$, we apply $\D^{k}\D_\beta$ to \eqref{ROT RBE} and multiply by $\D^{k+2}
(R^T(Q_L)\D_\beta  Q_LR(Q_L))$ to obtain
	\begin{align*}
		&\int_{\R^3}\<\D^{k}\D_{\beta} \Big(R^T(Q_L)\big(\p_t Q_L+v_L\cdot \D Q_L\big)R(Q_L)\Big), \D^{k+2} (R^T(Q_L)\D_\beta  Q_LR(Q_L))
\> \,dx
		\\
		&+\int_{\R^3}\<\D^{k}\D_{\beta} \Big(R^T(Q_L)\big([Q_L, \Omega_L]\big)R(Q_L)\Big), \D^{k+2} (R^T(Q_L)\D_\beta  Q_LR(Q_L)) \> \,dx
		\\
		=&\int_{\R^3}\<\D^{k}\D_{\beta}\(R^T(Q_L)\MH(Q_L,\D Q_L) R(Q_L) \),\D^{k+2} (R^T(Q_L)\D_\beta  Q_LR(Q_L))\> \,dx
		\\
		&+\frac1L\int_{\R^3}\<\D^{k}\D_{\beta} g_B(\tilde Q_L),\D^{k+2} (R^T(Q_L)\D_\beta  Q_LR(Q_L))\> \,dx
		.\alabel{kth Delta Q eq}
	\end{align*}
	For the first term on the right-hand side of \eqref{kth Delta Q eq}, we have
	\begin{align*}
	&-\int_{\R^3}\<\D^{k}\D_{\beta}\(R^T(Q_L)\MH(Q_L,\D Q_L) R(Q_L) \),\D^{k+2} (R^T(Q_L)\D_\beta  Q_LR(Q_L))\> \,dx
	\\
	\leq&-\int_{\R^3}\<\D^{k}\D_{\beta}\D_\nu\( \p_{p^\nu} f_E\), \D^{k+2}  \D_\beta  Q_L \>
	\,dx+\eta\int_{\R^3}|\D^{k+2}  \p_p f_E |^2\,dx
	\\
		&+C(\eta)\sum_{\mu_1+\mu_2+\mu_3=k+1}\int_{\R^3}|\D^{\mu_1}\D Q_L|^2|\D^{\mu_2}\D R(Q_L)|^2|\D^{\mu_3}R(Q_L)|^2 \,dx
	\\
	&+C(\eta)\sum_{\mu_1+\mu_2+\mu_3=k}\int_{\R^3}|\D^{\mu_1}\MH(Q_L,\D Q_L)|^2|\D^{\mu_2}\D R(Q_L)|^2|\D^{\mu_3}R(Q_L)|^2 \,dx
	\\
	&+C(\eta)\int_{\R^3}|\D^{k+1}\p_{ Q} f_E|^2\,dx+\eta\int_{\R^3}|\D^{k+2}(R^T(Q_L) \D Q_LR(Q_L))|^2\,dx
	\alabel{kth H 1}
\end{align*}
	for some small $\eta$ to be chosen later. We deduce the first term on the right-hand side in \eqref{kth H 1} from \eqref{sec2 f_E} that
	\begin{align*}
		&-\int_{\R^3}\<\D^{k}\D_{\beta}\D_\nu\( \p_{p^\nu} f_E\), \D^{k+2}  \D_\beta  Q_L \>\,dx
		\\
		\leq &-\int_{\R^3} \p^2_{p^\nu_{ij}p^\gamma_{mn} }f_E\D^{k+1}\D^2_{\beta \gamma}(Q_L)_{mn}\D^{k+1}\D^2_{\beta\nu}(Q_L)_{ij} \,dx
		\\
		&+C\int_{\R^3}|\D^{k+3} Q_L|\sum_{\mu_1+\mu_2=k} |\D^{\mu_1}\D\p^2_{pp}f_E||\D^{\mu_2}\D^2 Q_L| \,dx
		\\
		&+C\int_{\R^3}|\D^{k+3} Q_L|\sum_{\mu_1+\mu_2=k+1}|\D^{\mu_1}\p^2_{pQ }f_E||\D^{\mu_2}\D Q_L| \,dx
		\\
		\leq&-\frac{3\alpha}{8}\int_{\R^3}|\D^{k+3} Q_L |^2\,dx+C\int_{\R^3}\sum_{\mu_1+\mu_2=k} |\D^{\mu_1}\D\p^2_{pp}f_E|^2|\D^{\mu_2}\D^2
Q_L|^2 \,dx
		\\
		&+C\int_{\R^3}\sum_{\mu_1+\mu_2=k+1}|\D^{\mu_1}\p^2_{pQ }f_E|^2|\D^{\mu_2}\D Q_L|^2 \,dx,
		\alabel{kth Delta Q}
	\end{align*}
	where $\alpha$ is a constant defined in \eqref{sec2 f_E}.
Using \eqref{kth aspt}, \eqref{kth
Sob1}-\eqref{kth f1f2^2}, the second last term in \eqref{kth H 1} becomes
	\begin{align*}
		&\int_{\R^3}|\D^{k+2}(R^T(Q_L) \D Q_LR(Q_L))|^2\,dx
		\\
		\leq& C\int_{\R^3}\sum_{\mu_1+\mu_2+\mu_3=k+2}|\D^{\mu_1}\D Q_L|^2|\D^{\mu_2}R(Q_L)|^2|\D^{\mu_3}R(Q_L)|^2\,dx
		\\
		\leq& C\int_{\R^3}|\D^{k+3} Q_L|^2\,dx
		+C\|\D Q_L(x)\|_{L^\infty(\R^3)}^2\int_{\R^3}|\D^{k+2} Q_L|^2\,dx
		\\
		&+C\(\int_{\R^3}|\D^{k+2} Q_L|^2+|\D^{k+1} Q_L|^2\,dx\)^2+C
		\\
		\leq&C\|\D^{k+3} Q_L\|_{L^2(\R^3)}^2
		+C (\|\D^{k+2} Q_L\|_{L^2(\R^3)}^2+1)^2
		.\alabel{kth lot}
	\end{align*}
	Similarly, we obtain that
	\begin{align*}
		&\int_{\R^3}|\D^{k}\D_{\beta}(\p_{ Q} f_E)|^2+\int_{\R^3} \sum_{\mu_1+\mu_2=k}|\D^{\mu_1}\D\p^2_{pp}f_E|^2|\D^{\mu_2}\D^2 Q_L|^2\,dx
		\\
		&+\sum_{\mu_1+\mu_2+\mu_3=k}\int_{\R^3}|\D^{\mu_1}\MH(Q_L,\D Q_L)|^2|\D^{\mu_2}\D R(Q_L)|^2|\D^{\mu_3}R(Q_L))|^2 \,dx
		\\
		&+\int_{\R^3}\sum_{\mu_1+\mu_2=k+1}|\D^{\mu_1}\p^2_{pQ }f_E|^2|\D^{\mu_2}\D Q_L|^2  \,dx \leq C (\|\D^{k+2} Q_L\|_{L^2(\R^3)}^2+1)^2
		.\alabel{kth lower}
	\end{align*}	
	Substituting \eqref{kth Delta Q}-\eqref{kth lower} into \eqref{kth H 1} and choosing sufficiently small $\eta$, we have
	\begin{align*}
	&-\int_{\R^3}\<\D^{k}\D_{\beta}\(R^T(Q_L)\MH(Q_L,\D Q_L) R(Q_L) \),\D^{k+2} (R^T(Q_L)\D_\beta  Q_LR(Q_L))\> \,dx
	\\
	\leq&-\frac{\alpha}{4}\|\D^{k+3} Q_L\|_{L^2(\R^3)}^2 +C(\|\D^{k+2} Q_L\|_{L^2(\R^3)}^2+1)^2
		.\alabel{kth H}
	\end{align*}
	To estimate the second term on the right-hand side of \eqref{kth Delta Q eq}, we utilize Lemma
	\ref{lem gb}, \eqref{kth aspt}, and \eqref{kth Sob1}-\eqref{kth f1f2^2}  to find
	\begin{align*}
		&-\frac1L\int_{\R^3}\<\D^{k}\D_{\beta} g_B(\tilde Q_L),\D^{k+2} (R^T(Q_L)\D_\beta  Q_LR(Q_L))\> \,dx
		\\
		\leq&-\frac{\lambda}{8L} \|\D^{k+2} (Q_L -\pi(Q_L ))\|_{L^2(\R^3)}^2
		\\
		&+C(\|\D^{k+2} Q_L\|_{L^2(\R^3)}^2+1)\(\frac{1}{L}\|\D^{k+1} (Q_L-\pi(Q_L)) \|_{L^2(\R^3)}^2+1\)
		.\alabel{kth g_B}
	\end{align*}
	Applying \eqref{kth aspt}, \eqref{kth Sob1}-\eqref{kth f1f2^2} and \eqref{kth lot} to the left-hand side of \eqref{kth Delta Q eq}, we obtain
	\begin{align*}
		&\int_{\R^3}\<\D^{k}\D_{\beta} \Big(R^T(Q_L)\big(\p_t Q_L+v_L\cdot \D Q_L\big)R(Q_L)\Big), \D^{k+2} (R^T(Q_L)\D_\beta  Q_LR(Q_L))
		\> \,dx
				\\
				&+\int_{\R^3}\<\D^{k}\D_{\beta} \Big(R^T(Q_L)\big([Q_L, \Omega_L]\big)R(Q_L)\Big), \D^{k+2} (R^T(Q_L)\D_\beta  Q_LR(Q_L)) \> \,dx
				\\
				\leq&\int_{\R^3}\<\D^{k}\D_{\beta}\p_t Q_L, \D^{k+2}\D_\beta  Q_L\> \,dx+\eta\int_{\R^3}|\D^{k+2} (R^T(Q_L)\D_\beta
Q_LR(Q_L))|^2\,dx
		\\
		&+C(\eta)\int_{\R^3}\sum_{\mu_1+\mu_2+\mu_3=k}|\D^{ \mu_1} \p_t Q_L|^2|\D^{\mu_2}\D R(Q_L)|^2|\D^{\mu_3} R(Q_L)|^2\,dx
		\\
		&+C(\eta)\int_{\R^3}\sum_{\mu_1+\mu_2+\mu_3=k+1}|\D^{ \mu_1}(v_L\cdot\D Q_L+[Q_L,\Omega_L])|^2|\D^{\mu_2} R(Q_L)|^2|\D^{\mu_3}
R(Q_L)|^2\,dx
		\\
		\leq& -\frac 12\frac{d}{dt}\|\D^{k+2} Q_L\|_{L^2(\R^3)}^2 +\frac{\alpha}8\|\D^{k+3} Q_L\|_{L^2(\R^3)}^2+ C\|\D^{k+2}
v_L\|_{L^2(\R^3)}^2
		\\
		&+C(\|\D^{k+1} v_L\|_{L^2(\R^3)}^2+\| \p_tQ_L\|_{H^k(\R^3)}^2+1)(\|\D^{k+2} Q_L\|_{L^2(\R^3)}^2+1)
		.\alabel{kth k+3 Q left}
	\end{align*}
	Substituting \eqref{kth H}-\eqref{kth k+3 Q left} to \eqref{kth Delta Q eq} and integrating in $t$, we find
	\begin{align*}
		&\frac 12 \|\D^{k+2} Q_{L}(s)\|_{L^2(\R^3)}^2 +\int_{\tau_L}^{s}\frac{\alpha}{8}\|\D^{k+3}
Q_L\|^2_{L^2(\R^3)}+\frac{\lambda}{8L}\|\D^{k+2} (Q_L-\pi(Q_L)) \|_{L^2(\R^3)}^2 \,dt
		\\
		\leq& C\int_{\tau_L}^{s}\|\D^{k+2} v_L\|_{L^2(\R^3)}^2\,dt+\int_{\tau_L}^{s}\|\D^{k+2} Q_L\|_{L^2(\R^3)}^2+\frac{1}{L}\|\D^{k+1}
(Q_L-\pi(Q_L)) \|_{L^2(\R^3)}^2\,dt
		\\
		&+C\int_{\tau_L}^{s}(\|\D^{k+2} Q_L\|_{L^2(\R^3)}^2+\frac{1}{L}\|\D^{k+1} (Q_L-\pi(Q_L)) \|_{L^2(\R^3)}^2)(\|\D^{k+2}
Q_L\|_{L^2(\R^3)}^2+1)\,dt
		\\
		&+C\int_{\tau_L}^{s}(\|\D^{k+1} v_L\|_{L^2(\R^3)}^2+\| \p_tQ_L\|_{H^k(\R^3)}^2+1)(\|\D^{k+2} Q_L\|_{L^2(\R^3)}^2+1)\,dt
		.\alabel{kth Delta Q f}
	\end{align*}
Applying $\D^{k+1}$  to \eqref{ROT RBE} and multiplying the resulting expression by
$\D^{k+1}(R^T(Q_L)\p_tQ_LR(Q_L))$, we have
	\begin{align*}
		&\int_{\tau_L}^{s}\int_{\R^3}\Big<\D^{k+1}\Big( R^T(Q_L)(\p_t Q_L+v_L\cdot \D  Q_L
		\\
		&\qquad+[Q_L, \Omega_L])R^T(Q_L)\Big) , \D^{k+1}(R^T(Q_L)\p_tQ_LR(Q_L))\Big>\,dxdt
		\\
		=&\int_{\tau_L}^{s}\int_{\R^3}\<\D^{k+1}\(R^T(Q_L)\MH(Q_L,\D Q_L) R(Q_L)\),\D^{k+1}(R^T(Q_L)\p_tQ_LR(Q_L))\>\,dxdt
		\\
		&+\int_{\tau_L}^{s}\int_{\R^3}\<\D^{k+1}\frac1Lg_B(\tilde Q_L),\D^{k+1}(R^T(Q_L)\p_tQ_LR(Q_L))\>\,dxdt
		.\alabel{kth pt Q}
	\end{align*}
	It follows from \eqref{kth aspt} and \eqref{kth
	Sob1}-\eqref{kth f1f2^2} that
	\begin{align*}
		&\|\D^{k+1}(R^T(Q_L) \p_t Q_LR(Q_L))\|_{L^2(\R^3)}^2
		\\
		\leq& C\|\D^{k+1}\p_t Q_L\|_{L^2(\R^3)}^2+C\|\p_tQ_L\|^2_{H^k(\R^3)}(\|\D^{k+2} Q_L\|^2_{L^2(\R^3)}+1)
		\alabel{kth R pt Q}
\end{align*}
and		
\begin{align*}
		&\|\D^{k+1} \MH(Q_L,\D Q_L)\|_{L^2(\R^3)}^2
		\leq  C\|\D^{k+3} Q_L\|_{L^2(\R^3)}^2+C(\|\D^{k+2} Q_L\|^2_{L^2(\R^3)}+1)^2
		.\alabel{kth R pt Q2}
	\end{align*}
	Using \eqref{sec2 f_E}, \eqref{kth tau}, \eqref{kth lower} and \eqref{kth R pt Q}-\eqref{kth R pt Q2}, we obtain
	\begin{align*}
	&\int_{\tau_L}^{s}\int_{\R^3}\<\D^{k+1}\(R^T(Q_L)\MH(Q_L,\D Q_L) R(Q_L)\),\D^{k+1}(R^T(Q_L)\p_tQ_LR(Q_L))\>\,dxdt
	\\
		\leq&\int_{\tau_L}^{s}\int_{\R^3}\p^2_{p^\nu_{ij} p^\gamma_{mn}}f_E \D^{k}\D^3_{\mu\gamma\nu}( Q_L)_{mn} \D^{k}\D_\beta\p_t(
Q_L)_{ij}\,dx dt
		\\
		&+C\int_{\tau_L}^{s}\int_{\R^3}|\D^{k+1}\p_t Q_L|\sum_{\mu_1+\mu_2=k}|\D^{\mu_1}\D\p^2_{pp}f_E||\D^{\mu_2} \D^2 Q_L| \,dx dt
		\\
		&+C\int_{\tau_L}^{s}\int_{\R^3}|\D^{k+1}\p_t Q_L|\Big(|\D^{k+1}(\p^2_{pQ}f_E\cdot\D Q_L)|+|\D^{k+1}\p_Qf_E|\Big)\,dx dt
		\\
		&+\eta\int_{\tau_L}^{s}\int_{\R^3}|\D^{k+1} \MH(Q_L,\D Q_L)|^2+|\D^{k+1}(R^T(Q_L) \p_t Q_LR(Q_L))|^2\,dx  dt
		\\
		&+C(\eta)\int_{\tau_L}^{s}(\|\D^{k+2}Q_L\|_{L^2(\R^3)}^2+\|\p_tQ_L\|^2_{H^k(\R^3)})\|\D^{k+2} Q_L\|^2_{L^2(\R^3)}\, dt+C
		\\
		\leq& -\int_{\tau_L}^{s}\int_{\R^3} \frac12\p_t\Big(\p^2_{p_{ij}^\nu p_{mn}^\gamma} f_E
\D^{k+1}\D_\gamma(Q_L)_{mn}\D^{k+1}\D_\nu(Q_L)_{ij}\Big)\,dxdt
		\\
		&+\int_{\tau_L}^{s}\int_{\R^3}\frac 12\p_t\p^2_{p_{ij}^\nu p_{mn}^\gamma} f_E
\D^{k+1}\D_\gamma(Q_L)_{mn}\D^{k+1}\D_\nu(Q_L)_{ij}\,dxdt
		\\
		&-\int_{\tau_L}^{s}\int_{\R^3}\D_\nu\p^2_{p^\nu_{ij} p^\gamma_{mn}}f_E \D^{k}\D^2_{\beta\gamma}( Q_L)_{mn} \D^{k}\D_\beta\p_t(
Q_L)_{ij}\,dx dt
		\\
		&+\eta\int_{\tau_L}^{s}\int_{\R^3}|\D^{k+1} \MH(Q_L,\D Q_L)|^2+|\D^{k+1}(R^T(Q_L) \p_t Q_LR(Q_L))|^2\,dx  dt
		\\
		&+\eta\int_{\tau_L}^{s}\int_{\R^3}|\D^{k+1}\p_t Q_L|^2\,dx  dt
		\\
		&+C(\eta)\int_{\tau_L}^{s}(\|\D^{k+2}Q_L\|_{L^2(\R^3)}^2+\|\p_tQ_L\|^2_{H^k(\R^3)})\|\D^{k+2} Q_L\|^2_{L^2(\R^3)}\, dt+C
		\\
		\leq&-\frac {\alpha}4\|\D^{k+2}  Q_{L}(s) \|_{L^2(\R^3)}^2
		+\int_{\tau_L}^{s} \frac\alpha{8}\|\D^{k+3} Q_L\|^2_{L^2(\R^3)}+\frac18\|\D^{k+1}\p_t Q_L\|^2_{L^2(\R^3)}\, dt
		\\
		&+C\int_{\tau_L}^{s}(\|\D^{k+2}Q_L\|_{L^2(\R^3)}^2+\|\p_tQ_L\|^2_{H^k(\R^3)})\|\D^{k+2} Q_L\|^2_{L^2(\R^3)}\, dt+C
		.\alabel{kth H Q_t}
	\end{align*}
	Replacing $\D$ by $\p_t$ in \eqref{eq: R4}, \eqref{eq: R0}-\eqref{eq: R} and choosing $\xi=\D^{2k+2} g_B(\tilde Q)$, we have
	\begin{align*}
		A_L:=&\p_t  R^T(Q_L) [Q_L-\pi(Q_L)]R(Q_L)+R^T(Q_L)[Q_L-\pi(Q_L)]\p_t  R(Q_L)
		\\
		&+\p_t  [R^T(Q_L)-R^T(\pi(Q_L))]\pi(Q_L) R(Q_L)
		\\
		&+R^T(Q_L)\pi(Q_L)\p_t [R(Q_L)-R(\pi(Q_L))]
		\\
		&+\p_t  R^T(\pi(Q_L))\pi(Q_L)[R(Q_L)-R(\pi(Q_L))]
		\\
		&+[R^T(Q_L)-R^T(\pi(Q_L))]\pi(Q_L)\p_t  R(\pi(Q_L))
	\end{align*}
	and
	\begin{align*}
		&\int_{\tau_L}^{s}\int_{\R^3}\<\D^{k+1}\frac1Lg_B(\tilde Q_L),\D^{k+1}(R^T(Q_L)\p_tQ_LR(Q_L))\>\,dxdt
		\\
		=&\frac{(-1)^{k+1}}L\int_{\tau_L}^{s}\int_{\R^3}\<\D^{2k+2} g_B(\tilde Q),\p_t\tilde Q_L-A_L\>\,dxdt
		\\
		=&-\frac 1L\int_{\tau_L}^{s}\int_{\R^3} \D^{k+1}\p_t(\tilde Q_L) _{ij}\D^{k} \(\p^2_{\tilde Q_{ij}\tilde Q_{kl}}  f_B(\tilde Q_L )\D
(\tilde Q_L)_{kl}\)\,dxdt
		\\
		&-\frac 1L\int_{\tau_L}^{s}\int_{\R^3}\<\D^{k+2} g_B(\tilde Q),\D^k A_L\>\,dxdt
		\\
		=&-\frac 1{2L}\int_{\tau_L}^{s}\int_{\R^3}\p_t\(\p^2_{\tilde Q_{ij}\tilde Q_{kl}}  f_B(\tilde Q_L)\D^{k+1} (\tilde
Q_L)_{kl}\D^{k+1}(\tilde Q_L) _{ij}\)\,dxdt
		\\
		&+\frac 1{2L}\int_{\tau_L}^{s}\int_{\R^3}\p_t\(\p^2_{\tilde Q_{ij}\tilde Q_{kl}}  f_B(\tilde Q_L)\)\D^{k+1} (\tilde
Q_L)_{kl}\D^{k+1}(\tilde Q_L) _{ij}\,dxdt
		\\
		&+\frac CL\int_{\tau_L}^{s}\int_{\R^3} |\D^{k}\p_t\tilde Q_L|\sum_{\mu_1+\mu_2=k}|\D^{\mu_1}\D\p^2_{\tilde Q \tilde Q }  f_B(\tilde
Q_L)||\D^{\mu_2}\D \tilde Q_L|\,dxdt
		\\
		&+\frac CL\int_{\tau_L}^{s}\int_{\R^3}\sum_{\mu_1+\mu_2=k} |\D^{\mu_1}\p_t Q_L||\D^{\mu_2}(Q_L-\pi(Q_L))||\D^{k+2}g_B(\tilde
Q_L)|\,dxdt
		\\
		\leq&-\frac 1{2L}\frac{d}{dt}\int_{\tau_L}^{s}\int_{\R^3}\p^2_{\tilde Q_{ij}\tilde Q_{kl}}  f_B\D^{k+1} (\tilde
Q_L)_{kl}\D^{k+1}(\tilde Q_L)_{ij}\,dxdt
		\\
		&+\eta\int_{\tau_L}^{s}\int_{\R^3}|\D^{k}\p_t(R^T(Q_L) Q_LR(Q_L))|^2+|\D^2\p_t \tilde Q|^2+\frac{1}L |\D^{k+2}\tilde Q_L|^2\,dxdt
		\\
		&+C(\eta) \int_{\tau_L}^{s}\int_{\R^3}\frac{1}{L}|\D^k \tilde Q_L|^2|\p_t Q_L|^2\,dxdt
		\\
		&+C(\eta) \int_{\tau_L}^{s}\int_{\R^3}\frac{1}{L}|\D^{k-1}\tilde Q_L|^2\(| \D\p_t Q_L|^2+\frac{1}{L}|\D^{k+1}\tilde Q_L|^2\)\,dxdt
		\\
		&+C(\eta) \int_{\tau_L}^{s}\int_{\R^3} \frac{1}{L^2} \sum_{\mu_1+\mu_2=k} |\D^{\mu_1}\D\p^2_{\tilde Q \tilde Q }  f_B(\tilde Q_L)|^2
|\D^{\mu_2}\D \tilde Q_L|^2 \,dxdt
		\\
		\leq&- \frac{\lambda}{4L} \|\D^{k+1}( Q_L-\pi(Q_L))(s)\|_{L^2(\R^3)}^2 +C\int_{\tau_L}^{s}\frac{1}{L^2}\|\D^{k+1}
(Q_L-\pi(Q_L))\|_{L^2(\R^3)}^4 \,dt
		\\
		&+\int_{\tau_L}^{s} \frac\lambda{8L} \|\D^{k+2}( Q_L-\pi(Q_L))\|_{L^2(\R^3)}^2+\frac18\|\D^{k+1}\p_t Q_L\|^2_{L^2(\R^3)}\,dt
		\\
		&+C\int_{\tau_L}^{s}\|\D^{k+2} Q_L\|_{L^2(\R^3)}^2\(\frac{1}L\|\D^{k+1} (Q_L-\pi(Q_L))\|_{L^2(\R^3)}^2 +\|\p_t
Q_L\|_{H^k(\R^3)}^2\)\,dt
		\\
		&+C\int_{\tau_L}^{s}\frac{1}L\|\D^{k+1} (Q_L-\pi(Q_L))\|_{L^2(\R^3)}^2\|\p_t Q_L\|_{H^k(\R^3)}^2\,dt
		.
		\alabel{kth g_B Q_t p3}
	\end{align*}
	Using \eqref{kth R pt Q}, we deduce the left-hand side of \eqref{kth pt Q} to
	\begin{align*}
	&-\int_{\tau_L}^{s}\int_{\R^3}\Big<\D^{k+1}\Big( R^T(Q_L)(\p_t Q_L+v_L\cdot \D  Q_L
			\\
			&\qquad\qquad\qquad+[Q_L, \Omega_L])R^T(Q_L)\Big) , \D^{k+1}(R^T(Q_L)\p_tQ_LR(Q_L))\Big>\,dxdt
			\\
		\leq& -\int_{\tau_L}^{s}\int_{\R^3}\<\D^{k+1}\p_t Q_L,\D^{k+1}
\p_tQ_L\>\,dxdt
\\
&+\eta\int_{\tau_L}^{s}\int_{\R^3}|\D^{k+1}(R^T(Q_L)\p_tQ_LR(Q_L))|^2\,dxdt
		\\
		&+C\int_{\tau_L}^{s}\int_{\R^3}\sum_{\mu_1+\mu_2+\mu_3=k}|\D^{\mu_1}\p_tQ_L|^2|\D^{\mu_2} \D R(Q_L)|^2|\D^{\mu_3} R(Q_L)|^2\,dxdt
		\\
		&+C\int_{\tau_L}^{s}\int_{\R^3}\sum_{\mu_1+\mu_2+\mu_3=k+1}\hspace*{-4ex}|\D^{\mu_1}(v_L\cdot \D Q_L+[Q_L,\Omega_L])|^2|\D^{\mu_2}
R(Q_L)|^2|\D^{\mu_3} R(Q_L)|^2\,dxdt
		\\
		\leq&-\int_{\tau_L}^{s}\frac34\|\D^{k+1}\p_t Q_L\|_{L^2(\R^3)}^2+C\|\D^{k+2} v_L\|_{L^2(\R^3)}^2 \,dt
		\\
		&+C\int_{\tau_L}^{s}(\|\D^{k+1} v_L\|_{L^2(\R^3)}^2+\|\p_t Q_L\|_{H^k(\R^3)}^2+1)(\|\D^{k+2} Q_L\|_{L^2(\R^3)}^2+1)\,dt
		.\alabel{kth k+1 Q_t}
	\end{align*}
	Substituting \eqref{kth H Q_t}-\eqref{kth k+1 Q_t} into \eqref{kth pt Q}, we have
	\begin{align*}
		&\frac {\alpha}4 \|\D^{k+2} Q_{L}(s) \|_{L^2(\R^3)}^2+\frac{\lambda}{2L} \|\D^{k+1}( Q_L-\pi(Q_L))(s)\|_{L^2(\R^3)}^2
		+\frac12\int_{\tau_L}^{s}\|\D^{k+1}\p_t Q\|_{L^2(\R^3)}^2 \,dt
		\\
		&\leq C\int_{\tau_L}^{s}\|\D^{k+2} v_L\|_{L^2(\R^3)}^2 \,dt+ \frac{\lambda}{8L}\int_{\tau_L}^{s} \|\D^{k+2}(
Q_L-\pi(Q_L))\|_{L^2(\R^3)}^2 \,dt
		\\
		&
		+\frac18\int_{\tau_L}^{s}\alpha\|\D^{k+3} Q_L\|_{L^2(\R^3)}^2 \,dt+\frac{C}{L^2}\int_{\tau_L}^{s} \|\D^{k+1}
(Q_L-\pi(Q_L))\|_{L^2(\R^3)}^4\,dt
		\\
		&+C\int_{\tau_L}^{s}(\|\D^{k+1} v_L\|_{L^2(\R^3)}^2+\|\p_t Q_L\|_{H^k(\R^3)}^2)\|\D^{k+2} Q_L\|_{L^2(\R^3)}^2+\|\D^{k+2} Q_L\|_{L^2(\R^3)}^4\,dt
		\\
		&+\frac CL\int_{\tau_L}^{s}\(\|\D^{k+2} Q_L\|_{L^2(\R^3)} ^2+\|\p_t Q_L\|_{H^k(\R^3)}^2\) \|\D^{k+1} (Q_L-\pi(Q_L))\|_{L^2(\R^3)}^2
\,dt
		.\alabel{kth k+3 Q final}
	\end{align*}
	Combining \eqref{kth k+3 Q final} with \eqref{kth Delta Q f}   yields
	\begin{align*}
		&\|\D^{k+2}  Q_{L}(s) \|_{L^2(\R^3)}^2+\frac{1}{L}\|\D^{k+1}( Q_L-\pi(Q_L))(s)\|_{L^2(\R^3)}^2
		\\
		&+\int_{\tau_L}^{s} \|\D^{k+3} Q_L\|^2_{L^2(\R^3)}+\|\D^{k+1}\p_t Q\|_{L^2(\R^3)}^2 +\frac{1}{L}\|\D^{k+2} (Q_L-\pi(Q_L))
\|_{L^2(\R^3)}^2 \,dt
		\\
		\leq&C\int_{\tau_L}^{s}\|\D^{k+2} v_L\|_{L^2(\R^3)}^2 \,dt+C\int_{\tau_L}^{s}\(\frac{1}{L}\|\D^{k+1} (Q_L-\pi(Q_L))\|_{L^2(\R^3)}^2
+1\)^2\,dt
		\\
		&+C\int_{\tau_L}^{s}(\|\D^{k+2} Q_L\|_{L^2(\R^3)}^2+\|\D^{k+1} v_L\|_{L^2(\R^3)}^2+\|\p_t Q_L\|_{H^k(\R^3)}^2)\|\D^{k+2}
Q_L\|_{L^2(\R^3)}^2\,dt
		\\
		&+C\int_{\tau_L}^{s}\(\|\D^{k+2} Q_L\|_{L^2(\R^3)} ^2+\|\p_t Q_L\|_{H^k(\R^3)}^2\)\frac{1}{L}\|\D^{k+1} (Q_L-\pi(Q_L))\|_{L^2(\R^3)}^2
\,dt
		.\alabel{kth Q}
	\end{align*}
	To estimate $\D^{k+2} v_L$ in \eqref{kth Q}, we apply $\D^{k+1}$ to \eqref{RBE1} and multiply it by $\D^{k+1} v_L$ to obtain
	\begin{align*}
		&\frac{1}{2}\|\D^{k+1} v_L(s)\|_{L^2(\R^3)}^2+ \int_{\tau_L}^{s}\|\D^{k+2} v_L\|_{L^2(\R^3)}^2dt
		\\
		=& \int_{\tau_L}^{s}\int_{\R^3}\D^{k+1}\(\p_{p^j_{mn}}
		f_E\D_i (Q_L)_{mn} -[Q_L, \MH(Q_L,\D Q_L)]_{ij}\)\D^{k+1}\D_j (v_L)_i\,dxdt
		\\
		\leq&- \int_{\tau_L}^{s}\int_{\R^3}\D^{k+1}[Q_L, \MH(Q_L,\D Q_L)]_{ij}\D^{k+1}\D_j (v_L)_i\,dxdt+\frac14 \int_{\tau_L}^{s}\|\D^{k+2}
v_L\|_{L^2(\R^3)}^2dt
		\\
		&+C\int_{\tau_L}^{s}\int_{\R^3}\sum_{\mu_1+\mu_2=k+1}|\D^{\mu_1}\p_{p}
		f_E|^2|\D^{\mu_2}\D Q_L|^2\,dx.
		\\
		\leq&\frac14 \int_{\tau_L}^{s}\|\D^{k+2} v_L\|_{L^2(\R^3)}^2dt+C\int_{\tau_L}^{s}(\|\D^{k+2} Q_L\|_{L^2(\R^3)}^2+1)^2\,dt
		\\
		&- \int_{\tau_L}^{s}\int_{\R^3}\D^{k+1}[Q_L, \MH]_{ij}\D^{k+1}\D_j (v_L)_i\,dxdt
		.\alabel{kth p2}
	\end{align*}
	The last step follows from the argument in \eqref{kth lower}.
	
	In order to cancel the Lie bracket term in \eqref{kth p2}, we differentiate \eqref{ROT RBE},  multiply by
$\D^{k+1}\(R^T(Q_L)\big(\MH(Q_L,\D Q_L) +\frac 1Lg_B(Q_L)\big)R(Q_L)\)$ and combine with \eqref{kth H +gB}, \eqref{kth p2.3} to obtain
	\begin{align*}
		&\int_{\tau_L}^{s}\int_{\R^3}\left|\D^{k+1}\(R^T(Q_L)\big(\MH(Q_L,\D Q_L) +\frac 1Lg_B(Q_L)\big)R(Q_L)\)\right|^2\,dxdt
		\\
		=&\int_{\tau_L}^{s}\int_{\R^3}\Big<\D^{k+1}(R^T(Q_L)(\p_tQ_L+v\cdot \D  Q_L+[Q_L, \Omega_L])R(Q_L)),
		\\
		&\qquad\qquad\D^{k+1}\big(R^T(Q_L)\big(\MH(Q_L,\D Q_L) +\frac 1Lg_B(Q_L)\big)R(Q_L)\big)\Big>\,dxdt.
		\alabel{kth lie}
	\end{align*}
	Recall from \eqref{kth H Q_t}-\eqref{kth g_B Q_t p3} that
	\begin{align*}
	&\int_{\tau_L}^{s}\int_{\R^3}\Big<\D^{k+1}(R^T(Q_L)\p_tQ_LR(Q_L)),
			\\
			&\qquad\qquad\D^{k+1}\big(R^T(Q_L)\big(\MH(Q_L,\D Q_L) +\frac 1Lg_B(Q_L)\big)R(Q_L)\big)\Big>\,dxdt
			\\
		\leq&-\(\frac{\alpha}{4}\|\D^{k+2}Q_L(s)\|^2_{L^2(\R^3)}+\frac {\lambda}{2L} \|\D^{k+1} (Q_L-\pi(Q_L))(s)\|^2_{L^2(\R^3)}\)
		\\
		&\hspace{-4ex}+\eta_1\int_{\tau_L}^{s}\|\D^{k+3} Q_L\|_{L^2(\R^3)}^2+\|\D^{k+1}\p_t Q_L\|^2_{L^2(\R^3)}+ \frac{1}{L}\|\D^{k+2}(
	Q_L-\pi(Q_L))\|_{L^2(\R^3)}^2 \,dt
		\\
		&+C\int_{\tau_L}^{s}\|\D^{k+2} Q_L\|_{L^2(\R^3)} ^2\(\|\D^{k+2}Q_L\|_{L^2(\R^3)}^2+\|\p_t Q_L\|_{H^k(\R^3)}^2\)\,dt
		\\
		&+C\int_{\tau_L}^{s}\frac{1}L\|\D^{k+1} (Q_L-\pi(Q_L))\|_{L^2(\R^3)}^2\( \|\D^{k+2}Q_L\|_{L^2(\R^3)}^2+\|\p_t Q_L\|_{H^k(\R^3)}^2\)\,dt
		\\
		&+C\int_{\tau_L}^{s}\(\frac{1}L\|\D^{k+1} (Q_L-\pi(Q_L))\|_{L^2(\R^3)}^2+1\)^2 \,dt
		.\alabel{kth H +gB}
	\end{align*}
 We apply Lemma \ref{Lie} to $\Delta^{k+1}\Omega_L$ with $A=Q_L, B=\MH(Q_L,\D Q_L)+\frac 1Lg_B( Q_L), F=\Delta^{k+1}\Omega_L $ and obtain
	\begin{align*}
		&\<[Q_L , \Delta^{k+1}\Omega_L ],\MH(Q_L,\D Q_L)+\frac 1Lg_B( Q_L)\>=\Delta^{k+1}\D_jv_i[Q_L,\MH(Q_L,\D Q_L)]_{ij}.
		\alabel{Lie k}
	\end{align*}
Note from  \eqref{RBE3} that
\begin{align*}
	|\D^k(\MH(Q_L,\D Q_L)+\frac 1Lg_B(Q_L))|^2\leq&C(|\D^k \p_t Q_L|^2+\sum_{\mu_1=\mu_2=k+1}|\D^{\mu_1} Q_L|^2|\D^{\mu_2} v_L|^2)
	.\alabel{H gB}
\end{align*}
	Then using \eqref{kth aspt}, \eqref{kth Sob1}-\eqref{kth f1f2^2}, \eqref{Lie k} and \eqref{H gB}, we find
	\begin{align*}
	&\int_{\tau_L}^{s}\int_{\R^3}\Big<\D^{k+1}(R^T(Q_L)[Q_L, \Omega_L]R(Q_L)),
			\\
			&\qquad\qquad\D^{k+1}\big(R^T(Q_L)\big(\MH(Q_L,\D Q_L) +\frac 1Lg_B(Q_L)\big)R(Q_L)\big)\Big>\,dxdt
			\\
		\leq&(-1)^{k+1}\int_{\R^3}\<\Delta^{k+1}[Q_L , \Omega_L ],\MH(Q_L,\D Q_L)+\frac 1Lg_B(Q_L)\> \,dx
		\\
		&+\eta\int_{\R^3}|\D^{k+1} (\MH(Q_L,\D Q_L)+\frac 1Lg_B(Q_L) )|^2+|\D^{k+1}(R^T(Q_L)[Q_L , \Omega_L ]R(Q_L))|^2\,dx
		\\
		&+C(\eta)\int_{\R^3}\sum_{\mu_1+\mu_2+\mu_3=k}|\D^{\mu_1}[Q_L , \Omega_L ]|^2|\D^{\mu_2}\D R(Q_L)|^2|\D^{\mu_3} R(Q_L)|^2\,dx
		\\
		&+C(\eta)\int_{\R^3}\sum_{\mu_1+\mu_2+\mu_3=k}\hspace*{-4ex}|\D^{\mu_1}(\MH(Q_L,\D Q_L)+\frac {g_B(Q_L)}L)|^2|\D^{\mu_2}\D
R(Q_L)|^2|\D^{\mu_3}R(Q_L)|^2 \,dx
		\\
		\leq& (-1)^{k+1}\int_{\R^3}\D^{2k+2}\D_j(v_L)_i[Q_L,(\MH(Q_L,\D Q_L)+\frac 1Lg_B(Q_L) )]_{ij}\,dx
		\\
		&+C\int_{\R^3}|\D^{k+1} \Omega_L|\sum_{\substack{\mu_1 +\mu_2+= k}}|\D^{\mu_1}\D Q_L||\D^{\mu_2}(\MH(Q_L,\D Q_L)+\frac 1Lg_B(Q_L)
)|\,dx
		\\
		&+\eta\int_{\R^3}|\D^{k+1} (\MH(Q_L,\D Q_L)+\frac 1Lg_B(Q_L) )|^2+|\D^{k+1}(R^T(Q_L)[Q_L , \Omega_L ]R(Q_L))|^2\,dx
		\\
		&+C(\|\D^{k+1} v_L\|_{L^2(\R^3)}^2+1)(\|\D^{k+2} Q_L\|_{L^2(\R^3)}^2+1)
		\\
		\leq& \int_{\R^3}\D^{k+1}\D_j (v_L)_i\D^{k+1}[Q_L, \MH(Q_L,\D Q_L)]_{ij}\,dx
		\\
		&+\frac14 \|\D^{k+2} v_L\|_{L^2(\R^3)}^2+\eta_1(\|\D^{k+2} Q_L\|_{L^2(\R^3)}^2+\|\p_t Q_L\|_{H^{k+1}}^2)
		\\
		& +C(\|\D^{k+1} v_L\|_{L^2(\R^3)}^2+1)(\|\D^{k+2} Q_L\|_{L^2(\R^3)}^2+1)
		.\alabel{kth p2.3}
	\end{align*}

	We substitute \eqref{kth H +gB}-\eqref{kth p2.3} into \eqref{kth lie} and find
	\begin{align*}
		&\int_{\tau_L}^{s}\int_{\R^3}\left|\D^{k+1}\(R^T(Q_L)\big(\MH(Q_L,\D Q_L) +\frac 1Lg_B(Q_L)\big)R(Q_L)\)\right|^2\,dxdt
		\\
		\leq&\int_{\tau_L}^{s} \int_{\R^3}\D^{k+1}\D_j (v_L)_i\D^{k+1}[Q_L, \MH(Q_L,\D Q_L)]_{ij}\,dx+\frac14 \|\D^{k+2}
v_L\|_{L^2(\R^3)}^2\,dt
		\\
		&\hspace{-3ex}+2\eta_1\int_{\tau_L}^{s}\|\D^{k+3} Q_L\|_{L^2(\R^3)}^2+\|\D^{k+1}\p_t Q_L\|^2_{L^2(\R^3)}+ \frac{1}L\|\D^{k+2}(
Q_L-\pi(Q_L))\|_{L^2(\R^3)}^2\,dt
		\\
		&\hspace{-3ex}+C\int_{\tau_L}^{s}\|\D^{k+2} Q_L\|_{L^2(\R^3)} ^2\(\|\D^{k+2}Q_L\|_{L^2(\R^3)}^2+\|\D^{k+1} v_L\|_{L^2(\R^3)}^2+\|\p_t
Q_L\|_{H^k(\R^3)}^2\)\,dt
		\\
		&+C\int_{\tau_L}^{s}\frac{1}L\|\D^{k+1} (Q_L-\pi(Q_L))\|_{L^2(\R^3)}^2\( \|\D^{k+2}Q_L\|_{L^2(\R^3)}^2+\|\p_t
Q_L\|_{H^k(\R^3)}^2\)\,dt
		\\
		&+C\int_{\tau_L}^{s}\(\frac{1}L\|\D^{k+1} (Q_L-\pi(Q_L))\|_{L^2(\R^3)}^2+1\)^2
		\\
		&+\frac12\|\D^{k+1}(R^T(Q_L)(\MH(Q_L,\D Q_L)+\frac 1Lg_B(Q_L))R(Q_L))\|_{L^2(\R^3)}^2\,dt.
		\alabel{kth p2.1}
	\end{align*}
By adding \eqref{kth p2.1} to \eqref{kth p2},  we obtain
	\begin{align*}
		& \frac{1}{2}\|\D^{k+1} v_L(s)\|_{L^2(\R^3)}^2+\frac 14\int_{\tau_L}^{s}\|\D^{k+2} v_L\|_{L^2(\R^3)}^2\,dt
		\\
		\leq &2\eta_1\int_{\tau_L}^{s}\|\D^{k+3} Q_L\|_{L^2(\R^3)}^2+\|\D^{k+1}\p_t Q_L\|^2_{L^2(\R^3)}+ \frac{1}L\|\D^{k+2}(
Q_L-\pi(Q_L))\|_{L^2(\R^3)}^2\,dt
		\\
		&\hspace{-2ex}+C\int_{\tau_L}^{s}\|\D^{k+2} Q_L\|_{L^2(\R^3)} ^2\(\|\D^{k+2}Q_L\|_{L^2(\R^3)}^2+\|\D^{k+1} v_L\|_{L^2(\R^3)}^2+\|\p_t
Q_L\|_{H^k(\R^3)}^2\)\,dt
		\\
		&+C\int_{\tau_L}^{s}\frac{1}L\|\D^{k+1} (Q_L-\pi(Q_L))\|_{L^2(\R^3)}^2\( \|\D^{k+2}Q_L\|_{L^2(\R^3)}^2+\|\p_t
Q_L\|_{H^k(\R^3)}^2\)\,dt
		\\
		&+C\int_{\tau_L}^{s}\(\frac{1}L\|\D^{k+1} (Q_L-\pi(Q_L))\|_{L^2(\R^3)}^2+1\)^2\,dt
		.\alabel{kth v all}
	\end{align*}
	Substituting \eqref{kth Q} into \eqref{kth v all}, choosing sufficiently small $\eta_1$ and combining it with \eqref{kth Q}, we conclude
	\begin{align*}
		&\|\D^2  Q_{L}(s) \|_{L^2(\R^3)}^2+\|\D^{k+1} v_L(s)\|_{L^2(\R^3)}^2+\frac{1}{L}\|\D^{k+1}( Q_L-\pi(Q_L))(s)\|_{L^2(\R^3)}^2
		\\
		&+\int_{\tau_L}^{s} \|\D^{k+3} Q_L\|^2_{L^2(\R^3)}+\|\D^{k+2} v_L\|_{L^2(\R^3)}^2+\|\D^{k+1}\p_t Q\|_{L^2(\R^3)}^2\,dt
		\\
		&+ \int_{\tau_L}^{s}\frac{1}{L}\|\D^{k+2} (Q_L-\pi(Q_L)) \|_{L^2(\R^3)}^2 \,dt
		\\
		\leq&C\int_{\tau_L}^{s}\(\|\D^{k+2} Q_L\|_{L^2(\R^3)}^2+\|\D^{k+1} v_L\|_{L^2(\R^3)}^2+\frac{1}{L}\|\D^{k+1}
(Q_L-\pi(Q_L))\|_{L^2(\R^3)}^2 \)
		\\
		&\qquad\qquad\Big(\|\D^{k+2} Q_L\|_{L^2(\R^3)}^2+\|\D^{k+1} v_L\|_{L^2(\R^3)}^2+\|\p_t Q_L\|_{H^k(\R^3)}^2
		\\
		&\qquad\qquad+\frac{1}{L}\|\D^{k+1} (Q_L-\pi(Q_L))\|_{L^2(\R^3)}^2 \Big) \,dt+C
		.\alabel{kth final}
	\end{align*}
	Applying the Gronwall inequality to \eqref{kth final}  with \eqref{kth aspt} for $t\in (\tau_L,s)$, it concludes that \eqref{kth estimates}
holds for $m=k+1$ on the $(\tau,s)$. Since $\tau\geq T_0$ is an arbitrary positive constant,  we prove \eqref{kth estimates} for any $s\in
(\tau,T_M]$ and $m=k+1$,  which completes a proof of this lemma.
\end{proof}

\begin{proof}[\bf Proof of Theorem \ref{thm 3}]	Let $(Q,v)$ be the strong solution  to \eqref{BE1}-\eqref{BE3} in $\R^3\times[0,T^\ast)$ with
initial data $(Q_0,v_0)\in H^2_{Q_e}(\R^3)\times H^1(\R^3)$, where  $T^\ast$ is the maximal existence time. Given any $T\in (0,T^\ast)$, set
	\begin{equation*}
		M= 2\sup_{0\leq t\leq T}\|(\D Q,v)\|_{H^1(\R^3)}^2.
	\end{equation*}
	Using Theorem \ref{thm 2},    there exists a subsequence $(Q_{L},v_{L})$ such that
	\begin{equation*}
		(\D Q_{L},v_{L})\rightarrow(\D Q,v),\qquad\text{in }~~L^\infty(0,T_M;L_{loc}^2(\R^3))\cap L^2(0,T_M;H_{loc}^1(\R^3)).
	\end{equation*}
	Suppose that $T_M<T$. We apply Lemma \ref{kth} with  $k\geq 2$ to obtain
	\[  \sup_{\tau\leq s\leq T_M}\sum_{i=0}^{k}\int_{\R^3}\left(|\D^{i+1}Q_L|^2+|\D^i v_L|^2
 \right)(\cdot,s)\,dx\leq C_k \] for a uniform constant $C_k$ in $L$.
Similarly to Lemma \ref{lem energy}, one can show the energy identity:
	\begin{align*}
		&\int_{\R^3}\(f_E(Q ,\D Q ) +\frac{|v |^2}{2}\)(\cdot,s)\,dx+\int_0^s\int_{\R^3}\left|H(Q,\D Q ) \right|^2+|\D v |^2\,dxdt
		\\
		&=\int_{\R^3}\Big(f_E(Q_{0},\D Q_{0}) +\frac{|v_{0}|^2}{2}\Big)\,dx.\alabel{energy Q}
	\end{align*}
		Then   comparing \eqref{energy Q} with \eqref{energy eq} (cf. Lemma 4.3 in \cite{FHM}),
	integrating by parts and using H\"older's inequality, we obtain
	\begin{align*}
		& \lim\limits_{L\rightarrow 0}\|(\D^2 Q_L-\D^2 Q)(T_M)\|_{L^2(\R^3)}^2
		\\
		\leq& \lim\limits_{L\rightarrow 0}\(\int_{\R^3}|(\D Q_L-\D Q)(T_M)|^2\,dx\)^{\frac12}\(\int_{\R^3}|(\D^3 Q_L-\D^3
Q)(T_M)|^2\,dx\)^{\frac12}
		=0.
	\end{align*}
	Similarly, we find
	\begin{equation*}
		\lim\limits_{L\rightarrow 0}\|(\D v_L-\D v)(T_M)\|_{L^2}^2=0,\quad\lim\limits_{L\rightarrow 0}\frac 1L\||\D
(Q_L-\pi(Q_L))|^2)(T_M)\|_{L^2}^2=0.
	\end{equation*}
	Therefore, we obtain
	\begin{align*}
		&\lim\limits_{L\rightarrow 0}\left(\|\D Q_L(T_M)\|_{H^1(\R^3)}^2+\|v_L(T_M)\|_{H^1(\R^3)}^2+\frac1L
\|(Q_L-\pi(Q_L))(T_M)\|_{H^1(\R^3)}^2\right)\\
		=&\|\D Q(T_M)\|_{H^1(\R^3)}^2+\|v(T_M)\|_{H^1(\R^3)}^2\leq\frac{M}{2}.\nonumber
	\end{align*}
	Hence,  for sufficiently small $L$, one has
	\begin{equation*}
		\|\D Q_L(T_M)\|_{H^1(\R^3)}^2+\|v_L(T_M)\|_{H^1(\R^3)}^2+\frac1L \|(Q_L-\pi(Q_L))(T_M)\|_{H^1(\R^3)}^2\leq M.
	\end{equation*}
	Utilizing Proposition \ref{prop Extension} with   initial data $(Q_L(T_M),v_L(T_M))$, we  extend the strong solution $(Q_L,v_L)$ to the
time  $T_1=:\min\{T,2T_M\}>T_M$. That is
	\begin{equation}
		(\D Q_L,v_L)\rightarrow(\D Q,v),\qquad\text{in }~~L^\infty(0,T_1;L^2(\R^3))\cap L^2(0,T_1;H^1(\R^3))
	\end{equation}
and
	\begin{equation}
		(\D Q_L,v_L)\rightarrow(\D Q,v) \qquad\text{in }~~C^\infty(\tau,T_1;C_{loc}^\infty(\R^3)) \quad\text{for  any }\tau>0.
	\end{equation}
	  Repeating the above argument, we establish \eqref{con1}-\eqref{con2} for any $T<T^\ast$ as any sequence $L\to 0$ due to the   uniqueness  of the solution $(Q, v)$. We prove   Theorem
\ref{thm 3}.
\end{proof}

\section{Appendix: Local existence and proof of Theorem \ref{thm loc}}

For any $f(x)\in H^1(\R^3)$, it follows from the Gagliardo–Nirenberg interpolation that
\[\int_{\R^3} |f(x)|^4\,dx\leq \(\int_{\R^3} |f(x)|^2\,dx\)^{\frac12}\(\int_{\R^3} |\D f(x)|^2\,dx\)^{\frac32}.\]
Then  we have
\begin{align*}
	\(\int_{\R^3} |f(x)|^4\,dx\)^\frac12\leq \eta\int_{\R^3} |\D f(x)|^2\,dx+\frac{C}{\eta^3} \int_{\R^3} |f(x)|^2\,dx.\alabel{L4}
\end{align*}
Using \eqref{L4}, we now prove the local existence of \eqref{RBE1}-\eqref{RBE3} with initial data $(Q_{L,0},v_{L,0})$.
\begin{proof}[\bf Proof of Theorem \ref{thm loc}] Without loss of generality, we assume $L=1$ and omit the subscript $L$ in the proof.

Assume that the initial data $(Q_{L,0},v_{L,0})\in H^2_{Q_e}(\R^3)\times H^1(\R^3)$ satisfies $\div
v_{L,0}=0$, $\|Q_{L,0}\|_{L^\infty(\R^3)}\leq K$   in the assumption of Theorem   \ref{thm loc} and set
\begin{align}
	\|Q_{L,0}\|^2_{H^2_{Q_e}(\R^3)}+\|v_{L,0}\|^2_{H^1(\R^3)} = M_1
	.\alabel{t=0}
\end{align}
For a given $T$, define
the space
	\begin{align*}
		\mathcal V(0,T)=\Big\{(Q,v): &\sup_{0\leq t\leq T}\big(\| Q (t)\|_{H^2_{Q_e}(\R^3)}^2+\|v (t)\|_{H^1(\R^3)}^2 \big)+\|\D^3 Q
\|_{L^2(0,T ;L^2(\R^3))}^2
		\\
		&+\|\p_t Q \|_{L^2(0,T ;H^1(\R^3))}^2 +\|\D^2 v \|_{L^2(0,T;L^2(\R^3))}^2 \leq  C_1M_1,
		\\
		& \D\cdot v=0,\quad \sup_{0\leq t\leq T} \| Q (t)\|_{L^\infty(\R^3)}\leq 2K  \Big\}
	\end{align*}
	for  a uniform constant $C_1$ in $M_1$ to be chosen later.

 For a given pair $(Q_m,v_m)\in \mathcal V(0,T_m)$,  there exist  a $T_{m+1}\leq T_m$ and  a unique strong solution $(\D Q_{m+1},v_{m+1})\in L^2(0,T_{m+1};H^2(\R^3))\cap L^\infty (0,T_{m+1};H^1(\R^3))$
 of
  the linearized system:
	\begin{align*}
		&(\p_t -\Delta)v_{m+1} +\D P_{m+1}-\D\cdot[Q_{m},h(Q_m,Q_{m+1})]
		\\
		&=-v_m\cdot \D v_m -\D\cdot\sigma_{ij}(Q_m,\D Q_m),
		\alabel{v +1}
		\\[2ex]
		&\D\cdot v_{m+1}=0,\alabel{div v +1}
		\\[2ex]
		&\p_tQ_{m+1}+[Q_{m}, \Omega_{m+1}]-h(Q_m,Q_{m+1})
		=   -v_m \cdot\D Q_{m}+g_B(Q_m)
		\alabel{Q +1}
	\end{align*}
with initial data $(Q_0,v_0)$, where
\begin{align*}
	h_{ij}(Q_m,Q_{m+1}):=&\frac12\(\D_\beta[\p_{p_{ij}^\beta} f_E(Q_m,\D Q_{m+1})]+\D_\beta[\p_{p_{ji}^\beta} f_E(Q_m,\D Q_{m+1})]\)
	\\
	&- \frac12\(\p_{Q_{ij}} f_E(Q_m,\D Q_m)-\p_{Q_{ji}} f_E(Q_m,\D Q_m)\)
	\\
	&-\frac{\delta_{ij}}3\sum_{l=1}^3\(\D_\beta[\p_{p_{ll}^\beta} f_E(Q_m,\D Q_{m+1})]-\p_{Q_{ll}} f_E(Q_m,\D Q_m)\)
	.\alabel{Mol m}
\end{align*}

\noindent {\bf Claim 1}:  There exists a uniform $T_{M_1}$ in $m$  such that
 $T_{M_1}\leq T_{m+1}$ for all $m\geq 1$ and  $(Q_{m+1},v_{m+1})\in \mathcal V(0,T_{M_1})$.
	
To establish the $L^2$-norm of $\D^3 Q_{m+1}.$  we multiply \eqref{Q +1} with $\Delta^2 Q_{m+1}$ and observe
	\begin{align*}
		&\int_0^s\int_{\R^3}\<\(\p_tQ_{m+1}+[Q_m, \Omega_{m+1}]-\D_\beta\p_{p^\beta}f_E(Q_m,\D Q_{m+1})\),\Delta^2 Q_{m+1}\>\,dxdt
		\\
		=&\int_0^s\int_{\R^3}\< \p_Qf_E(Q_m,\D Q_m)-v_m \cdot\D Q_{m}+g_B(Q_m),\Delta^2 Q_{m+1}\>\,dxdt
		.\alabel{D2 Q+1}
	\end{align*}
We can compute the second term in the left-hand side of \eqref{D2 Q+1}
\begin{align*}
	&\int_0^s\int_{\R^3}\<[Q_m, \Omega_{m+1}] ,\Delta^2 Q_{m+1}\>\,dxdt
	\\
	\leq&  \frac{\alpha}{8} \int_0^s\int_{\R^3} |\D^3 Q_{m+1}|^2\,dxdt+C\int_0^s\int_{\R^3} |\D^2 v_{m+1}|^2\,dxdt
	\\
	&+C\int_0^s\int_{\R^3} |\D v_{m+1}|^2\,dx \int_{\R^3}|\D^2  Q_{m}|^2\,dxdt+\int_0^s\int_{\R^3} |v_{m+1}|^2 |\D^2  Q_{m}|^2\,dxdt
	.\alabel{Omega m}
\end{align*}
Using the Sobolev inequality and \eqref{t=0}, we have
\begin{align*}
	\sup_{0\leq t\leq T_{m+1}}\(\|\D Q_{m+1}(t)\|^2_{H^1(\R^3)}+\|v_{m+1}(t)\|^2_{H^1(\R^3)}\)\leq CM_1
	.\alabel{sup t m}
\end{align*}
We employ the inequalities \eqref{L4} and \eqref{sup t m} to obtain
\begin{align*}
	&\int_0^s\int_{\R^3}|v_{m+1}|^2 |\D^2  Q_{m}|^2 \,dxdt
	\leq \int_0^s\(\int_{\R^3}|v_{m+1}|^4\,dx\)^{\frac12}\(\int_{\R^3}|\D^2 Q_m|^4 \,dx\)^{\frac12}dt
	\\
	\leq &C\int_0^s\(\int_{\R^3}|\D v_{m+1}|^2+|v_{m+1}|^2\,dx\)\(\int_{\R^3}\eta_1|\D^3 Q_m|^3 +\frac{C}{\eta^3_1}|\D^2 Q_m|^2\,dx\)dt
	\\
	\leq&CM_1^2(\eta_1 +\frac{s}{\eta_1^3})
	.\alabel{key est}
\end{align*}
Substituting \eqref{key est} to  \eqref{Omega m} yields
\begin{align*}
	&\int_0^s\int_{\R^3}\<[Q_m, \Omega_{m+1}] ,\Delta^2 Q_{m+1}\>\,dxdt
	\\
	\leq& \frac{\alpha}{8} \int_0^s\int_{\R^3} |\D^3 Q_{m+1}|^2\,dxdt+C\int_0^s\int_{\R^3} |\D^2 v_{m+1}|^2\,dxdt
	 +CM_1^2(\eta_1 +s+\frac{s}{\eta_1^3})
	.\alabel{Omega m2}
\end{align*}
To estimate  the third term on the left-hand side of \eqref{D2 Q+1} , it follows from integrating by parts and using \eqref{sec2 f_E} that
	\begin{align*}
		&-\int_0^s\int_{\R^3}\D_\beta\p_{p^\beta_{ij}}f_E(Q_m,\D Q_{m+1})\Delta^2( Q_{m+1})_{ij}\,dxdt
		\\
		=&-\int_0^s\int_{\R^3}\p^2_{p^\beta_{ij}p^\nu_{kl}}f_E(Q_m,\D Q_{m+1})\D^3_{\mu\gamma\nu} (Q_{m+1})_{kl}\D^3_{\beta\gamma\mu}
(Q_{m+1})_{ij}\,dxdt
		\\
		&+\int_0^s\int_{\R^3}\D_{\gamma}\p^2_{p^\beta_{ij} p_{kl}^\nu }f_E(Q_m,\D Q_{m+1})\D^2_{\mu \nu}  (Q_{m+1})_{kl}\D^3_{\beta\gamma\mu}
(Q_{m+1})_{ij}\,dxdt
		\\
		&+\int_0^s\int_{\R^3}\D_{\gamma}\(\p^2_{p^\beta_{ij} Q_{kl}}f_E(Q_m,\D Q_{m+1}) \D_{\mu}( Q_m)_{kl}\)\D^3_{\beta\gamma\mu}
(Q_{m+1})_{ij}\,dxdt
		\\
		\leq&-\frac{3\alpha}8 \int_0^s\int_{\R^3} |\D^3 Q_{m+1}|^2\,dx +C\int_0^s\int_{\R^3}|\D Q_{m+1}|^2 (|\D^2 Q_m|^2+|\D Q_m|^4)\,dxdt
		\\
		\leq&-\frac{3\alpha}8 \int_0^s\int_{\R^3} |\D^3 Q_{m+1}|^2\,dx+CM_1^2(\eta_1 +s+\frac{s}{\eta_1^3}),
		\alabel{fE m}
		\end{align*}
where we used the argument of \eqref{key est} in the last calculation.

Using the argument in \eqref{key est} again, we obtain
\begin{align*}
	&\int_0^s\int_{\R^3}|\D\( \p_Qf_E(Q_m,\D Q_m)-v_m \cdot\D Q_{m}+g_B(Q_m)\)|^2\,dxdt
	\\
	\leq& C\int_0^s\int_{\R^3}|\D Q_m|^2|\D Q_m|^4+ |\D Q_m|^2|\D^2 Q_m|^2+|\D v_m|^2|\D Q_{m}|^2\,dxdt
	\\
	&+C\int_0^s\int_{\R^3}|v_m|^2|\D^2 Q_{m}|^2+|\D Q_m|^2 \,dxdt \leq CM_1^2(\eta_1 +s+\frac{s}{\eta_1^3})+CM_1s
	.\alabel{D2 Q+1 r}
\end{align*}
In view of  \eqref{Omega m2}-\eqref{D2 Q+1 r},  we deduce \eqref{D2 Q+1} to
	\begin{align*}
		&\frac12 \int_{\R^3}|\D^2 Q_{m+1}|^2(\cdot,s)\,dx+\frac{\alpha}{4}\int_0^s \int_{\R^3} |\D^3 Q_{m+1}|^2\,dxdt
		\\
		\leq&\frac12 \int_{\R^3}|\D^2 Q_0|^2\,dx+C \int_0^s\int_{\R^3} |\D^2 v_{m+1}|^2\,dxdt +CM_1^2(\eta_1 +s+\frac{s}{\eta_1^3})+CM_1s
		.\alabel{D2 Q m}
	\end{align*}
In order to estimate the $L^2$-norm of $\D \p_t Q_{m+1}$, we multiply \eqref{Q +1} by $\Delta \p_t Q_{m+1}$ and compute
\begin{align*}
	&\<\(\p_tQ_{m+1}+[Q_m, \Omega_{m+1}]-\D_\beta\p_{p^\beta}f_E(Q_m,\D Q_{m+1})\),\Delta \p_t Q_{m+1}\>
	\\
	=&\<\p_Qf_E(Q_m,\D Q_m)-v_m \cdot\D Q_{m}+g_B(Q_m),\Delta \p_t Q_{m+1}\>
	.\alabel{D pt Q m}
\end{align*}
Using a similar argument to \eqref{Omega m2}, we have
\begin{align*}
	&\int_0^s\int_{\R^3}\<[Q_m, \Omega_{m+1}] ,\Delta \p_t Q_{m+1}\>\,dxdt
	\\
	\leq& \frac{1}{8} \int_0^s\int_{\R^3} |\D \p_t Q_{m+1}|^2\,dxdt+C\int_0^s\int_{\R^3} |\D [Q_m, \Omega_{m+1}] |^2\,dxdt
	\\
	\leq& \frac{1}{8} \int_0^s\int_{\R^3} |\D \p_t Q_{m+1}|^2\,dxdt+CK^2\int_0^s\int_{\R^3} |\D^2 v_{m+1}|^2\,dxdt   +CM_1^2(\eta_1
+s+\frac{s}{\eta_1^3})
	.\alabel{Omega m t}
\end{align*}
In view of \eqref{sec2 f_E}, we compute the third term in \eqref{D pt Q m} in the following
\begin{align*}
	&\int_0^s\int_{\R^3}\<\(-\D_\beta\p_{p^\beta}f_E(Q_m,\D Q_{m+1})\),\Delta\p_t Q_{m+1}\>\,dxdt
	\\
	=&-\int_0^s\int_{\R^3}\p^2_{p^\beta_{ij}p^\gamma_{kl}}f_E(Q_m,\D Q_{m+1})\D^2_{\gamma\nu}( Q_{m+1})_{kl}\p_t\D^2_{\beta\nu}(
Q_{m+1})_{ij}\,dxdt
	\\
	&-\int_0^s\int_{\R^3}\p^2_{p^\beta_{ij}Q{kl}}f_E(Q_m,\D Q_{m+1})\D_{\nu}( Q_{m+1})_{kl}\p_t\D^2_{\beta\nu}( Q_{m+1})_{ij}\,dxdt
	\\
	\leq&-\frac12\int_0^s\frac{d}{dt}\int_{\R^3}\p^2_{p^\beta_{ij}p^\gamma_{kl}}f_E(Q_m,\D Q_{m+1})\D^2_{\gamma\nu}( Q_{m+1})_{kl}
\D^2_{\beta\nu}( Q_{m+1})_{ij}\,dxdt
	\\
	&+C\int_0^s\int_{\R^3}|\D Q_{m+1}|\(|\D \p_t Q_m||\D^2 Q_{m+1}|+|\p_t Q_m||\D^3 Q_{m+1}|\)\,dxdt
	\\
	&+C\int_0^s\int_{\R^3}|\D Q_{m+1}|\(|\D Q_m||\p_t Q_m||\D^2 Q_{m+1}|\)\,dxdt
	\\
	&+C\int_0^s\int_{\R^3}|\p_t \D  Q_{m+1}|\(|\D^2 Q_{m+1}||\D Q_m|+|\D Q_{m+1}||\D Q_m|^2\)\,dxdt
	\\
	&+C\int_0^s\int_{\R^3}|\p_t \D  Q_{m+1}||\D Q_{m+1}||\D^2 Q_m|\,dxdt
	\\
	\leq&- \frac\alpha4  \int_{\R^3}|\D^2 Q_{m+1}|^2(\cdot,s)\,dx+\frac{\Lambda(1+4K^2)}2\int_{\R^3}|\D^2 Q_0|^2\,dx +CM_1^2(\eta_1+s+\frac{s}{\eta_1^3})
	\\
	&+CM_1^5s+\int_0^s\int_{\R^3}\frac{\alpha}{8}|\D^3 Q_{m+1}|^2+\frac{1}{8}|\D \p_t Q_{m+1}|^2\,dxdt,
	\alabel{pt Q+1}
\end{align*}
where in the last step, we used the argument in \eqref{key est} and the following estimate
\begin{align*}
	&C\int_0^s\int_{\R^3}|\D Q_{m+1}|^2 |\D^2 Q_{m+1}|^2\,dxdt\leq \frac{\alpha}{16}\int_0^s\int_{\R^3}|\D^3 Q_{m+1}|^2\,dxdt+CM_1^5s.
\end{align*}
We apply  \eqref{Omega m t}, \eqref{pt Q+1} and \eqref{D2 Q+1 r} to obtain
\begin{align*}
	&\frac\alpha4 \int_{\R^3}|\D^2 Q_{m+1}|^2(\cdot,s)\,dx+\frac1{2} \int^s_0\int_{\R^3} |\D \p_t Q_{m+1}|^2\,dxdt
	\\
	\leq&\frac{\Lambda(1+4K^2)}2 \int_{\R^3}|\D^2 Q_0|^2\,dx+\frac{\alpha}{4}\int^s_0\int_{\R^3}|\D^3 Q_{m+1}|^2\,dxdt
	\\
	&+C\int^s_0\int_{\R^3} |\D^2 v_{m+1}|^2\,dxdt+CM_1^2(\eta_1 +s+\frac{s}{\eta_1^3})+CM_1^5s.
	\alabel{D t Q final m}
\end{align*}
Adding \eqref{D t Q final m} to \eqref{D2 Q m}, we have
\begin{align*}
	&\int_{\R^3}|\D^2 Q_{m+1}|^2(\cdot,s)\,dx+\int^s_0\int_{\R^3}|\D^3 Q_{m+1}|^2+|\D \p_t Q_{m+1}|^2\,dxdt
	\\
	\leq&C\int_{\R^3}|\D^2 Q_0|^2\,dx+C\int_{\R^3} |\D^2 v_{m+1}|^2\,dx +CM_1^2(\eta_1 +s+\frac{s}{\eta_1^3})+CM_1^5s+CM_1s
	.\alabel{Q m final}
\end{align*}
Here $C$ only depends on the following constants $\alpha, K$ and $\Lambda$.

	To estimate $\D^2 v_{m+1}$ in \eqref{Q m final}, we multiply \eqref{v +1} by $-\Delta v_{m+1}$ and compute
	\begin{align*}
		&\frac12\int_{\R^3} |\D v_{m+1}|^2(\cdot,s) \,dx+\int^s_0\int_{\R^3}|\D^2v_{m+1}|^2\,dxdt
		\\
		&-\int^s_0\int_{\R^3}[Q_m,h(Q_m,Q_{m+1})]_{ij}\D_j\Delta (v_{m+1})_i\,dxdt
		\\
		=&\frac12\int_{\R^3} |\D v_0|^2 \,dx-\int^s_0\int_{\R^3}\Big((v_m)_j\D_j (v_m)_i+\D_j\sigma_{ij}(Q_m,\D Q_m)\Big)(\Delta v_{m+1})_i
\,dxdt
		\\
		\leq& \frac12\int_{\R^3} |\D v_0|^2 \,dx+\frac14\int^s_0\int_{\R^3}|\D^2 v_{m+1}|^2\,dxdt+CM_1^2(\eta_1 +\frac{s}{\eta_1^3})
		.\alabel{D v +1}
	\end{align*}
	To cancel the term involving $h(Q_m,Q_{m+1})$ in \eqref{D v +1}, we differentiate \eqref{Q +1} in $x$, multiply  by $\D h(Q_m,Q_{m+1})$
and obtain
	\begin{align*}
		&\int_{\R^3}\<\D_\beta\(\p_t Q_{m+1}+[Q_m, \Omega_{m+1}]\),\D_\beta h(Q_m,Q_{m+1})\>\,dx +\int_{\R^3}|\D h(Q_m,Q_{m+1})|^2\,dx
		\\
		=&\int_{\R^3}\<\D_\beta\(-v_m\cdot \D Q_m +g_B(Q_m)\), \D_\beta h(Q_m,Q_{m+1})\>\,dx
		.\alabel{fE +1}
	\end{align*}
Choosing  $A=Q,B=h(Q_m,Q_{m+1}), F=\Delta\D v$ in  Lemma \ref{Lie}, we obtain
\begin{align*}
	\<[Q_m,\Delta\Omega_{m+1}],h(Q_m,Q_{m+1})\>=\<\Delta\D v_{m+1},[Q_m,h(Q_m,Q_{m+1})]\>
	.\alabel{Lie h}
\end{align*}
Note that \[h(Q_m,Q_{m+1})\leq C(|\D^2 Q_{m+1}|+|\D Q_{m+1}|^2+|\D Q_{m+1}||\D Q_m|).\] Then using \eqref{Lie h}, we compute the second term
in \eqref{fE +1} to  get
\begin{align*}
	&\int_{\R^3}\<\D_\beta[Q_m, \Omega_{m+1}],\D_\beta h(Q_m,Q_{m+1})\>\,dx
	\\
	=&\int_{\R^3}\<[Q_m, \Delta\Omega_{m+1}], h(Q_m,Q_{m+1})\>\,dx
	\\
	&+\int_{\R^3}\<[\Delta Q_m, \Omega_{m+1}]+2[\D Q_m, \D \Omega_{m+1}], h(Q_m,Q_{m+1})\>\,dx
	\\
	=&\int_{\R^3}\<[Q_m, \Delta\Omega_{m+1}], h(Q_m,Q_{m+1})\>\,dx +\int_{\R^3}\<[\D Q_m, \D \Omega_{m+1}], h(Q_m,Q_{m+1})\>\,dx
	\\
	&-\int_{\R^3}\<[\D_\alpha Q_m, \Omega_{m+1}],\D_\alpha h(Q_m,Q_{m+1})\>\,dx
	\\
	\geq&\int_{\R^3}\<\Delta\D v_{m+1},[Q_m,h(Q_m,Q_{m+1})]\>\,dx -\frac14 \int_{\R^3} |\D^2 v_{m+1}|^2\,dx
	\\
	&-\frac12 \int_{\R^3} |\D  h(Q_m,Q_{m+1})|^2\,dx
	-\eta_1\int_{\R^3}|\D^3 Q_m|^2\,dx-CM_1^2 \frac{s}{\eta_1^3}
	.\alabel{lie m}
\end{align*}
We repeat the argument in  \eqref{pt Q+1} for the first term in \eqref{fE +1}, apply Young's inequality to the right-hand side of \eqref{fE
+1}. Then integrate  \eqref{fE +1} in $t$ and combine  with \eqref{D v +1} yield
\begin{align*}
	&\frac12 \int_{\R^3}|\D v_{m+1}(\cdot,s)|^2 \,dx+\frac12\int^s_0\int_{\R^3}|\D^2v_{m+1}|^2\,dxdt
	\\
	\leq&\frac12\int_{\R^3}|\D v_0|^2\,dx+\frac{\Lambda(1+4K^2)}2\int_{\R^3}|\D^2 Q_0|^2\,dx+CM_1^2(\eta_1 +\frac{s}{\eta_1^3})+C(\eta_2)M_1^5s
	\\
	&+ \eta_2\int^s_0\int_{\R^3}|\D^3 Q_{m+1}|^2+|\D \p_t Q_{m+1}|^2\,dxdt
	+C(\eta_2)M_1^2(\eta_1 +s+\frac{s}{\eta_1^3})
	\alabel{D v final s1}
\end{align*}
for some small $\eta_1,\eta_2$. Substituting \eqref{Q m final} into \eqref{D v final s1} and choosing $\eta_2$ sufficiently small, we obtain
the estimates for $v_{m+1}$. Combining the resulting expression with \eqref{Q m final} yields
\begin{align*}
	& \int_{\R^3}(|\D^2 Q_{m+1}|^2+|\D v_{m+1}|^2)(\cdot,s) \,dx
	\\
	&+ \int^s_0\int_{\R^3}|\D^3 Q_{m+1}|^2+|\D \p_t Q_{m+1}|^2+|\D^2v_{m+1}|^2\,dxdt
	\\
	\leq&CM_1+CM_1^2(\eta_1 +s+\frac{s}{\eta_1^3})+CM_1^5s+CM_1s
	.\alabel{D2 final s2}
\end{align*}
Here $C$ only depends on $\alpha,K$ and $ \Lambda$.

It remains to check the $L^2$-norm of the lower order terms in $\mathcal V(0,T)$.  We multiply \eqref{v +1} by $v_{m+1}$ to obtain
\begin{align*}
	&\frac12\int_{\R^3}|v_{m+1}|^2(\cdot,s) \,dx + \frac12 \int^s_0\int_{\R^3}|\D v_{m+1}|^2\,dxdt
	\\
	\leq&\frac12\int_{\R^3}|v_0|^2 \,dx  +C\int^s_0\int_{\R^3}|[Q_{m},h(Q_m,Q_{m+1})]|^2+|\sigma_{ij}(Q_m,\D Q_m)|^2\,dxdt
	\\
	\leq&CM_1+C\int^s_0\int_{\R^3}|\D^2 Q_{m+1}|^2+|\D Q_{m+1}|^4+|\D Q_m|^4\,dxdt
	\leq CM_1+CM_1s
	.\alabel{v low}
\end{align*}
By using the mean value theorem with the fact that $g_B(Q_e)=0$, we find
\begin{align*}
	|g_B(Q_m)|\leq C(K)|Q_m-Q_e|.
	\alabel{3}
\end{align*}
Multiplying \eqref{Q +1} by $\p_t Q_{m+1}$ and $Q_{m+1}-Q_e$ respectively and then using \eqref{3}  yield
\begin{align*}
	&\frac12 \int_{\R^3}|(Q_{m+1}-Q_e)|^2(\cdot,s)\,dx+\frac12\int^s_0\int_{\R^3}| \p_t Q_{m+1}|^2\,dxdt
	\\
	\leq&\frac12\int_{\R^3}|Q_0-Q_e|^2\,dx+\frac12 \int^s_0\int_{\R^3}|Q_{m+1}-Q_e |^2\,dxdt
	\\
	&+ C\int^s_0\int_{\R^3}|[Q_{m}, \Omega_{m+1}]|^2+|h(Q_m,Q_{m+1})|^2+|v_m \cdot\D Q_{m}|^2+|g_B(Q_m)|^2 dxdt
	\\
	\leq& CM_1+CM_1s
	.\alabel{Q m+1 low}
\end{align*}
Note that the estimate of $L^2$-norm of $\D Q_{m+1}$ comes from \eqref{D2 final s2} and \eqref{Q m+1 low}. Now adding \eqref{v low} and \eqref{Q m+1 low} to
\eqref{D2 final s2}, we have
\begin{align*}
	& \big(\| Q (s)\|_{H^2_{Q_e}(\R^3)}^2+\|v (s)\|_{H^1(\R^3)}^2 \big)+\|\D^3 Q \|_{L^2(0,T ;L^2(\R^3))}^2
	\\
	&+\|\p_t Q \|_{L^2(0,T ;H^1(\R^3))}^2 +\|\D^2 v \|_{L^2(0,T_0;L^2(\R^3))}^2
	\\
	\leq&\frac{C_1}4M_1+\frac{C_1}4M_1^2(\eta_1 +s+\frac{s}{\eta_1^3})+\frac{C_1}4M_1^5s+\frac{C_1}4M_1s \leq C_1M_1
	\alabel{Q v m+1}
\end{align*}
for some $C_1$ depending on $\alpha, K$ and $\Lambda$. Here in the last step, we set $\eta_1=M_1^{-1}$ and $s\leq
\min\left\{\frac12M_1^{-4},\frac12M_1^{-1},\frac12\right\}$.

It remains to verify that $\| Q_{m+1} (s)\|_{L^\infty(\R^3)}\leq 2K$. Note from \eqref{Q m+1 low}  that
\begin{align*}
	\int_{\R^3}|Q_{m+1}(\cdot,s)-Q_0|^2\,dx=& \int_{\R^3}\(\int_0^s \p_t Q_{m+1}\,dt\)^2\,dx\leq s\int_0^s  \int_{\R^3}|\p_t Q_{m+1}|^2\,dt\,dx
	\\
	\leq& s \int_{\R^3}\int_0^s |\p_t Q_{m+1}|^2\,dxdt\leq CM_1s(1+s)\leq CM_1s
\end{align*}for $s\leq \frac12$. By using the Gagliardo–Nirenberg interpolation (cf. \cite{FHM}) and choosing $s\leq  C_2^{-8} K^8M_1^{-4} $,
we have
\begin{align*}
	\|Q_{m+1}(s)-Q_0\|_{L^\infty(\R^3)}\leq& C\| Q_{m+1}(s)-Q_0 \|_{L^2(\R^3)}^{\frac14}\|\D^2( Q_{m+1}(s)-Q_0)\|_{L^2(\R^3)}^{\frac34}
	\\
	\leq& C_2(M_1s )^{\frac18} M_1 ^{\frac38}\leq K,
\end{align*}
where $C_2$ is independent from $m$. Therefore, we prove Claim 1 by choosing
\[T_{M_1}:=\min\left\{ C_2^{-8} K^8M_1^{-4},\frac12M_1^{-4},\frac12M_1^{-1},\frac12\right\}.\]

\noindent{\bf Claim 2}: There exists a uniform $T>0$ with $T\leq T_{M_1}$ such that
\begin{align*}
	&\sup_{0\leq t\leq T}\big(\| Q_{m+1}- Q_{m}(t)\|_{H^1_{Q_e}(\R^3)}+\|v_{m+1}- v_{m} (t)\|_{L^2(\R^3)} \big)
	\\
	&+\|\D^2 (Q_{m+1}- Q_{m}) \|_{L^2(0,T ;L^2(\R^3))} +\|\D (v_{m+1}- v_{m}) \|_{L^2(0,T;L^2(\R^3))}
	\\
	\leq&\frac12\sup_{0\leq t\leq T}\big(\| Q_{m}- Q_{m-1}(t)\|_{H^1_{Q_e}(\R^3)}+\|v_{m}- v_{m-1} (t)\|_{L^2(\R^3)}\big)
	\\
	&+\frac12\|\D^2 (Q_{m}- Q_{m-1}) \|_{L^2(0,T ;L^2(\R^3))}  +\frac12\|\D (v_{m}- v_{m-1}) \|_{L^2(0,T;L^2(\R^3))}
\end{align*}
For given pairs $(Q_m,v_m)$ and $(Q_{m-1},v_{m-1})\in \mathcal V$, we have
	\begin{align*}
		&(\p_t -\Delta)(v_{m+1}-v_{m}) +\D (P_{m+1}-P_{m})
		\\
		=& \D\cdot[Q_m,h(Q_{m},Q_{m+1} )]-\D\cdot[Q_{m-1},h(Q_{m-1}, Q_m )]
		\\
		&-v_m\cdot\D v_m+v_{{m-1}}\cdot\D v_{m-1}+ \sigma(Q_m,\D Q_m))-\sigma(Q_{{m-1}},\D Q_{{m-1}})),
		\alabel{v +2}
		\\[2ex]
		&\D\cdot (v_{m+1}-v_{m})=0,\alabel{div v +2}
		\\[2ex]
		&\p_t(Q_{m+1}-Q_{m}) +[Q_m, \Omega_{m+1}]-[Q_{m-1}, \Omega_{m}]
		\\
		=&\D_\beta\p_{p^\beta}f_E(Q_m,\D Q_{m+1})-\D_\beta\p_{p^\beta}f_E(Q_{m-1},\D Q_{m})-v_m\cdot \D Q_m+v_{m-1}\cdot \D Q_{m-1}
		\\
		&-\p_Qf_E(Q_m,\D Q_m)+\p_Qf_E(Q_{m-1},\D Q_{m-1}) +g_B(Q_m)-g_B(Q_{m-1})
		.\alabel{Q +2}
	\end{align*}
	Multiplying \eqref{Q +2} by $-\Delta(Q_{m+1}-Q_{m})$  yields
	\begin{align*}
		&\frac12\int_{\R^3}|\D(Q_{m+1}-Q_m)|^2(\cdot,s)\,dx
		\\
		=&\int_0^s\int_{\R^3}\<[Q_m, \Omega_{m+1}]-[Q_{m-1}, \Omega_{m}],\Delta(Q_{m+1}-Q_{m})\>\,dxdt
		\\
		&-\int_0^s\int_{\R^3}\<\D_\beta\p_{p^\beta}f_E(Q_m,\D Q_{m+1})-\D_\beta\p_{p^\beta}f_E(Q_{m-1},\D Q_{m}),\Delta(Q_{m+1}-Q_{m})\>\,dxdt
		\\
		&-\int_0^s\int_{\R^3}\<-v_m\cdot \D Q_m+v_{m-1}\cdot \D Q_{m-1},\Delta(Q_{m+1}-Q_{m})\>\,dxdt
		\\
		&-\int_0^s\int_{\R^3}\<-\p_Qf_E(Q_m,\D Q_m)+\p_Qf_E(Q_{m-1},\D Q_{m-1}),\Delta(Q_{m+1}-Q_{m})\>\,dxdt
		\\
		&-\int_0^s\int_{\R^3}\<g_B(Q_m)-g_B(Q_{m-1}),\Delta(Q_{m+1}-Q_{m})\>\,dxdt
		.\alabel{Q d}
	\end{align*}

Using Young's inequality and \eqref{L4}, we compute the first term in the right-hand side of \eqref{Q d}
\begin{align*}
	&\int_0^s\int_{\R^3}\<[Q_m, \Omega_{m+1}]-[Q_{m-1}, \Omega_{m}],\Delta(Q_{m+1}-Q_{m})\>\,dxdt
	\\
	\leq &\eta\int_0^s\int_{\R^3}|\D^2(Q_{m+1}-Q_{m})|^2\,dxdt +C(\eta)\int_0^s\int_{\R^3}|\D (v_{m+1}-v_{m})|^2\,dxdt
	\\
	&+C(\eta)\int_0^s\|Q_{m}-Q_{m-1}\|_{L^\infty(\R^3)}^2\int_{\R^3}|v_{m}|^2\,dxdt
	\\
		\leq &\eta\int_0^s\int_{\R^3}|\D^2(Q_{m+1}-Q_{m})|^2\,dxdt +C(\eta)\int_0^s\int_{\R^3}|\D (v_{m+1}-v_{m})|^2\,dxdt
	\\
	&+C(\eta)M_1\int_0^s\(\int_{\R^3}|\D(Q_{m}-Q_{m-1})|^4\,dx\)^{\frac12}dt
	\\
	\leq &\eta\int_0^s\int_{\R^3}|\D^2(Q_{m+1}-Q_{m})|^2\,dxdt +C\int_0^s\int_{\R^3}|\D (v_{m+1}-v_{m})|^2\,dxdt
	\\
	&+\eta_1\int_0^s\int_{\R^3}|\D^2(Q_{m}-Q_{m-1})|^2\,dxdt+CM_1^4\int_0^s\int_{\R^3}|\D(Q_{m}-Q_{m-1})|^2\,dxdt
	,\alabel{Q d1}
\end{align*}
where $\eta$ and $\eta_1$ are some small constants to be chosen later.

Applying \eqref{L4} again to the second term in \eqref{Q d} and using \eqref{sec2 f_E} yields
\begin{align*}
	&-\int_0^s\int_{\R^3}\<\D_\beta\p_{p^\beta}f_E(Q_m,\D Q_{m+1})-\D_\beta\p_{p^\beta}f_E(Q_{m-1},\D Q_{m}),\Delta(Q_{m+1}-Q_{m})\>\,dxdt
	\\
	=&-\int_0^s\int_{\R^3}\p^2_{p^\beta_{ij} p^\nu_{kl}}f_E(Q_m,\D
Q_{m+1})\D^2_{\gamma\nu}(Q_{m+1})_{kl}\D^2_{\beta\gamma}(Q_{m+1}-Q_{m})_{ij} \,dxdt
	\\
	&+\int_0^s\int_{\R^3}\p^2_{p^\beta_{ij} p^\nu_{kl}}f_E(Q_{m-1},\D Q_{m})\D^2_{\gamma\nu}(Q_{m})_{kl}\D^2_{\beta\gamma}(Q_{m+1}-Q_{m})_{ij}
\,dxdt
	\\
	&-\int_0^s\int_{\R^3}\p^2_{p^\beta_{ij} Q_{kl}}f_E(Q_{m},\D Q_{m+1})\D_\gamma( Q_m)_{kl}\D^2_{\beta\gamma}(Q_{m+1}-Q_{m})_{ij}\,dxdt
	\\
	&+\int_0^s\int_{\R^3}\p^2_{p^\beta_{ij} Q_{kl}}f_E(Q_{m-1},\D Q_{m})\D_\gamma( Q_{m-1})_{kl}\D^2_{\beta\gamma}(Q_{m+1}-Q_{m})_{ij}\,dxdt
	\\
	\leq&-(\frac{\alpha}2-2\eta)\int_0^s\int_{\R^3}|\D^2 (Q_{m+1}-Q_{m})|^2 \,dxdt
	\\
	&-\int_0^s\int_{\R^3}\p^2_{p^\beta_{ij} p^\nu_{kl}}f_E(Q_m,\D Q_{m+1})\D^2_{\gamma\nu}(Q_{m})_{kl}\D^2_{\beta\gamma}(Q_{m+1}-Q_{m})_{ij}
\,dxdt
	\\
	&+\int_0^s\int_{\R^3}\p^2_{p^\beta_{ij} p^\nu_{kl}}f_E(Q_{m-1},\D Q_{m})\D^2_{\gamma\nu}(Q_{m})_{kl}\D^2_{\beta\gamma}(Q_{m+1}-Q_{m})_{ij}
\,dxdt
	\\
	&+C(\eta)\int_0^s\int_{\R^3}|\p^2_{p Q}f_E(Q_{m},\D Q_{m+1})\D Q_m-\p^2_{p Q}f_E(Q_{m-1},\D Q_{m})\D Q_{m-1}|^2\,dxdt
	\\
	\leq&-(\frac{\alpha}2-2\eta)\int_0^s\int_{\R^3}|\D^2 (Q_{m+1}-Q_{m})|^2 \,dxdt
	\\
	&+C(\eta)\int_0^s\(\|Q_m-Q_{m-1}\|_{L^\infty(\R^3)}^2\)\int_{\R^3}|\D^2 Q_{m}|^2  \,dx dt
	\\
	&+C(\eta) \int_0^s\int_{\R^3}(|\D Q_{m+1}-\D Q_{m} |^2+|\D Q_{m}|^2|Q_m-Q_{m-1}|^2)|\D Q_m|^2\,dxdt
	\\
	&+C(\eta) \int_0^s\int_{\R^3} |\D Q_{m}|^2|\D Q_{m}- \D Q_{m-1}|^2\,dxdt
	\\
	\leq&-(\frac{\alpha}2-2\eta)\int_0^s\int_{\R^3}|\D^2 (Q_{m+1}-Q_{m})|^2 \,dxdt+\eta_1\int_0^s\int_{\R^3}|\D^2(Q_{m+1}-Q_{m})|^2\,dxdt
	\\
	&+C(M_1+M_1^4)\int_0^s \int_{\R^3}|\D(Q_{m}-Q_{m-1})|^2\,dxdt
	.\alabel{Q d2}
\end{align*}
The remaining terms in \eqref{Q d} are
\begin{align*}
	&-\int_0^s\int_{\R^3}\<-v_m\cdot \D Q_m+v_{m-1}\cdot \D Q_{m-1},\Delta(Q_{m+1}-Q_{m})\>\,dxdt
	\\
	&-\int_0^s\int_{\R^3}\<-\p_Qf_E(Q_m,\D Q_m)+\p_Qf_E(Q_{m-1},\D Q_{m-1}),\Delta(Q_{m+1}-Q_{m})\>\,dxdt
	\\
	&-\int_0^s\int_{\R^3}\<g_B(Q_m)-g_B(Q_{m-1}),\Delta(Q_{m+1}-Q_{m})\>\,dxdt
	\\
	\leq&\eta\int_0^s\int_{\R^3}|\D^2 (Q_{m+1}-Q_{m})|^2 \,dxdt+C\int_0^s\int_{\R^3}|v_m|^2| \D Q_m-\D Q_{m-1}|^2\,dxdt
	\\
	&+C\int_0^s\int_{\R^3} |v_m-v_{m-1}|^2 |\D Q_{m-1}|^2+|\D Q_m|^4|Q_m-Q_{m-1}|^2\,dxdt
	\\
	&+C\int_0^s\int_{\R^3} \(|\D Q_{m-1}|^2+|\D Q_{m}|^2\)|\D (Q_m-Q_{m-1})|^2+|Q_m-Q_{m-1}|^2\,dxdt
	\\
	\leq&\eta\int_0^s\int_{\R^3}|\D^2(Q_{m+1}-Q_{m})|^2\,dxdt +C\int_0^s \int_{\R^3}|Q_m-Q_{m-1}|^2\,dxdt
	\\
	&+\eta_1\int_0^s\int_{\R^3}|\D^2 (Q_{m}-Q_{m-1})|^2 +|\D(v_m-v_{m-1})|^2\,dxdt
	\\
	&+C(M_1+M_1^4)\int_0^s \int_{\R^3}|\D(Q_{m}-Q_{m-1})|^2+|v_m-v_{m-1}|^2\,dxdt
	.\alabel{Q d3}
\end{align*}
Substituting \eqref{Q d1}-\eqref{Q d3} into \eqref{Q d}, we find
\begin{align*}
	&\frac12\int_{\R^3}|\D(Q_{m+1}-Q_m)|^2(\cdot,s)\,dx+\frac{\alpha}{4}\int_0^s\int_{\R^3}|\D^2 (Q_{m+1}-Q_{m})|^2 \,dxdt
	\\
	\leq&C\int_0^s\int_{\R^3}|\D (v_{m+1}-v_{m})|^2\,dxdt +C\int_0^s \int_{\R^3}|Q_m-Q_{m-1}|^2\,dxdt
	\\
	&+C\eta_1\int_0^s\int_{\R^3}|\D^2(Q_{m}-Q_{m-1})|^2+|\D(v_m-v_{m-1})|^2\,dxdt
	\\
	&+C(M_1+M_1^4)\int_0^s\int_{\R^3}|\D(Q_{m}-Q_{m-1})|^2+|v_m-v_{m-1}|^2 \,dxdt
	.\alabel{Q d4}
\end{align*}
Now, we compute the difference $(Q_{m+1}-Q_{m})$. Multiplying $\eqref{Q +2}$ by $(Q_{m+1}-Q_{m})$, one can show
\begin{align*}
&\frac12\int_{\R^3}|(Q_{m+1}-Q_m)|^2(\cdot,s)\,dx
\\
\leq &CM_1\int_0^s\int_{\R^3}|Q_{m+1}-Q_{m}|^2\,dxdt+ C\int_0^s\int_{\R^3}|\D (v_{m+1}-v_{m})|^2\,dxdt
\\
&+C\eta_1\int_0^s\int_{\R^3}|\D^2(Q_{m+1}-Q_{m})|^2+|\D(v_m-v_{m-1})|^2\,dxdt
\\
&+C(M_1+M_1^4)\int_0^s\int_{\R^3}|\D(Q_{m}-Q_{m-1})|^2+|v_m-v_{m-1}|^2+|Q_m-Q_{m-1}|^2\,dxdt
.\alabel{Q d0}
\end{align*}
Combining \eqref{Q d0} with \eqref{Q d4}, we find
\begin{align*}
	&\int_{\R^3}\(|(Q_{m+1}-Q_m)|^2+|\D(Q_{m+1}-Q_m)|^2\)(\cdot,s)\,dx
	\\
	&+\int_0^s\int_{\R^3}|\D^2 (Q_{m+1}-Q_{m})|^2 \,dxdt
	\\
	\leq&C\int_0^s\int_{\R^3}|\D (v_{m+1}-v_{m})|^2\,dxdt+C\int_0^s\int_{\R^3}|Q_{m+1}-Q_{m}|^2\,dxdt
	\\
	&+C(M_1+M_1^4)\int_0^s\int_{\R^3}|\D(Q_{m}-Q_{m-1})|^2+|v_m-v_{m-1}|^2+|Q_m-Q_{m-1}|^2\,dxdt
	\\
	&+C\eta_1\int_0^s\int_{\R^3}|\D^2(Q_{m+1}-Q_{m})|^2+|\D(v_m-v_{m-1})|^2\,dxdt
	.\alabel{Q d final}
\end{align*}
Next we compute the difference involving $v_m$. Multiplying \eqref{v +2} by $(v_{m+1}-v_{m})$, we have
	\begin{align*}
		&\frac12\int_{\R^3}|(v_{m+1}-v_m)|^2(\cdot,s)\,dx +\frac34\int_0^s\int_{\R^3}|\D(v_{m+1}-v_{m})|^2\,dxdt
		\\
		\leq&C\int_0^s\int_{\R^3}|\sigma(Q_m,\D Q_m))-\sigma(Q_{{m-1}},\D Q_{{m-1}})|^2\,dxdt
		\\
		& +\int_0^s\int_{\R^3}\<\D\cdot[Q_m,h(Q_{m},Q_{m+1} )]-\D\cdot[Q_{m-1},h(Q_{m-1}, Q_m )], v_{m+1}-v_{m} \>\,dx
		\\
		&+\int_0^s\int_{\R^3}\<-v_m\cdot\D v_m+v_{{m-1}}\cdot\D v_{m-1},v_{m+1}-v_{m}\>\,dxdt
		.\alabel{v d}
	\end{align*}
Using \eqref{L4}, we find
\begin{align*}
	&C\int_0^s\int_{\R^3}|\sigma(Q_m,\D Q_m)-\sigma(Q_{{m-1}},\D Q_{{m-1}})|^2\,dxdt
	\\
	\leq& C\int_0^s\int_{\R^3}|\D(Q_m-Q_{m-1})|^2\(|\D Q_{m}|^2+|\D Q_{m-1}|^2\)+|Q_m-Q_{m-1}|^2|\D Q_{m-1}|^4\,dxdt
	\\
	\leq& \eta_1\int_0^s\int_{\R^3}|\D^2(Q_m-Q_{m-1})|^2dxdt+C(M_1+M_1^4)\int_0^s\int_{\R^3}|\D(Q_m-Q_{m-1})|^2dxdt
	.\alabel{v d1}
\end{align*}
Applying \eqref{L4} to the last term in \eqref{v d}, we have
	\begin{align*}
		&\int_0^s\int_{\R^3}\<-v_m\cdot\D v_m+v_{{m-1}}\cdot\D v_{m-1},v_{m+1}-v_{m}\>\,dxdt
		\\
		&=\int_0^s\int_{\R^3}\(-v_m\cdot\D( v_m-v_{m-1})+(v_{m-1}-v_{m})\cdot\D v_{m-1}\)_j(v_{m+1}-v_{m})_j\,dxdt
		\\
		&\leq C\int_0^s\int_{\R^3}|v_m||\D (v_m-v_{m-1})||v_{m+1}-v_{m}|\,dxdt
		\\
		&+C\int_0^s\int_{\R^3}|v_{m-1}|(|\D (v_m-v_{m-1})||v_{m+1}-v_{m}|+|v_m-v_{m-1}||\D (v_{m+1}-v_{m})|)\,dxdt
		\\
		&\leq \eta_1\int_0^s\int_{\R^3}|\D (v_m-v_{m-1})|^2\,dxdt+\frac14\int_0^s\int_{\R^3}|\D (v_{m+1}-v_m)|^2\,dxdt
		\\
		&+CM_1^4\int_0^s \int_{\R^3}|v_{m+1}-v_{m}|^2+|v_m-v_{m-1}|^2\,dxdt
		.\alabel{v d2.0}
	\end{align*}
Thus we can write \eqref{v d} as
\begin{align*}
	&\frac12\int_{\R^3}|(v_{m+1}-v_m)|^2(\cdot,s)\,dx+\frac12\int_0^s\int_{\R^3}|\D(v_{m+1}-v_{m})|^2\,dxdt
	\\
	\leq&3\eta_1\int_0^s\int_{\R^3}|\D^2(Q_m-Q_{m-1})|^2+|\D (v_m-v_{m-1})|^2 \,dxdt
	\\
	&+CM_1^4\int_0^s\int_{\R^3}|\D(Q_m-Q_{m-1})|^2+|v_{m+1}-v_{m}|^2 +|v_m-v_{m-1}|^2\,dxdt
	\\
	& +\int_0^s\int_{\R^3}\<\D\cdot[Q_m,h(Q_{m},Q_{m+1} )]-\D\cdot[Q_{m-1},h(Q_{m-1}, Q_m )],(v_{m+1}-v_{m})\>\,dx
	.\alabel{v d2}
\end{align*}
It follows from  Lemma \ref{Lie} with the substitution
$A=Q_m,B=h(Q_{m},Q_{m+1}), F=\Omega_{m+1}$ and the other three cases that
\begin{align*}
	&\< [Q_m,\Omega_{m+1}]-[Q_{m-1},\Omega_{m}],h(Q_{m},Q_{m+1} )-h(Q_{m-1}, Q_m ))\>
	\\
	=&\< [Q_m,h(Q_{m},Q_{m+1} )]-[Q_{m-1},h(Q_{m-1}, Q_m )],\D(v_{m+1}-v_{m})\>
	.\alabel{lie h-}
\end{align*}
Multiplying \eqref{Q +2} by
$\big(h(Q_{m},Q_{m+1})-h(Q_{m-1}, Q_m)\big)$ and using \eqref{lie h-}, we obtain
\begin{align*}
	&\int_0^s\int_{\R^3}\<\D\cdot[Q_m,h(Q_{m},Q_{m+1} )]-\D\cdot[Q_{m-1},h(Q_{m-1}, Q_m )],(v_{m+1}-v_{m})\>\,dxdt
	\\
	&+\int_0^s\int_{\R^3}|h(Q_{m-1}, Q_m)-h(Q_{m},Q_{m+1})|^2\,dxdt
	\\
	=&\int_0^s\int_{\R^3}\<\p_t(Q_{m+1}-Q_{m}),h(Q_{m},Q_{m+1})-h(Q_{m-1}, Q_m)\>\,dxdt
	\\
	&-\int_0^s\int_{\R^3}\<-v_m\cdot \D Q_m+v_{m-1}\cdot \D Q_{m-1},h(Q_{m},Q_{m+1})-h(Q_{m-1}, Q_m)\>\,dxdt
	\\
	&-\int_0^s\int_{\R^3}\<g_B(Q_m)-g_B(Q_{m-1}),h(Q_{m},Q_{m+1})-h(Q_{m-1}, Q_m)\>\,dxdt
	\\
	\leq&\frac12\int_0^s\int_{\R^3}|h(Q_{m-1}, Q_m)-h(Q_{m},Q_{m+1})|^2\,dxdt+\eta_3\int_0^s\int_{\R^3}|\p_t(Q_{m+1}-Q_{m})|^2\,dxdt
	\\
	&+C(\eta_3)\int_0^s\int_{\R^3}|\p_Qf_E(Q_m,\D Q_m)-\p_Qf_E(Q_{m-1},\D Q_{m-1})|^2\,dxdt
		\\
	&+C(\eta_3)\int_0^s\int_{\R^3}|v_m\cdot \D Q_m-v_{m-1}\cdot \D Q_{m-1}|^2+|g_B(Q_m)-g_B(Q_{m-1})|^2\,dxdt
	\\
	&+\int_0^s\int_{\R^3}\<\p_t(Q_{m+1}-Q_{m}),\D_\beta\p_{p^\beta}f_E(Q_m,\D Q_{m+1})-\D_\beta\p_{p^\beta}f_E(Q_{m-1}\>\,dxdt
	.\alabel{h d1}
\end{align*}
In a similar calculation to \eqref{Q d2}, using \eqref{sec2 f_E}, we estimate the last term in \eqref{h d1}
\begin{align*}
	&\int_0^s\int_{\R^3}\<\p_t(Q_{m+1}-Q_{m}),\D_\beta\(\p_{p^\beta}f_E(Q_m,\D Q_{m+1})- \p_{p^\beta}f_E(Q_{m-1},\D Q_{m}\)\>\,dxdt
	\\
	=&\int_0^s\int_{\R^3}\p^2_{p^\beta_{ij} p^\nu_{kl}}f_E(Q_m,\D Q_{m+1})\D^2_{\beta\nu}(Q_{m+1})_{kl}\p_t(Q_{m+1}-Q_{m})_{ij} \,dxdt
	\\
	&-\int_0^s\int_{\R^3}\p^2_{p^\beta_{ij} p^\nu_{kl}}f_E(Q_{m-1},\D Q_{m})\D^2_{\beta\nu}(Q_{m})_{kl}\p_t(Q_{m+1}-Q_{m})_{ij} \,dxdt
	\\
	&+\int_0^s\int_{\R^3}\p^2_{p^\beta_{ij} Q_{kl}}f_E(Q_{m},\D Q_{m+1})\D_\beta( Q_m)_{kl}\p_t(Q_{m+1}-Q_{m})_{ij}\,dxdt
	\\
	&-\int_0^s\int_{\R^3}\p^2_{p^\beta_{ij} Q_{kl}}f_E(Q_{m-1},\D Q_{m})\D_\beta( Q_{m-1})_{kl}\p_t(Q_{m+1}-Q_{m})_{ij}\,dxdt
	\\
	\leq&- \frac{\alpha}4 \int_{\R^3}|\D (Q_{m+1}-Q_{m})|^2 (\cdot,s)\,dx +2\eta_3\int_0^s\int_{\R^3}|\p_t(Q_{m+1}-Q_{m})|^2 \,dxdt
	\\
	&+\eta_2 \int_0^s \int_{\R^3}|\D^2(Q_{m+1}-Q_{m})|^2\,dx dt+ CM_1^2\int_0^s \int_{\R^3}|\D (Q_{m+1}-Q_{m})|^2 \,dx dt
	\\
	&+\eta_1 \int_0^s \int_{\R^3}|\D^2(Q_{m}-Q_{m-1})|^2\,dx dt+ CM_1^4\int_0^s \int_{\R^3}|\D (Q_{m}-Q_{m-1})|^2 \,dx dt
	.\alabel{h d2}
\end{align*}
Here we used the fact from   \eqref{Q +1} that
\begin{align*}
	\int_{\R^3}|\p_t Q_m|^2\,dx \leq& C\int_{\R^3}|\D v_{m+1}|^2+|v_m|^2|\D Q_m|^2+|\D^2 Q_{m+1}|^2\,dx
	\\
	&+C\int_{\R^3}|\D Q_{m+1}|^2|\D Q_{m}|^2+|\D Q_{m}|^4+|g_B(Q_m)|^2\,dx\leq CM_1.
\end{align*} For the term $\p_t(Q_{m+1}-Q_{m})$, it follows from \eqref{Q +2} that
	\begin{align*}
		&\int_0^s\int_{\R^3}|\p_t(Q_{m+1}-Q_{m})|^2 \,dxdt
		\\
		\leq& C\int_0^s\int_{\R^3}|v_m\cdot \D Q_m-v_{m-1}\cdot \D Q_{m-1}|^2+|[Q_m, \Omega_{m+1}]-[Q_{m-1}, \Omega_{m}]|^2\,dxdt
		\\
		&+C\int_0^s\int_{\R^3}|h(Q_m, Q_{m+1})-h(Q_{m-1}, Q_{m})|^2+|g_B(Q_m)-g_B(Q_{m-1})|^2\,dxdt
		\\
		\leq&C\int_0^s\int_{\R^3}|\D^2(Q_{m+1}-Q_{m})|^2+|\D(v_{m+1}-v_{m})|^2\,dxdt
		\\
		&+C\int_0^s\int_{\R^3}|\D^2(Q_{m}-Q_{m-1})|^2+|\D(v_{m}-v_{m-1})|^2\,dxdt
		\\
		&+CM_1\int_0^s \int_{\R^3}|\D(Q_{m+1}-Q_{m})|^2+|\D(Q_{m}-Q_{m-1})|^2\,dxdt
		\\
		&+CM_1 \int_0^s \int_{\R^3}|v_m-v_{m-1}|^2+|Q_m-Q_{m-1}|^2\,dxdt
		.\alabel{h d3}
	\end{align*}
Substituting \eqref{h d2}-\eqref{h d3} into \eqref{h d1} with sufficiently small $\eta_3$, we find
\begin{align*}
	&\int_0^s\int_{\R^3}\<\D\cdot[Q_m,h(Q_{m},Q_{m+1} )]-\D\cdot[Q_{m-1},h(Q_{m-1}, Q_m )],(v_{m+1}-v_{m})\>\,dxdt
	\\
	&+\frac12\int_0^s\int_{\R^3}|h(Q_{m-1}, Q_m)-h(Q_{m},Q_{m+1})|^2\,dxdt
	\\
	 \leq&2\eta_2 \int_0^s \int_{\R^3}|\D^2(Q_{m+1}-Q_{m})|^2\,dx dt+\frac14\int_0^s\int_{\R^3} |\D(v_{m+1}-v_{m})|^2\,dxdt
	 \\
	&+C\eta_1\int_0^s\int_{\R^3}|\D^2(Q_{m}-Q_{m-1})|^2+|\D(v_{m}-v_{m-1})|^2\,dxdt
	\\
	&+C(M_1+M_1^4)\int_0^s \int_{\R^3}|\D(Q_{m+1}-Q_{m})|^2+|\D(Q_{m}-Q_{m-1})|^2\,dxdt
	\\
	&+C(M_1+M_1^4)\int_0^s \int_{\R^3}|v_m-v_{m-1}|^2+|Q_m-Q_{m-1}|^2\,dxdt
	.\alabel{h d f}
\end{align*}
Adding \eqref{h d f} to \eqref{v d2}, we have
\begin{align*}
	&\frac12\int_{\R^3}|v_{m+1}-v_m|^2(\cdot,s)\,dx+\frac14\int_0^s\int_{\R^3}|\D(v_{m+1}-v_{m})|^2\,dxdt
	\\
	\leq&2\eta_2 \int_0^s \int_{\R^3}|\D^2(Q_{m+1}-Q_{m})|^2\,dx dt
	\\
	&+C\eta_1\int_0^s\int_{\R^3}|\D^2(Q_{m}-Q_{m-1})|^2+|\D(v_{m}-v_{m-1})|^2\,dxdt
	\\
	&+C(M_1+M_1^4)\int_0^s \int_{\R^3}|\D(Q_{m+1}-Q_{m})|^2+|\D(Q_{m}-Q_{m-1})|^2\,dxdt
	\\
	&+C(M_1+M_1^4)\int_0^s \int_{\R^3}|v_m-v_{m-1}|^2+|Q_m-Q_{m-1}|^2\,dxdt
	.\alabel{v f}
\end{align*}
Substituting \eqref{Q d final} into \eqref{v f} and choosing sufficiently small $\eta_2$, we obtain
\begin{align*}
	&\sup_{0\leq s\leq T}\int_{\R^3}\(|Q_{m+1}-Q_m|^2+|\D(Q_{m+1}-Q_m)|^2+|v_{m+1}-v_m|^2\)(\cdot,s)\,dx
	\\
	&+\int_0^T\int_{\R^3}|\D^2 (Q_{m+1}-Q_{m})|^2+|\D(v_{m+1}-v_{m})|^2\,dxdt
	\\
	\leq&  C_3\eta_1\int_0^T\int_{\R^3}|\D^2(Q_{m+1}-Q_{m})|^2+ |\D^2(Q_{m}-Q_{m-1})|^2+|\D(v_{m}-v_{m-1})|^2\,dxdt
	\\
	&+C_3(M_1+M_1^2+M_1^4)s\sup_{0\leq s\leq T}\int_{\R^3}\(|\D(Q_{m+1}-Q_{m})|^2+|\D(Q_{m}-Q_{m-1})|^2\)(\cdot,s)\,dx
	\\
	&+C_3(M_1+M_1^4)s\sup_{0\leq s\leq T}  \int_{\R^3}\(|v_m-v_{m-1}|^2+|Q_m-Q_{m-1}|^2\)(\cdot,s)\,dx,
	\alabel{diff}
\end{align*}
where $C_3$ is a constant independent of $m$. Then for $m>1$, choosing $\eta_1=\frac18C_3^{-1}$, we prove the claim 2 with
\[T:=\min\{(8C_3M_1)^{-1},(8C_3M_1^{4})^{-1},T_{M_1}\}.\]
It follows from Claim 1 that $(Q_{m+1},v_{m+1})$  and $(Q_m,v_m)$ have two limits. By Claim 2, $(Q_{m+1},v_{m+1})$ is a Cauchy sequence in
$L^{\infty} ([0, T]; H^1_{Q_e}\times L^2) \cap L^2([0,T];  H^2_{Q_e}\times H^1])$, so
two weak limit of  $(Q_{m+1},v_{m+1})$  and $(Q_m,v_m)$ are the same. One can estimate $P_m$ using \eqref{v +1} and the argument in \eqref{eq
P}-\eqref{eq P2}. As $m\to \infty$, we prove Theorem \ref{thm loc}.
\end{proof}

\medskip\noindent
{\bf Acknowledgement:} { The third author Y. Mei is supported by the National Natural Science Foundation of China No. 12101496.}

\end{document}